\documentclass[a4paper]{amsart}

\usepackage{amsmath,amssymb}
\usepackage{eucal}
\usepackage{todonotes}

\usepackage{enumitem}
\usepackage{hyperref}
\usepackage{stmaryrd}
\usepackage{comment}
\usepackage[all]{xy} 
\usepackage{tikz-cd}
\usepackage{xparse} 
\usepackage{xspace} 
\usepackage{cancel} 
\usepackage{scrextend} 
\usepackage{mathtools} 
\usepackage{dsfont} 

\usepackage{fullpage}

\usepackage{caption}
\usepackage{subcaption}

\usepackage[ruled, lined, linesnumbered, commentsnumbered, longend]{algorithm2e}

\usepackage{graphicx, graphbox}
\graphicspath{{figures/}}

\numberwithin{equation}{section}

\DeclareMathAlphabet{\mathcal}{OMS}{cmsy}{m}{n}  

\renewcommand{\paragraph}[1]{\noindent\textbf{#1}}

\newtheorem{theorem}{Theorem}[section]
\newtheorem{lemma}[theorem]{Lemma}
\newtheorem{proposition}[theorem]{Proposition}
\newtheorem{corollary}[theorem]{Corollary}
\theoremstyle{definition}
\newtheorem{definition}[theorem]{Definition}

\newtheorem{example}[theorem]{Example}
\theoremstyle{remark}
\newtheorem{remark}[theorem]{Remark}

\def\N{{\mathbb N}}    
\def\Z{{\mathbb Z}}    
\def\R{{\mathbb R}}    
\def\id{\mathrm{id}}

\newcommand{\e}{\varepsilon} 
\newcommand{\libr}{\llbracket} 
\newcommand{\ribr}{\rrbracket} 
\newcommand{\integint}[1]{{\libr 1,#1\ribr}} 
\newcommand{\one}{\mathds{1}}

\DeclareMathOperator{\Rk}{Rk}
\DeclareMathOperator{\rank}{rank}
\DeclareMathOperator{\rep}{rep}
\DeclareMathOperator{\Rep}{Rep}

\DeclareMathOperator{\Hom}{Hom}
\DeclareMathOperator{\End}{End}

\newcommand{\RelationAsSet}[2][\leq]{\{ {#1}_{#2} \}}

\newcommand{\Erank}{\mathcal{E}_{\Rk}}
\newcommand{\K}{\operatorname{K_0}}

\newcommand{\card}[1]{\##1}

\def\pos{P}
\def\subpos{Q}
\newcommand{\seg}[1]{\langle#1\rangle}
\newcommand{\dhook}[1]{\langle#1\langle}
\newcommand{\uhook}[1]{\rangle#1\rangle}

\DeclareMathOperator{\convol}{\ast}

\def\field{\mathbf{k}}
\def\Vec{\mathrm{Vec}}
\def\vec{\mathrm{vec}}

\def\dgm{\textnormal{Bar}\,}

\def\Mod{M}
\def\Nod{N}

\newcommand{\lkan}{\mathrm{Lan}\,} 

\def\int{{I}}
\def\jnt{{J}}
\def\knt{{K}}
\def\lnt{{L}}
\def\rec{{R}}
\def\sec{{S}}

\newcommand{\mult}{\operatorname{mult}}

\newcommand{\Int}{\ensuremath{\mathcal{I}}\xspace}

\newcommand{\Rec}{\ensuremath{\mathcal{R}}\xspace}
\newcommand{\Sec}{\ensuremath{\mathcal{S}}\xspace}
\newcommand{\Tec}{\ensuremath{\mathcal{T}}\xspace}

\newcommand{\Intt}{\ensuremath{\Int_t}} 
\newcommand{\Intd}{\ensuremath{\Int_d}} 

\newcommand{\grid}{G}

\def\line{\ell}

\def\dist{\mathrm{d}}
\def\disti{\dist_{\mathrm{i}}}
\def\distb{\dist_{\mathrm{b}}}
\def\distm{\dist_{\mathrm{match}}}

\usepackage{adjustbox}

\DeclareMathOperator{\dimvect}{\underline{dim}}

\newcommand{\revisedversion}[1]{#1}
\newcommand{\newversion}[1]{#1}

\newcommand{\newnewversion}[1]{#1}
\title{Signed Barcodes for Multi-Parameter Persistence via Rank Decompositions and Rank-Exact Resolutions}
\author{Magnus Bakke Botnan}
\author{Steffen Oppermann}
\author{Steve Oudot}

\renewcommand{\qedhere}{}

\newtheorem{resultx}{Result}

\begin{document}

\begin{abstract}
  In this paper we introduce the signed barcode, a new visual representation of the global structure of the rank invariant of a multi-parameter persistence module or, more generally, of a poset representation. Like its unsigned counterpart in one-parameter persistence, the signed barcode decomposes the rank invariant as a $\Z$-linear combination of rank invariants of indicator modules supported on segments in the poset.
  We  develop the theory behind these decompositions, both for the usual rank invariant and for its generalizations, showing under what conditions they exist and are unique.
  We also show that, like its unsigned counterpart, the  signed barcode reflects in part the algebraic structure of the module: specifically, it derives from the terms in the minimal rank-exact resolution of the module, i.e., its minimal projective resolution  relative to the class of short exact sequences on which the rank invariant is additive. To complete the picture, we show some experimental results that illustrate the contribution of the signed barcode in the exploration of multi-parameter persistence modules. 
\end{abstract}

\maketitle

\section{Introduction}
\label{sec:intro}

  \subsection{Context}
  Given a {\em persistence module}~$\Mod$, i.e., a representation of a poset~$\pos$ in the \newversion{category of} vector spaces, the rank invariant $\Rk \Mod$~\cite{Carlsson2009} encodes the (possibly infinite) ranks of all the internal morphisms of~$\Mod$: 
  \begin{equation}\label{eq:rank_inv}
    \Rk \Mod (s,t)\ :=\ \rank \left[ \Mod(s)\to \Mod(t) \right] \quad \mbox{for every} \ (s,t)\in\RelationAsSet{\pos}.
  \end{equation}
Here, the poset~$\pos$ is fixed and viewed as a category with a single object per element $s\in \pos$ and a single morphism per couple of comparable elements $s\leq t\in\pos$ (we write~$\RelationAsSet{\pos}$ for the set of all these couples), and $\Mod$ is a functor from $\pos$ to the category $\Vec_{\field}$ of vector spaces  over an arbitrary but fixed field~$\field$. We write~$\Rep \pos$ for the category of such functors, and $\rep \pos$ for the full subcategory spanned by the objects~$\Mod$ that are {\em pointwise finite dimensional} (pfd), i.e., that satisfy $\dim \Mod(s) <\infty$ for all $s\in \pos$. \newversion{For $\int$ an interval (i.e., convex\footnote{Convexity means that $u\in \int$ whenever one has $s\leq u\leq t$ with $s,t\in \int$.} and connected\footnote{Connectedness means that, for any $s,t\in\int$, there are $u_0, \ldots, u_{n+1}\in \int$ such that $u_i$ and $u_{i+1}$ are comparable for all $0\leq i\leq n$ and $u_0=s$ and $u_{n+1} = t$. This condition is automatically satisfied when the poset~$\pos$ is totally ordered, for instance when $\pos\subseteq\R$.}) in $P$, we let 
$\field_\int\in \rep\pos$ denote the \emph{interval} persistence module given by
\newnewversion{\[ (\field_\int)(s) = \begin{cases} \field & \text{if $s\in \int$,} \\ 0 & \text{if $s\not\in \int$,}\end{cases}\]}
where $(\field_\int)(s)\to (\field_\int)(t)$ is the identity map for any $s\leq t\in I$.} 

In practical applications of TDA, the poset $\pos$ is usually taken to be $\R^d$, viewed as a product of $d$ copies of the totally ordered real line, or any subposet thereof (notable cases include the integer grid~$\Z^d$ and its  finite subgrids). \newversion{Importantly, for totally ordered sets and pfd representations, the rank invariant describes any module up to isomorphism and is therefore a \emph{complete} invariant. This is a direct consequence of the following fundamental result (Figure~\ref{fig:decomp_1d}).
\begin{theorem}[\cite{botnan2020decomposition,Crawley-Boevey2012}]
 Let $\pos$ be a totally ordered set. Then, any $\Mod\in \rep \pos$ decomposes as
  \begin{equation}\label{eq:decomp_mod_1d}
    \Mod \simeq \bigoplus_{\int\in\dgm\Mod} \field_\int,
  \end{equation}
  for a unique multiset of intervals $\dgm\Mod$ called the \emph{persistence barcode} of $\Mod$. 
 \end{theorem}
Remarkably, what at first glance appears to be a complicated collection of vector spaces and linear maps admits a compact description by \newnewversion{simple subsets of~$\R$}. This result has played a pivotal role in TDA and nearly all of the theory for persistent homology in a single-parameter ($d=1$) has been developed from the barcode.  It is also worth observing   } that it is easy to compute the barcode from the rank invariant when $\pos$ is \emph{locally finite} via the following inclusion-exclusion formula (which assumes for simplicity that $\pos=\Z$, and where $\mult_\int \dgm \Mod$ denotes the multiplicity of $\int$ in the multiset~$\dgm\Mod$):
  \begin{equation}\label{eq:incl_excl_1d}
    \begin{gathered}
      \begin{array}{rcl}
        \mult_{\llbracket s,t\rrbracket}\dgm\Mod &=& \Rk \Mod(s,t) - \Rk \Mod(s-1, t) \\[0.5em]
        && - \Rk \Mod(s, t+1) + \Rk \Mod(s-1, t+1).
        \end{array}
      \end{gathered}
    \end{equation}

  \begin{figure}[tb]
    \centering
    \includegraphics[width=0.6\textwidth]{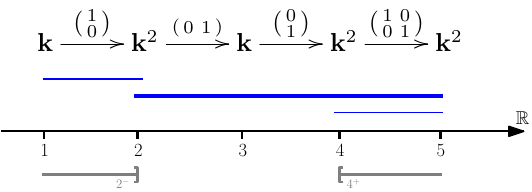}
    \caption{A one-parameter persistence module~$\Mod$ (top) indexed over $\{1, 2, 3, 4, 5\}\subset\R$, and the graphical representation of its barcode (in blue). The corresponding rank decomposition  $\Rk M = \Rk \field_{\llbracket1,2\rrbracket} + \Rk \field_{\llbracket2,5\rrbracket} + \Rk \field_{\llbracket4,5\rrbracket}$ is readily available, and the ranks can easily be read from it: for instance, the rank $\Rk\Mod(2,4)=1$ is given by the one bar (thickened) that connects the down-set $2^-$ to the up-set $4^+$.}
    \label{fig:decomp_1d}
  \end{figure} 

\medskip

The picture in the multi-parameter setting ($d>1$) is much less bright. \newversion{First and foremost, the representation theory of such posets is of ''wild type''. This  essentially means that classifying the space of indecomposable representations is a hopeless task. This poses a serious challenge for the development of an efficient theory of multi-parameter persistence, and for that reason most of the descriptors that have proven useful in applications have been defined using (a part of) the rank invariant; see \cite{Botnan2023} for a recent treatment on multi-parameter persistence. 

Since the rank invariant is \newnewversion{equivalent} to a set of intervals for $d=1$ (the barcode), a natural step forward is to develop a theory for $d>1$ based on (possibly unique) geometric decompositions of the rank invariant. If successful, then much of the theory developed for $d=1$ could potentially generalize in a natural way. However,} the rank invariant can no longer be decomposed as a sum of rank invariants of interval modules, as the following example shows.
\begin{figure}[tb]
  \centering
  \includegraphics[width=\textwidth]{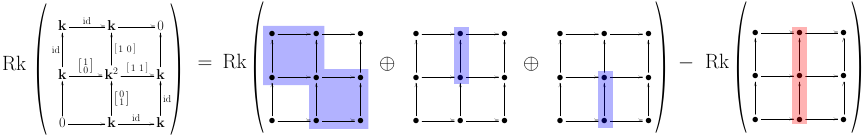}
  \caption{The indecomposable module~$\Mod$ on the left-hand side does not have the same rank invariant as any direct sum of interval modules on the  $3 \times 3$ grid. However, $\Rk\Mod$  is equal to the difference between the rank invariants of two direct sums of interval modules, as shown on the right-hand side. Blue is for intervals counted positively in the decomposition, while red is for intervals counted negatively. }
  \label{fig:decomp_indec2_int_grid}
\end{figure}
\begin{example} \label{ex:no_rank_sum}
Consider the module~$\Mod$ on the left-hand side of Figure~\ref{fig:decomp_indec2_int_grid}, which is a representation of the grid poset $\llbracket 1,3\rrbracket \times \llbracket 1,3\rrbracket\subset\R^2$, and assume there is a decomposition $\Rk\Mod = \sum_{\int\in \Int} \Rk \field_\int$. Since $\Mod((1,2)\to(1,3)) = \id = \Mod((1,3)\to(2,3))$,  the square $\llbracket 1,2\rrbracket \times \llbracket 2,3\rrbracket$ must be contained in one of the intervals, say~$\int$. Similarly, the square $\llbracket 2,3\rrbracket \times \llbracket 1,2\rrbracket$ must be contained in some interval~$\int'$ involved in the decomposition. If $\int$ and $\int'$ are the same interval, then the morphism $(2,1)\to (2,3)$ has non-zero rank in the decomposition, whereas it has zero rank in~$\Mod$, a contradiction. But if $\int$ and $\int'$ are not the same interval, then no other interval may appear in the decomposition because the dimension vector of~$\Mod$ is then saturated by $\int$ and $\int'$, so the  morphism $(1,2)\to (3,2)$ has zero rank in the decomposition, whereas it has non-zero rank in~$\Mod$, again a contradiction. 
\end{example}
Even in cases where $\Rk\Mod$ does decompose as a sum of rank invariants of interval modules, this decomposition may tell little to nothing about the direct-sum decomposition of~$\Mod$. \newversion{That is, the rank invariant is not even complete on the subcategory of {\em interval-decomposable modules}, which by definition only have interval modules as indecomposable direct summands. The following example illustrates this well-known fact.}
\begin{example}\label{ex:decomp_tells_nothing}
  Consider again the module~$\Mod$ on the left-hand side of  Figure~\ref{fig:decomp_indec2_int_grid}, and take its direct sum with the module~$\Nod$ highlighted in red on the right-hand side. Then, as shown in the figure, $\Rk (\Mod\oplus\Nod)$ does admit a decomposition as a sum of rank invariants of interval modules, however these intervals differ significantly from the indecomposable direct summands~$\Mod$ and~$\Nod$: in particular, the first interval in the decomposition creates the illusion that it should be spanned by some feature in~$\Mod\oplus\Nod$, whereas in fact no feature does so, as explained in Example~\ref{ex:no_rank_sum}.
\end{example}

\newversion{To overcome the issues raised in Example~\ref{ex:no_rank_sum} and Example~\ref{ex:decomp_tells_nothing}, we turn to \emph{signed} decompositions of the \emph{generalized} rank invariant.}

\subsection{Signed rank decompositions}
\label{sec:intro_rank-decomp}
\newversion{The {\em generalized rank invariant} is a construction introduced by Kim and M\'emoli~\cite{kim2018generalized} which}  probes the existence of `features' in the module across arbitrary intervals~$\int\subseteq\pos$:
\begin{definition}\label{def:generalized-rank_invariant}
  Let $\Mod\in\Rep\pos$.
  Given an interval $\int\subseteq\pos$, the {\em generalized rank} of $\Mod$ over $\int$, denoted by $\Rk_\int \Mod$, is defined by:
\[ \Rk_\int \Mod = \rank  \left[ \varprojlim \Mod|_\int \to \varinjlim \Mod|_\int \right], \]
where $\varprojlim \Mod|_\int \to \varinjlim
\Mod|_\int$ is the natural morphism from the limit to the co-limit of the diagram~$\Mod|_\int$ and is well-defined because $\int$~is both convex and connected.
Given a collection $\Int$ of intervals, the {\em generalized rank invariant} of $\Mod$ over $\Int$ is the map $\Rk_\Int \Mod \colon \Int \to \N \cup \{ \infty \}$ defined by $\Rk_\Int \Mod(\int) =  \Rk_\int \Mod$.
\end{definition}

\newversion{In the following, we  write~\( \rep_{\Int} \pos \) for  the full subcategory of~$\Rep \pos$ spanned by those objects $M$ such that $\Rk_\Int \Mod$ only takes finite values. Note that $\rep_{\Int}\pos$ contains $\rep \pos$ since the map $\varprojlim \Mod|_\int \to \varinjlim
\Mod|_\int$ factors through the internal spaces of~$\Mod|_\Int$.} 

To see that \newversion{Definition~\ref{def:generalized-rank_invariant}} generalizes the usual rank invariant, take $\int$ to be a {\em (closed) segment} $\seg{s,t} = \{u\in\pos \mid s\leq u\leq t\}$; in this case one has $\Rk_\int \Mod = \rank \left[ \Mod(s)\to\Mod(t) \right]$. 
This generalization was proven by Kim and M\'emoli~\cite{kim2018generalized} to be stronger than the usual rank invariant, and in fact complete on interval-decomposable modules, under some finiteness conditions on the poset~$\pos$---implying in particular that $\Rk_\Int \Mod$ takes only finite values for any pfd module~$\Mod$. The way they arrive at their conclusions is by viewing $\Rk_\Int \Mod$ as a map $\Int\to\Z$, and by taking its M\"obius inverse in the incidence algebra of~$\Int$---itself viewed as a poset equipped with the inclusion order. This inverse, which they call {\em generalized persistence diagram} of~$\Mod$ after Patel~\cite{patel2018generalized} because its expression generalizes the inclusion-exclusion formula~\eqref{eq:incl_excl_1d}, actually encodes the multiplicities of the intervals in the direct-sum decomposition of~$\Mod$ when $\Mod$ is interval-decomposable. When $\Mod$ is not interval-decomposable, M\"obius inversion can still be applied to~$\Rk_\Int \Mod$. \newversion{The resulting multiplicities of intervals can be negative, but the output still encodes the generalized rank invariant of $\Mod$.}

  \newnewversion{We now} put the generalized persistence diagrams in the larger context of {\em signed rank decompositions}.
\begin{definition}\label{def:rank_decomp}
Given a collection $\Int$ of intervals in $\pos$, and a function $r\colon \Int \to \mathbb{Z}$, a {\em (signed) rank decomposition of $r$ over $\Int$} is given by the following identity:
\[
r= \Rk_\Int \field_\Rec - \Rk_\Int \field_\Sec,
\]
where $\Rec$ and $\Sec$ are multi-sets of elements taken from~$\Int$ such that \( \field_{\Rec} \) and \( \field_{\Sec} \) lie in \( \rep_{\Int} \pos \), and where  by definition
\( \field_\Rec = \bigoplus_{\rec\in\Rec} \field_\rec \) and \( \field_\Sec = \bigoplus_{\sec\in\Sec} \field_\sec \)
(note that elements $\rec\in\Rec$ and $\sec\in\Sec$ are considered with multiplicity in the direct sums). By extension, we call the pair $(\Rec, \Sec)$ itself a rank decomposition of~$r$ over~$\Int$. It is a {\em minimal rank decomposition} if $\Rec$ and $\Sec$ are disjoint as multi-sets. 
\end{definition}

Note that rank decompositions in general are not unique, since extra
elements can be added to both $\Rec$ and $\Sec$ with no effect
on~$\Rk_\Int \field_\Rec - \Rk_\Int \field_\Sec$. 
\newversion{\begin{resultx}[Corollary~\ref{cor:min-decomp_exists_unique}]
  The minimal rank decomposition $(\Rec^*, \Sec^*)$ of any map $r:\Int\to\Z$ is unique if it exists. \newnewversion{Furthermore, if a rank decomposition $(\Rec, \Sec)$ of~$r$ exists, then a minimal one exists and is obtained from it by removing common intervals, that is:}
  \[ (\Rec^*, \Sec^*) = (\Rec\setminus\Rec\cap\Sec, \Sec\setminus\Rec\cap \Sec). \]
  \end{resultx}
This result is a consequence of the more fundamental fact that the minimal rank decomposition
is universal among all the rank decompositions of~$r$
(Theorem~\ref{thm:uniqueness}), which itself follows from the fact that the generalized rank invariant, as long as it is finite (i.e., it only takes finite values), is complete on interval-decomposable modules (Proposition~\ref{prop:complete}).} 
For this to hold, Definition~\ref{def:rank_decomp} requires that the intervals involved in the rank decompositions of~$r\colon\Int\to\Z$  be taken from the collection~$\Int$ itself, not from outside: otherwise, as Figures~\ref{fig:decomp_indec2_int_grid} and~\ref{fig:decomp_indec2_grid} together illustrate, the minimal rank decomposition may not be unique.

We also show that, under some finiteness conditions on~$\pos$ and~$r$ that are similar to the ones in~\cite{kim2018generalized}, rank decompositions of~$r$ exist. While this second result can be obtained via an adaptation of the approach from previous work using M\"obius inversion (see Appendix~\ref{sec:localfinite-mobius}), we give an alternative, more  direct proof that is much shorter given our previous results. 
\newversion{\begin{resultx}[Theorem~\ref{th:basis}]
Let \( \Int \) be a locally finite collection of intervals in \( \pos \). Then any function \( r \colon \Int \to \Z \) with upward finite support can uniquely be written as a (possibly infinite, but pointwise finite) $\Z$-linear combination of the functions \( \Rk_\Int \field_{\int} \) with \( \int \in \Int \).
\label{resultB}
\end{resultx}}
These results, detailed in Section~\ref{sec:decomp_exist_unique}, significantly extend previous work on generalized persistence diagrams, providing general conditions under which rank decompositions of $\Z$-valued functions can be considered, with guarantees of existence and uniqueness.

Decompositions of generalized rank invariants of persistence modules in $\rep_{\Int} \pos$
is a special case of this general theory.
In particular, when $\Int$ is the collection of (closed) segments in~$\pos$, $\Rk_\Int$ is just the usual rank invariant~$\Rk$ and our results tell under which conditions it can be decomposed uniquely over the family of rank invariants of {\em segment modules}---i.e., interval modules supported on segments.

We complete our study of rank decompositions of persistence modules with an analysis of their behavior under poset maps, showing that, in essence, rank decompositions commute with restrictions along morphisms of posets (Corollary~\ref{corollary restriction commutes rank dec}). \newversion{This result is important as it allows us to bring our theory to multi-parameter persistence; note that $\R^d$ does not meet the conditions of \newnewversion{Result~\ref{resultB}}. Observing that a left Kan extension of a finite grid in $\R^d$ is a particular type of restriction, we obtain the following result.

\begin{resultx}[Corollary~\ref{cor:decomp_half-open2}]
Let \( M \in \rep \mathbb{R}^d \) be finitely presented. Let \( \Int \) either be the collection of half-open rectangles or the collection of all half-open intervals. Then there are unique disjoint multi-sets \( \Rec \) and \( \Sec \) of elements of \( \Int \), such that
\[ \Rk_{\Int} M = \Rk_{\Int} \field_{\Rec} - \Rk_{\Int} \field_{\Sec}. \]
\end{resultx}

Furthermore, Corollary~\ref{corollary restriction commutes rank dec} allows us to  prove a form of interleaving stability of the rank decompositions under perturbations. 

\begin{resultx}[Corollary~\ref{cor:matching-distance_stability}]
  Let $\Mod, \Mod'$ be pfd persistence modules indexed over $\R^d$. Then, for any rank decompositions $(\Rec, \Sec)$ and $(\Rec', \Sec')$ of $\Rk \Mod$ and $\Rk \Mod'$ respectively over all rectangles in~$\R^d$, we have:
  \[
  \distm(\field_{\Rec} \oplus \field_{\Sec'}, \field_{\Rec'} \oplus \field_{\Sec}) \leq \disti(\Mod, \Mod').
  \]
\end{resultx}}

\begin{figure}[tb]
  \centering
  \includegraphics[width=\textwidth]{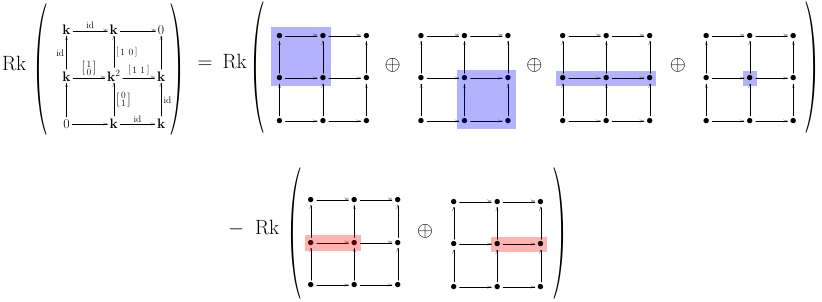}
  \caption{Minimal rank decomposition of the rank invariant of the  module~$\Mod$ from Figure~\ref{fig:decomp_indec2_int_grid} over the collection of segments (rectangles) in the $3\times 3$ grid. Blue is for rectangles in~$\Rec$, while red is for rectangles in~$\Sec$. This decomposition is unique as long as the intervals are constrained to be segments, otherwise Figure~\ref{fig:decomp_indec2_int_grid} would give another valid minimal decomposition.
  }
  \label{fig:decomp_indec2_grid}
\end{figure}

\subsection{Signed barcodes and prominence diagrams}

\begin{figure}[tb]
  \centering
  \includegraphics[scale=1]{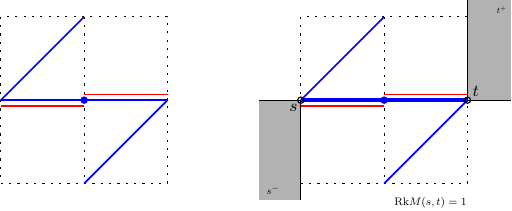}
  \caption{Left: the signed barcode corresponding to the rank decomposition of Figure~\ref{fig:decomp_indec2_grid}. Each bar is the diagonal with positive slope of one of the segments (rectangles) involved in the decomposition, with the same color code (blue for positive sign, red for negative sign). Right: computing $\Rk \Mod(s,t)$ for a pair of indices $s\leq t$, by counting with signed multiplicity the bars that connect the down-set $s^-$ to the up-set $t^+$ (here only the thickened bar does so).}
  \label{fig:barcode_indec2_grid}
\end{figure}

Our second contribution, detailed in Section~\ref{sec:signed_barcodes}, is to introduce an effective graphical representation of minimal rank decompositions as {\em signed barcodes}. \newnewversion{For this} we focus on the usual rank invariant and we consider persistence modules that are indexed over~$\R^d$ or some subposet thereof as in the standard multi-parameter persistence setting. Given the minimal decomposition $(\Rec, \Sec)$ of $\Rk \Mod$, each segment $\seg{s,t}\in\Rec$ is turned into the {\em bar} (line segment) $[s,t]\subset\R^d$ with sign~$\text{+}1$, while each segment $\seg{s',t'}\in\Sec$ is turned into the bar $[s',t']\subset\R^d$ with sign~$\text{-}1$. See Figure~\ref{fig:barcode_indec2_grid}~(left) for an example, and Figure~\ref{fig:barcode_indec2_grid}~(right) for an illustration of how the  rank invariant can be read off from the signed barcode similarly to what is done in the  1-parameter case.

In order to ease the task of discriminating between the signal and the noise in a signed barcode, we also propose a lighter representation as a {\em signed prominence diagram}, in which each signed bar $[s,t]$ is replaced by the signed vector~$t-s\in\R^d$. Guided by our stability results for rank decompositions, we explain how to navigate within this diagram, and in particular, where to expect the signal and the noise to appear (Lemma~\ref{lem:e-smoothed_decomp}).

In Section~\ref{sec:experiments} we illustrate the practical usage of these concepts on simulated data.

\subsection{Rank-exact resolutions}
\label{sec:intro_relative_resolutions}

While rank decompositions tell us everything about the (generalized) rank invariant $\Rk_\Int \Mod$ as a function~$\Int\to\Z$, it is unclear what they tell us about the module~$\Mod$ itself. A notable exception is when $\Mod$ is interval-decomposable, since in that case, the terms in the decomposition coincide with the indecomposable direct summands of~$\Mod$ as first shown by Kim and Memoli~\cite{kim2018generalized}. \newversion{A natural question, therefore, is whether the rank decomposition reflects a deeper algebraic fact}. In Section~\ref{sec:rank-exact}, we answer this question in the affirmative for the usual rank invariant by equipping $\rep \pos$ with an exact structure and studying the associated projective resolutions and Grothendieck group, both of which are classical constructions in homological algebra. 

\newnewversion{
To see the need for such a generalized framework}, recall that if $M$ is finitely presented, then a decomposition of its dimension vector (a.k.a. Hilbert function)~$\dimvect M$ arises as the terms in a finite projective resolution of $M$. Indeed, any such resolution~$\Mod_\bullet\twoheadrightarrow \Mod$ would yield a decomposition  $\dimvect M = \sum_{i\in\N} (-1)^i\, \dimvect M_i$ because it is known that the dimension vector $\dimvect$ is additive on all short exact sequences (by the rank-nullity theorem).  The same would hold for any invariant that is additive on all short exact sequences. Unfortunately, the rank invariant $\Rk M$ is not of this kind, as illustrated in the following example  (where the representations are indexed over the poset $\{1,2\}\subset\R$ and where the morphisms between them are the obvious ones):

\[\xymatrix{
  0 \ar[r] & \left(0\to\field\right) \ar[r] &  \left(\field\stackrel{\id}{\to}\field\right) \ar[r]
  &  \left(\field\to 0\right) \ar[r] & 0
}\]

Thus, it is in general not the case that $\Rk M = \sum_{i\in\N} (-1)^i\, \Rk M_i$ given a finite projective resolution \( M_{\bullet} \twoheadrightarrow M \) of~$M$. In order to decompose the rank invariant using projective resolutions, we are thus forced to only consider short exact sequences on which the rank invariant is additive---we call them short {\em rank-exact} sequences.  
By doing so, we rely on a smaller class of short exact sequences, hence on  a larger corresponding class of projectives\footnote{This is because, by definition, the projectives are those representations~$M$ whose corresponding covariant $\Hom$-functor~$\Hom(M,-)$ is exact on every short exact sequence. Thus, the smaller the class of short exact sequences, the larger the corresponding class of projectives.}, and the following facts turn out to be true:

\begin{itemize}
\item (Thm.~\ref{thm is exact cat}) The category $\rep P$, equipped only with the short rank-exact sequences, admits the structure of an exact category---called the {\em rank-exact category}---in which we can apply standard tools of homological algebra.
\item (Thm.~\ref{thm:enough proj}) When $\pos$ is finite, the corresponding indecomposable projectives (resp. injectives) are the interval representations supported on \emph{lower} (resp. \emph{upper}) \emph{hooks}, and they include the usual indecomposable projectives:
\begin{align*}
\dhook{s, t} & = \{ u \in \pos \mid s \leq u \not\geq t \} && \text{for } s < t \in \pos \cup \{ \infty \} & \mbox{(\em lower hook)},\\
\uhook{s, t} & = \{ \ell \in \pos \mid s \not\geq u \leq t \} && \text{for } s < t \in \pos \cup \{ - \infty \} & \mbox{(\em upper hook)}.
\end{align*}

Moreover, every object in the category admits finite projective and injective resolutions---called {\em rank-exact resolutions}. 
\item \newversion{(Prop.~\ref{prop:finitely-pres_finite-resols}) When \( \pos \) is an upper semi-lattice, then finitely presented representations, together with rank-exact sequences, form an exact category with enough projectives and injectives. The projectives and injectives in this exact category are precisely upper and lower hooks, respectively. Moreover, any object has a finite projective and a finite injective resolution.}
\end{itemize}
\begin{figure}[tb]
  \centering
  \includegraphics[width=\textwidth]{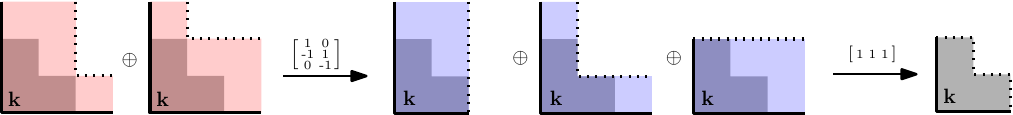}
  \caption{\revisedversion{Minimal rank-exact resolution of a finitely presented interval representation~$M$ of~$\R^2$ (in gray). For clarity, the support of the interval is superimposed with the support each term in the resolution. By construction, the alternating sum of the rank invariants of these terms is equal to~$\Rk(M)$, which can be readily seen on the picture as the ranks sum up to~$1$ within the superimposed areas while the dimensions sum up to~$0$ elsewhere.}}
\label{fig:rank-exact_resol}
\end{figure}
\newversion{For a finitely presented $M\in \rep P$, where $P$ is finite or an upper semi-lattice, we can therefore consider a finite rank-exact resolution \( M_{\bullet} \twoheadrightarrow M \).}  As the rank invariant is additive on all rank short exact sequences, $\Rk M = \sum_{i\in\N} (-1)^i\, \Rk M_i$ as desired.  This proves that the interval representations supported on lower hooks---called {\em lower hook modules} for simplicity---generate all the rank invariants of finitely presented representations. So do  the  {\em upper hook modules}, i.e., the interval representations supported on upper hooks, via injective resolutions. See Figure~\ref{fig:rank-exact_resol} for an example. This new type of rank decomposition defined by the terms in some rank-exact resolution is called a {\em rank-exact decomposition}.

Note that since the rank invariant of each hook is a linear combination of rank invariants of segments, the minimal rank decompositions obtained in Section~\ref{sec:decomp_exist_unique} do, in a certain sense, come from projective resolutions.
This is formalized in the following important theorem,
where $\K(\Erank)$ denotes the Grothendieck group of the exact structure.

\begin{resultx}[Theorem~\ref{thm:basis of K for finite}]
Let \( \pos \) be a finite poset. Then
\[ \Rk \colon \K(\Erank) \to \mathbb{Z}^{\RelationAsSet{\pos}} \]
is an isomorphism.
Moreover, the following three sets are bases for \( \K(\Erank) \):
\begin{align*}
  &\{ [\field_{\seg{i,j}}] \mid i \leq j \in P \},\\[0.5em]
  &\{ [\field_{\dhook{i,j}} ] \mid i < j \in P \cup \{ \infty \} \}, \ \text{and} \\[0.5em]
  &\{ [\field_{\uhook{i,j}} \mid i < j \in P \cup \{ - \infty \} \}.
\end{align*}
\end{resultx}

\subsubsection*{A remark on the two types of minimal decompositions}

By universality of minimal projective/injective resolutions, the rank-exact decomposition given by the terms in the minimal rank-exact resolution satisfies a universality property among all rank-exact decompositions, akin to the universality property of the minimal decomposition as per Definition~\ref{def:rank_decomp}. The two notions of minimal decomposition are different, though, because some terms in even degrees may coincide with terms in odd degrees in the minimal rank-exact resolution and therefore the intersection of the positive and negative parts may be non-empty.

The question arises whether one decomposition is better or more canonical than the other.  One immediate observation is that information is lost when common entries are removed from the positive and negative parts of a rank-exact decomposition. Consequently, while it is straightforward to derive a minimal rank decomposition (using hooks or segments) from a rank-exact decomposition, going in the opposite direction is not possible. Furthermore, in work subsequent to this paper, it has been shown that minimal rank-exact decompositions in $\R^d$ enjoy a  bottleneck stability theorem~\cite[Theorem 6.1]{botnan2022bottleneck}, and that no analogous result can exist for minimal rank decompositions~\cite[Proposition 3.1]{botnan2022bottleneck}. On the other hand, minimal rank decompositions can be effectively computed from the rank invariant using an inclusion-exclusion formula (Eq.~\ref{eq:incl_excl_rect}), while the computational complexity of obtaining a minimal rank-exact decomposition is unknown. Furthermore, rank-exact decompositions are limited to the standard rank invariant; see Section~\ref{sec:genrank-res}.

  \subsection{Related work}
  Our contributions relate to and extend the line of work on generalized persistence diagrams based on M\"obius inversion~\cite{asashiba2019approximation,betthauser2022graded,kim2018generalized,mccleary2021edit,patel2018generalized}.
  M\"obius inversion is but one of many possible invertible operators (including the identity operator) that can be applied to the rank invariant or its generalized version, viewed as functions $\Int\to\Z$, to decompose them over some basis. Prior to this work,  choosing M\"obius inversion over any other invertible operator was mainly motivated by an analogy with the 1-parameter setting, without further mathematical justification akin to the decomposition theorem for 1-parameter persistence modules. Here we provide such a justification, by connecting usual rank decompositions to projective resolutions in the rank-exact structure, the latter being induced precisely by those exact sequences that preserve the rank invariant.
  
Since its first appearance as a preprint, our work has sparked a novel line of research on rank decompositions, signed barcodes, and their connection to relative homological algebra. A new class of  invariants for multiparameter persistence modules has appeared, called the {\em homological invariants}, which include our rank-exact decompositions as a special case.  These invariants are derived from resolutions of the modules relative to some fixed classes of modules. Several such classes of {\em relative projectives} have been studied, notably the interval-decomposable modules and some of their subclasses (including the hook-decomposable modules). Upper bounds on  the corresponding  relative global dimensions have been obtained~\cite{aoki2023summand,asashiba2023relative,asashiba2023approximation,blanchette,botnan2022bottleneck,chacholski2023koszul}. Such bounds play a key part in the stability theory for homological invariants that has started being developed~\cite{botnan2022bottleneck,oudot2021stability}. Lately, computational aspects have also started being investigated~\cite{chacholski2023koszul}. We are only at the beginning of these developments, which we expect to expand in the future due to their potential impact on multi-parameter persistence theory.

\subsection{Outline of the paper}
In Section~\ref{sec:decomp_exist_unique}, we focus on the existence and uniqueness of minimal rank decompositions in varying levels of generality. In Appendix~\ref{sec:localfinite-mobius}, we explain how Möbius inversions can be used to compute minimal rank decompositions, which is then used in Section~\ref{sec:decomp_exist_unique} to derive a formula for computing the multiplicity of an interval in the minimal rank decomposition. \newversion{Additionally, in Section~\ref{sec:ri-poset}, we explore the behavior of rank decompositions under poset homomorphisms, with a particular focus on maps between lattices.}

In Section~\ref{sec:rank-exact}, we give a brief introduction to exact categories and study rank decompositions from the point of view of homological algebra. Here we work exclusively with the usual rank invariant. Importantly, minimal projective and injective resolutions in this category define a new type of rank decompositions called rank-exact decompositions. 

In Section~\ref{sec:rank-decomp_pers}, we reformulate and expand upon our earlier results in the specific context of multi-parameter persistence. We thus obtain unique minimal rank decompositions and rank-exact decompositions for finitely presented persistence modules over~$\R^d$, and of pfd persistence modules over finite grids. In the latter case, we also derive an explicit inclusion-exclusion formula to compute the coefficients in the minimal rank decompositions. The section concludes with a short discussion of stability. 

In Section~\ref{sec:signed_barcodes}, we introduce the signed barcode as a visual representation of the minimal rank decomposition of the usual rank invariant. We explain how the signed barcode reflects the global
structure of the usual rank invariant and how its role in
multi-parameter persistence is similar to the one played by the
unsigned barcode in one-parameter persistence. 

In Section~\ref{sec:experiments}, we carry out a round of experiments whose outcomes illustrate some of the key properties of the minimal rank decompositions and of their associated signed barcodes. 

Finally, in Section~\ref{sec:conclusion}, we conclude the paper with a detailed discussion of some aspects of our work.

\subsection*{Acknowledgements}

We thank the anonymous referees for their thorough reading of the paper and for their valuable feedback.

\section{Rank Decompositions: Existence and  Uniqueness}
\label{sec:decomp_exist_unique}

This section discusses the existence and uniqueness of minimal rank decompositions.
First, we show in Section~\ref{sec:rk-decomp_general} that a minimal rank decomposition is unique, provided it exists. Then, in Section~\ref{sec:rk-decomp_downward-finite} we show, using a short and elementary argument, that every map~$r \colon \Int\to\Z$ with upward finite support admits a unique minimal rank decomposition, provided the underlying collection of intervals~$\Int$ itself is locally finite (Definition~\ref{def:locallyfinite}). At the end of Section~\ref{sec:rk-decomp_downward-finite} we use Möbius inversions to construct an explicit formula for the multiplicity of an interval in the decompositions under mild constraints. This formula ultimately allows us to derive an algorithm for computing minimal rank decompositions in multiparameter persistence; see~\eqref{eq:incl_excl_rect}. For completeness, we include the full details of Möbius inversions in Appendix~\ref{sec:localfinite-mobius}. We remark that while our key result in Section~\ref{sec:rk-decomp_downward-finite} is a consequence \newversion{of} our work in Section~\ref{sec:rk-decomp_general} and the theory of Möbius inversions, the presented proof is both self-contained and simpler than those two results combined. We therefore think our novel approach is of independent interest.

\medskip
The following proposition, which generalizes \cite[Proposition 3.17]{kim2018generalized} by dropping the assumption of local finiteness of the poset~$\pos$ and allowing for generalized ranks, and which is given a more direct proof, will be instrumental throughout our analysis.

\begin{proposition}\label{prop:rank_count} 
Let \( \Rec \) be a multi-set of intervals of $\pos$. Then, for any interval $\int\subseteq\pos$, we have:
\[ \Rk_\int( \field_{\Rec} ) = \card{\{ \rec \in \Rec \mid \int \subseteq \rec \}}. \]
\end{proposition}

\begin{proof}
Note that
\[ \field_{\Rec} |_{\int} = \bigoplus_{\rec \in \Rec} \field_{\rec \cap \int} = \bigoplus_{\substack{\rec \in \Rec \\ \int \subseteq \rec}} \field_{\int} \oplus \bigoplus_{\rec \in \widetilde{\Rec}} \field_{\rec}, \]
where \( \widetilde{\Rec} \) is a collection of proper subintervals of \( \int \). (Note that while the \( \rec \cap \int \) are not necessarily connected themselves, they are disjoint unions of intervals.)

Since the rank commutes with finite sums, and clearly \( \Rk_{\int} (\bigoplus_{\substack{\rec \in \Rec \\ \int \subseteq \rec}} \field_{\int}) = \card{\{ \rec \in \Rec \mid \int \subseteq \rec \}} \), it suffices to show that \( \Rk_{\int}(\bigoplus_{\rec \in \widetilde{\Rec}} \field_{\rec}) = 0 \) for any multi-set \( \widetilde{\Rec} \) of proper subintervals of \( \int \).

Our next step is to note that for any such proper subinterval \( \rec \) of \( \int \) at least one of \( \varprojlim \field_{\rec}|_{\int} \) or \( \varinjlim \field_{\rec}|_{\int} \) is zero: indeed, we may observe that \( \varprojlim \field_{\rec}|_{\int} \) is non-zero precisely if \( \rec \) is closed under predecessors in~\( \int \). Dually, \( \varinjlim \field_{\rec}|_{\int} \neq 0 \) precisely if \( \rec \) is closed under successors in \( \int \). Since a proper subinterval cannot be closed both under predecessors and successors, it follows that at least one of limit or colimit is zero.

Going back to our multi-set \( \widetilde{\Rec} \) of proper subintervals of \( \int \), we can now decompose it as \( \widetilde{\Rec} = \Rec' \cup \Rec'' \) where \( \Rec' \) only contains subintervals \( \rec \subsetneq \int \) with \( \varprojlim_\int \field_{\rec} = 0 \), and \( \Rec'' \) only contains subintervals \( \rec \subsetneq \int \) with \( \varinjlim_{\int} \field_{\rec} = 0 \). (This decomposition will typically not be unique, as there may be subintervals for which both limit and colimit vanish.) Again invoking the fact that rank commutes with finite sums, it suffices to show that \( \Rk_{\int}(\field_{\Rec'}) = 0 \) and  \( \Rk_{\int}(\field_{\Rec''}) = 0 \).

The latter is immediate, since direct sums commute with colimits, and hence \( \varinjlim_{\int} \field_{\Rec''} = 0 \). For the former, we additionally use that the direct sum is naturally a subrepresentation of the direct product, and moreover that limits are left exact. This gives us
\[ \varprojlim_{\int} \field_{\Rec'} \subseteq \varprojlim_{\int} \prod_{\rec \in \Rec'} \field_{\rec} = \prod_{\rec \in \Rec'} \varprojlim_{\int}  \field_{\rec} = 0. \qedhere \]
\end{proof}

\begin{remark}
One might hope that an analogous result holds if one replaced the sum in the definition of \( \field_{\Rec} \) by a product. However, the following example shows this to not be the case:

Let \( \int = \{ (x, y) \in \mathbb R^2 \mid x \geq 0, y \geq 0, x+y \leq 1 \} \), and pick an infinite set \( \{ p_s \mid s \in S \} \) of pairwise different maximal points in \( \int \). Let \( \rec_s = \int \setminus \{ p_s \} \). Then
\[ \Rk_{\int} \left(\prod_{s \in S} \field_{\rec_s}\right) = \infty. \]
(The key point here is that \( \varinjlim_{\int}  \prod_{s \in S} \field_{\rec_s} = \revisedversion{\prod_{s \in S} \field_{\rec_s} / \bigoplus_{s \in S} \field_{\rec_s}}\neq 0 \).)
\end{remark}

\begin{corollary} \label{cor:description of rep_I}
Let \( \Int \) be a collection of intervals in \( \pos \). For a multi-set \( \Rec \) of intervals, we have that \( \field_{\Rec} \in \rep_{\Int} \pos \) if and only if
\[ \forall \int \in \Int, \quad \card{ \{ \rec \in \Rec \mid \int \subseteq \rec \} } < \infty. \]
\end{corollary}

\subsection{Uniqueness}
\label{sec:rk-decomp_general}
Let $\Int$ be a collection of intervals in a poset~$\pos$. While a minimal rank decomposition need not exist in this case, we show that it is unique, provided it exists. 

Our first result shows that $\Rk_\Int$ is a complete invariant when restricted to interval-decomposable representations supported on intervals in $\Int$. In fact, we show that the rank invariant is complete on a slightly larger collection of interval-decomposable modules: \revisedversion{let  $\widehat{\Int}\supseteq\Int$ be the collection of intervals (which by construction are also intervals, i.e., non-empty connected convex subsets of~$\pos$) given by 
\begin{equation}
\label{eq:hatInt} \widehat{\Int} = \left\{ \bigcup X \mid X \subseteq \Int \text{ directed}\right\}, 
\end{equation}
where directed means that there for all $\int_1, \int_2 \in X$ exists an 
 $\int_3 \in X$ such that $\int_1 \cup \int_2 \subseteq \int_3$.

The following proposition generalizes~\cite[Theorem~3.14]{kim2018generalized}. 
\begin{proposition}\label{prop:complete}
Let $\Int$ denote a collection of intervals in $\pos$. If \( \Rec \) and \( \Rec' \) are two multi-sets of elements in \( \widehat{\Int} \), such that \( \Rk_\Int \field_{\Rec} =  \Rk_\Int \field_{\Rec'} \) and this common rank invariant is finite \revisedversion{for all intervals of $\Int$}, then \( \Rec = \Rec' \).
\end{proposition}

\begin{proof}
Since the rank of a direct sum is the sum of the ranks we may remove the common elements from \( \Rec \) and \( \Rec' \), and thus assume that the two multi-sets are disjoint.
\revisedversion{If $\Rec \cup \Rec'$ is empty, then we are done. Otherwise, by Proposition~\ref{prop:rank_count} and the definition of \( \widehat{\Int} \), there is \( \int \in \Int \) such that \( \Rk_{\int} \field_{\Rec}=\Rk_{\int} \field_{\Rec'}>0 \). Since this rank \newversion{is finite}, it follows, again using Proposition~\ref{prop:rank_count}, that there is a non-zero finite number of intervals containing \( \int \) in \( \Rec \cup \Rec' \). In particular there is a (not necessarily unique) maximal one, which we will call \( \jnt \).}
Without loss of generality we assume \( \jnt \in \Rec \). By definition of \( \widehat{\Int} \) we have \revisedversion{\( \jnt = \bigcup X\) for some directed \( X \subseteq \Int \).} Now, by assumption, for every \revisedversion{$\int \in X$  we have 
\[ \Rk_{\int} \field_{\Rec'} = \Rk_{\int} \field_{\Rec} \geq \Rk_{\int} \field_{\jnt}, \]
which equals $1$ by  Proposition~\ref{prop:rank_count}.
It also follows from Proposition~\ref{prop:rank_count} that, for each $\int \in X$,  there is some interval \( \int' \in \Rec' \) such that \( \int \subseteq \int' \). Fix some \( \int_0 \in X \) \newversion{and for each $I\in X$  choose an} \( \int' \) containing \( \int_0 \cup \int \) -- employing the directness condition.

Since \( \Rk_{\int_0} \field_{\Rec'} \) is finite, Corollary~\ref{cor:description of rep_I} says that there are actually only finitely many choices for \( \int' \). It follows, using the fact that \( X \) is directed, that there is an \( \int' \in \Rec' \) containing all \( \int \in X \). Thus \( \jnt \subseteq \int' \).} If this is a proper inclusion then it contradicts the maximality of \( \jnt \), otherwise it contradicts the disjointness of~\( \Rec \) and~\( \Rec' \).
\end{proof}}

\revisedversion{
\begin{example}
Let \( \pos \) be any poset, \newversion{and denote by \( \Int \) the intervals that arise as convex hulls of finitely many elements of \( \pos \). Then, \( \widehat{\Int} \) is the collection of \emph{all} intervals. Thus, by Proposition~\ref{prop:complete}, the ranks of finitely spanned intervals are a complete invariant on pfd interval-decomposable modules (supported on arbitrary intervals)}.
\label{ex:closedtoall}
\end{example}
}
We can now show that minimal rank decompositions, whenever they exist, satisfy a universality property.

\begin{theorem}\label{thm:uniqueness}
Let $\Rec, \Sec, \Rec^*, \Sec^*$ be multi-sets of elements of~$\widehat{\Int}$, whose corresponding representations lie in \( \rep_{\Int} \pos \), and such that $\Rec^*\cap \Sec^* = \emptyset$. 
If \[\Rk_\Int \field_\Rec - \Rk_\Int \field_\Sec = \Rk_\Int \field_{\Rec^*} - \Rk_\Int \field_{\Sec^*}\] then $\Rec\supseteq \Rec^*$, $\Sec\supseteq\Sec^*$, and $\Rec\setminus\Rec^* = \Sec\setminus\Sec^*$.
\end{theorem}
\begin{proof}
Rewriting the equation yields
        \[
    \Rk_\Int \field_\Rec   + \Rk_\Int\field_{\Sec^*} = \Rk_\Int\field_{\Rec^*}+ \Rk_\Int\field_{\Sec},
    \]
and by additivity of the rank invariant  %
    \[
    \Rk_\Int(\field_\Rec \oplus \field_{\Sec^*}) = \Rk_\Int(\field_{\Rec^*} \oplus \field_{\Sec}).
    \]
By Proposition~\ref{prop:complete} it follows that $\Rec\cup \Sec^* = \Rec^*\cup \Sec$. As $\Rec^*\cap \Sec^* = \emptyset$, we conclude that $\Rec\supseteq \Rec^*$, $\Sec\supseteq\Sec^*$, and $\Rec\setminus\Rec^* = \Sec\setminus\Sec^*$.
\end{proof}

As an immediate consequence of  Theorem~\ref{thm:uniqueness}, we obtain uniqueness and conditional existence of minimal rank decompositions:
\begin{corollary}\label{cor:min-decomp_exists_unique}
  The minimal rank decomposition $(\Rec^*, \Sec^*)$ of any map $r:\Int\to\Z$ is unique if it exists. Furthermore, if a rank decomposition $(\Rec, \Sec)$ of~$r$ exists, then a minimal one exists and is obtained from it by removing common intervals, that is:
  \[ (\Rec^*, \Sec^*) = (\Rec\setminus\Rec\cap\Sec, \Sec\setminus\Rec\cap \Sec). \]
  \end{corollary}

As another consequence of Theorem~\ref{thm:uniqueness}, we get a connection between the various rank decompositions of a map $\Int\to\Z$:

\begin{corollary}\label{cor:rank_decomp_vs_rank_decomp-first}
Any two rank decompositions $(\Rec, \Sec)$ and $(\Rec', \Sec')$ of $r\colon \Int \to \Z$ satisfy $\Rec\cup \Sec' = \Rec' \cup \Sec$.
\end{corollary}
\begin{proof}
  Let $(\Rec^*, \Sec^*)$ be the minimal rank decomposition of $r$. From Theorem~\ref{thm:uniqueness}, we have $\Rec = \Rec^* \cup \Tec$ and $\Sec = \Sec^* \cup \Tec$ for some finite multi-set $\Tec$ of elements of $\Int$, and similarly, we have $\Rec' = \Rec^* \cup \Tec'$ and $\Sec' = \Sec^* \cup \Tec'$ for some multi-set $\Tec'$. Then,
  \[
  \Rec \cup \Sec' = \Rec^* \cup \Sec^* \cup \Tec \cup \Tec' = \Rec' \cup \Sec.
  \]
\end{proof}

\subsection{The locally finite case}
\label{sec:rk-decomp_downward-finite}

\revisedversion{In this section we equip the collection~\( \Int \) of intervals with the  inclusion order~$\subseteq$, and we assume that it is locally finite as a poset, according to the following definition.
\begin{definition}
\label{def:locallyfinite}
A poset $(P, \leq)$ is \emph{locally finite} if for all $p,q\in P$, the segment $\langle p,q\rangle = \{r\in P \mid p\leq r\leq q\}$ is a finite set. 
\end{definition}
}
We say a map \( \Int \to \Z \) has \revisedversion{\emph{upward finite support}} if its restriction to the upset of any element of \( \Int \)  has finite support. \newversion{As an example}, for any fixed \( \int \in \Int \), the map \( \Rk_\Int \field_\int\colon \jnt \mapsto \Rk_{\jnt} \field_{\int} \) has upward finite support by the description in Proposition~\ref{prop:rank_count} \revisedversion{and the fact that $\langle J,I\rangle$ is a finite set ($\Int$ is locally finite). }
%
More generally, we have the following result.
\newversion{\begin{proposition}
\label{prop:int_locally-finite-supp}
Let $\Rec$ be a multi-set of elements in \( \Int \). If \( \field_{\Rec} \in \rep_{\Int} \pos \), then the map \( \Rk_{\Int} \field_{\Rec} \) has upward finite support.
\end{proposition}}
\begin{proof}
 For any fixed \( \int \in \Int \), \revisedversion{we know from Corollary~\ref{cor:description of rep_I} that $\Rk_{I} \field_{R} \neq 0$ for a finite number intervals $R\in \Rec$. Let us denote this collection of intervals by $\Rec'$. Furthermore, if $I\subset I'$, then $\Rk_{I'} \field_{R} \neq 0$ implies that $R\in \Rec'$ by Proposition~\ref{prop:rank_count}. In particular, \[\Rk_{I'} \field_{\Rec} = \sum_{R\in \Rec'} \Rk_{I'} \field_{R},\]
and this sum is non-zero for a finite number of $I'\subseteq I$ since each $\Rk_{\Int} \field_{R}$ has upward finite support. }
\end{proof}

\revisedversion{
\begin{example}
The poset $\mathbb{Z}$ is locally finite but $\mathbb{R}$ is not. Moreover, let $R=\langle 0,1\rangle$ and let $\Int$ be the collection of intervals in $\mathbb{R}$. Then, \( \Rk_\Int \field_R\) does not have upward finite support: there is an infinite number of intervals $I$ satisfying $\langle 0.5, 0.5\rangle \subseteq I\subseteq \langle 0,1\rangle$. 
\end{example}}

\revisedversion{In the setting of upward finite supports, the existence of minimal rank decompositions can also be \newversion{shown} using Möbius inversions; we have included the full details in the appendix for completeness. Here we present a self-contained and simpler proof. }

\begin{theorem}\label{th:basis}
Let \( \Int \) be a locally finite collection of intervals in \( \pos \). Then any function \( r \colon \Int \to \Z \) with upward finite support can uniquely be written as a (possibly infinite, but pointwise finite) $\Z$-linear combination of the functions \( \Rk_\Int \field_{\int} \) with \( \int \in \Int \).
\end{theorem}

\begin{proof}
  Existence: For any \( \int \in \Int \) we set \( S_{\int} = \{ \jnt \supseteq \int \mid \exists \knt \supseteq \jnt \text{ with } r(\knt) \neq 0 \} \).
  Since \( r \) has upward finite support, its support restricted to the upset of \( \int \) is finite, and so is $S_{\int}$ since \( \Int \) is locally finite.

Now we define a collection of scalars \( \alpha_{\int} \in \Z \) for \( \int \in \Int \), inductively on the size of \( S_{\int} \): If \( S_{\int} = \varnothing \) we set \( \alpha_{\int} = 0 \). Otherwise we set 
\[ \alpha_{\int} = r(\int) - \sum_{\jnt \in S_{\int} \setminus \{ \int\}} \alpha_{\jnt}. \] 
Note that for \( \jnt \in S_{\int} \setminus \{ \int\} \) we have \( S_{\jnt} \subsetneq S_{\int} \), so the terms on the right hand side are already defined. 

Now, using the description of the map \( \Rk_\Int \field_{\int} \) in Proposition~\ref{prop:rank_count}, one immediately verifies that \( r = \sum_{\int \in \Int} \alpha_{\int} \Rk_\Int \field_{\int} \). (Note in particular that this infinite sum is pointwise finite --- on a given interval \( \jnt \) the only possibly non-zero terms are the ones in \( S_{\jnt} \) ---hence well-defined.)

Uniqueness: subtracting two different $\Z$-linear combinations realizing \( r \) from each other, we get a single linear combination \( \sum_{\int \in \Int} \alpha_{\int} \Rk_\Int \field_{\int} \) with non-zero coefficients which sums up to zero. Note that there is at least one maximal \( \int \in \Int \) such that \( \alpha_{\int} \neq 0 \), for otherwise the sum would not be defined. It follows, again using Proposition~\ref{prop:rank_count}, that \( (\sum_{\jnt \in \Int} \alpha_{\jnt} \Rk_\Int \field_{\jnt})(\int) = \alpha_{\int} \neq 0 \), contradicting our assumption.
\end{proof}

\begin{corollary}\label{cor:min_gen-rank_decomp}
Let $\Int$ be a locally finite collection of intervals in $\pos$. Then, for any map $r \colon \Int \to\Z$ with upward finite support, there is a unique pair $\Rec, \Sec$ of disjoint multi-sets of elements of $\Int$ such that \( \field_{\Rec} \) and \( \field_{\Sec} \) lie in $\rep_\Int \pos$  and satisfy the following identity:
\[ r = \Rk_\Int \field_{\Rec} - \Rk_\Int \field_{\Sec}. \]
\end{corollary}
\begin{proof}
  By Theorem~\ref{th:basis}, there is a unique (possibly infinite, but pointwise finite) $\Z$-linear combination of functions $r=\sum_{\int\in\Int} \alpha_{\int}\, \Rk_\Int \field_{\int}$.
  Let then $\Rec = \{ \int\in\Int \mid \alpha_{\int}>0\}$ with multiplicities $\int\mapsto \alpha_{\int}$, and  $\Sec = \{ \int\in\Int \mid \alpha_{\int}<0\}$ with multiplicities $\int\mapsto |\alpha_{\int}|$. It follows from the pointwise-finiteness of the linear combination that \( \Rec \) and \( \Sec \) satisfy the condition in Corollary~\ref{cor:description of rep_I}, so in particular \( \field_{\Rec} \) and \( \field_{\Sec} \) lie in \( \rep_{\Int} \pos \).
\end{proof}

Specializing Theorem~\ref{th:basis} and Corollary~\ref{cor:min_gen-rank_decomp} to the case where $\pos$ is finite and $\Int = \{\seg{i,j} \mid i\leq j\in\pos\}\simeq \RelationAsSet{\pos}$ yields the following results---where $\Rk_\Int$ becomes the usual rank invariant~$\Rk$:
\begin{corollary}
Let \( \pos \) be a finite poset. Then the collection of maps \( \Rk \field_{\seg{a, b}} \) with \( a \leq b \) is a basis \revisedversion{for the free abelian group} \( \Z^{\RelationAsSet{\pos}} \).
\end{corollary}

\begin{corollary}\label{cor:min_rank_decomp}
Given a finite poset $\pos$, for any map $r: \RelationAsSet{\pos} \to\Z$  there is a unique pair $\Rec, \Sec$ of disjoint finite multi-sets of closed segments such that
\[ r = \Rk \field_{\Rec} - \Rk \field_{\Sec}. \]
\end{corollary}

\subsubsection*{Möbius inversions and an explicit formula}
\revisedversion{In Appendix~\ref{sec:localfinite-mobius} we recall how Möbius inversions can be used to construct minimal rank decompositions. Importantly, the Möbius inversion approach gives a way of computing the multiplicity of each interval in the decomposition from a simple inclusion-exclusion formula. The following is a direct consequence of Proposition~\ref{prop explicit mobius inverse} and Proposition~\ref{prop:alpha_mobius}. }

\begin{corollary}\label{cor:alpha_formula}
Let \( \Int \) be a locally finite collection of intervals in a poset \( \pos \), and suppose  that, for any \( \int \in \Int \), there is a finite set \( \int^+ \subseteq \Int \) with the property that
\[ \{ \jnt \in \Int \mid \jnt \supsetneq \int \} = \{ \jnt \in \Int \mid \exists \knt \in \int^+ \colon \jnt \supseteq \knt \}, \]
and that any subset of \( \int^+ \) has a join in \( \Int \).
Let $r\colon \Int\to\Z$ have upward finite support. A pair $(\Rec, \Sec)$ of locally finite multisets of elements of \( \Int \) is a rank decomposition of \( r \) if and only if
 \[ \mult_{\int} \Rec - \mult_{\int} \Sec = r(\int) +  \sum_{\varnothing \neq x \subseteq \int^+} (-1)^{\card x}\, r( \vee x ) \qquad \forall \int \in \Int. \]
\end{corollary}

\begin{remark}\label{rem:alpha_formula}
Dropping the assumption that \( \Int \) is locally finite as a poset, we still get that if \( (\Rec, \Sec) \) is a rank decomposition then the equation of the corollary is satisfied---this follows from Remark~\ref{rem Mobius one sided}.
\end{remark}

The formula of Corollary~\ref{cor:alpha_formula} is most interesting in certain scenarios, e.g. when $\pos$ is a finite grid as in multi-parameter persistence---details and examples are given in Section~\ref{sec:rank-decomp_pers}.

\newversion{
\section{Rank invariants and poset homomorphisms}
\label{sec:ri-poset} Let \( f \colon P \to Q \) be a poset homomorphism and let  \( \operatorname{res}_f \colon \Rep Q \to \Rep P \) denote the restriction functor obtained by pre-composing with $f$, i.e., the functor given on objects by \( (\operatorname{res}_f M)(p) = M(f(p)) \) and on morphisms in the obvious way. A typical example to keep in mind is \( P \) being a subposet of \( Q \), \( f \) being the inclusion of this subposet, and \( \operatorname{res}_f \) the usual restriction. Another example is when $Q\subset P=\R^d$ for $Q$ a finite grid, and $f$ is the lattice homomorphism given by ''rounding down'' to the nearest grid point. Importantly, any finitely presented $\R^d$-module is of the form $ \operatorname{res}_f M$ for some $M\in \rep Q$. 

In this section we focus on the following problem: does a rank decomposition of $M\in \rep Q$ pull back to a rank decomposition of $\operatorname{res}_f M$? In Section~\ref{sec:restrictions_usual}, we answer this in the affirmative for the standard rank invariant. In  Section~\ref{sec:restrictions}, we show that the case of the generalized rank invariant is more subtle and that one in general cannot obtain a rank decomposition of $\operatorname{res}_f M$ from $M$. However, we show that one does obtain such a rank decomposition if $f$ is assumed to be a homomorphism of lattices. This allows us to study (generalized) rank decompositions of finitely presented $\R^d$-modules. Such modules are integral to multiparameter persistence and are not covered by our work in Section~\ref{sec:rk-decomp_downward-finite} as $\R^d$ is not locally finite.
}

\revisedversion{
\subsection{The usual rank invariant and poset homomorphisms}
\label{sec:restrictions_usual}
\newversion{First we show that the usual rank invariant behaves well under restrictions along poset maps. Then, we specialize to the setting of grids in $\R^d$ in order to obtain minimal rank decompositions of finitely presented $\R^d$-modules.} 

\begin{lemma} \label{lemma usual rank restricts}
Let \( M \in \rep Q \). Then for any \( a \leq b \in P \) we have
\[ \Rk_{\seg{a,b}} \operatorname{res}_f M = \Rk_{\seg{f(a), f(b)}} M. \]
\end{lemma}

\begin{proof}
This follows directly from the construction of \( \operatorname{res}_f \): \( (\operatorname{res}_f M)(a) = M(f(a)) \), \( (\operatorname{res}_f M)(b) = M(f(b)) \), and the structure map \( (\operatorname{res}_f M)(a) \to (\operatorname{res}_f M)(b) \) is the structure map \( M(f(a)) \to M(f(b)) \).
\end{proof}

\begin{proposition} \label{pro usual rank and restriction}
Let \( f \colon P \to Q \) be a poset homomorphism. Let \( M \in \rep Q \), and assume that there are multi-sets \( \Rec \) and \( \Sec \) of intervals in \( P \) such that \( \Rk M = \Rk \field_{\Rec} - \Rk \field_{\Sec} \).

Then
\[ \Rk \operatorname{res}_f M = \Rk \field_{f^{-1}(\Rec)} - \Rk \field_{f^{-1}(\Sec)}, \]
where \( f^{-1}(\Rec) = \{ f^{-1}(R) \mid R \in \Rec \} \) and similar for \( \Sec \).
\end{proposition}

\begin{remark}
Note that even if \( R \) is an interval, \( f^{-1} \) might not be: It is convex, but might be disconnected or empty.
\end{remark}

\begin{proof}
For \( a \leq b \in P \) we have \( \Rk_{\seg{a,b}} \operatorname{res}_f M = \Rk_{\seg{f(a), f(b)}} M \) by Lemma~\ref{lemma usual rank restricts}. On the other hand 
\begin{align*}
\Rk_{\seg{a,b}} f^{-1}(\Rec) & = \card{\{ f^{-1}(R) \in f^{-1}(\Rec) \mid \seg{a,b} \subseteq f^{-1}(R) \}}  \\
& = \card{\{ R \in \Rec \mid \underbrace{f(\seg{a,b})}_{=\seg{f(a),f(b)}} \subseteq R \}} \\
& = \Rk_{\seg{f(a), f(b)}} \Rec
\end{align*}
where the first and last equalities hold by Proposition~\ref{prop:rank_count}. It follows that the given equality of ranks implies the desired one.
\end{proof}

In \( \mathbb{R}^d \) we often consider ``half open rectangles'', that is, intervals of the form \( \prod_{i=1}^d [a_i, b_i) \). \newversion{Unlike segments, these have ﬁnitely presented indicator modules, which will be useful for extending our theory to ﬁnitely presented multiparameter persistence modules.} Note that it follows from \newversion{Example~\ref{ex:closedtoall} and Theorem~\ref{thm:uniqueness}} that a decomposition of the usual rank invariant into half open rectangles is still unique if it exists. }

\begin{example}\label{ex upward line} 
Let \( P = \mathbb{R} \), \( Q = \mathbb{R}^d \), and \( s \) and \( t \) points in \( \mathbb{R}^d \) such that \( s_i < t_i \text{ for all } i \). Consider the map
\[ f \colon P \to Q \colon \lambda \mapsto (1-\lambda)s + \lambda t. \]
This is a poset homomorphism, whose image in~$Q$ is an upward sloping line. Let \( M \in \rep Q \), and assume there is a decomposition of the usual rank as
\[ \Rk M = \Rk \field_\Rec - \Rk \field_\Sec \]
for multi-sets \( \Rec \) and \( \Sec \) of half open (resp. open, closed, arbitrary) rectangles. Then we also have a decomposition of the usual rank invariant of the restriction as
\[ \Rk \operatorname{res}_f M = \Rk \field_{f^{-1}(\Rec)} - \Rk \field_{f^{-1}(\Sec)} \]
where the preimage of each half-open (resp. open, closed, arbitrary) rectangle is just the intersection of this rectangle with the upward sloping line, therefore it is either empty or a half-open (resp. open, closed, arbitrary) interval.
\end{example}

\newversion{For the remainder of this section we assume that $P = \prod_{i=1}^d F_i \subset \mathbb{R}^d$ where each $F_i$ is a finite subset of $\mathbb{R}$. We shall refer to $P$ as a \emph{finite grid}. }

\begin{example} \label{example restriction to grid}

Let \( M \in \rep \mathbb{R}^d \), and assume there is a decomposition of the usual rank as
\[ \Rk M = \Rk \field_\Rec - \Rk \field_\Sec \]
for multi-sets \( \Rec \) and \( \Sec \) of half open rectangles. Then \( (\Rec \cap P, \Sec \cap P) \) is a usual rank decomposition of the restriction of \( M \) to \( P \). Here \( \Rec \cap P = \{ R \cap P \mid R \in \Rec \} \) and similar for \( \Sec \). Note that the intersection of any half-open rectangle with the finite grid \( P \) is either a segment or empty, whence we get the honest rank decomposition.
\end{example}

Our final aim of this subsection is to do the opposite of the above example, and go from a finite grid to \( \mathbb{R}^d \). We can extend modules and intervals as follows:

The restriction functor \( \operatorname{res} \colon \rep \mathbb{R}^d \to \rep P \) has a left adjoint, called ``left Kan extension'' and denoted by \( \lkan \) (see Section~\ref{sec:upper_semi-lattices} for an explicit definition). In the next proposition the left Kan extension will take a particularly simple form.

For segments \( R \) in \( P \) we define the extension to \( \mathbb{R}^d \) as
\[ \widehat{R} = \{ p \in \R^d \mid \exists r \in R \colon r \leq p \text{ and } \nexists r \in R,s\in P\setminus R \colon r \leq s \leq p \}, \]
that is, we extend \( R \) upward while we forbid passing grid-points that are not in~\( R \).

With these two concepts, we get the following result.

\begin{proposition} \label{prop usual rank and lkan}
With \( P \) a finite grid in \( \mathbb{R}^d \) as above, let \( M \in \rep P \). Let \( ( \Rec, \Sec ) \) be a rank decomposition of \( M \). Then the left Kan extension of \( M \) to \( \mathbb{R}^d \) has a decomposition of the usual rank invariant as
\[ \Rk \lkan M = \Rk \field_{\widehat{\Rec}} - \Rk \field_{\widehat{\Sec}} \]
where \( \widehat{\Rec} = \{ \widehat{R} \mid R \in \Rec \} \), and similar for \( \widehat{\Sec} \) are multi-sets of half-open rectangles.
\end{proposition}

\begin{proof}
We consider the extended
\[ \prod_{i=1}^d \{ - \infty \} \cup F_i \subseteq ( \{ - \infty \} \cup \mathbb{R})^d \] 
for the following reason: This inclusion has a right adjoint, which we will call \( \lfloor \cdot \rfloor \), given by ``rounding down to the nearest grid-point''. We will apply Proposition~\ref{pro usual rank and restriction} to this right adjoint. Clearly we can consider \( M \) a representation of this extended grid (by putting \( 0 \) in the new points). Likewise we can consider the segments in \( \Rec \) and \( \Sec \) as segments in the extended grid, and they still define a rank decomposition of \( M \).

Note that the adjoint pair of poset homomorphisms induces an adjoint pair of restriction functors (in the opposite order). In particular restriction along \( \lfloor \cdot \rfloor \) gives left Kan extensions along the inclusion of posets. Thus \( \lkan M = \operatorname{res}_{\lfloor \cdot \rfloor} M \).

Now by Proposition~\ref{pro usual rank and restriction} we have the rank decomposition
\[ \Rk \lkan M = \Rk \field_{\lfloor \cdot \rfloor^{-1}(\Rec)} - \Rk \field_{\lfloor \cdot \rfloor^{-1}(\Sec)}. \]

We note that \( \widehat{R} \) is constructed to match \( \lfloor \cdot \rfloor^{-1}(R) \), so the right hand side coincides with \( \Rk \field_{\widehat{\Rec}} - \Rk \field_{\widehat{\Sec}} \).

At first this equality holds in \( (\{ - \infty \} \cup \mathbb{R})^d \), but we may restrict to \( \mathbb{R}^d \) (and in fact all terms vanish on all points involving ``\(- \infty\)").

Finally, it follows from the grid structure of \( P \) that the intervals of the form \( \widehat{R} \) are half-open rectangles.
\end{proof}

\begin{corollary}\label{cor:decomp_half-open}
Let \( M \in \rep \mathbb{R}^d \) be finitely presented. Then there are unique disjoint multi-sets \( \Rec \) and \( \Sec \) of half-open rectangles, such that we have a decomposition of the usual rank invariant as  
\[ \Rk M = \Rk \field_{\Rec} - \Rk \field_{\Sec}. \]
\end{corollary}

\begin{proof}
Since \( M \) is finitely presented it is the left Kan extension of a representation of a finite grid. (Consider the positions of all generators and relations of \( M \), and then form the smallest grid containing all these points.) For this finite grid, we have rank decompositions by Corollary~\ref{cor:min_gen-rank_decomp}. Now apply Proposition~\ref{prop usual rank and lkan}.
\end{proof}

\revisedversion{
\subsection{The generalized rank invariant and poset homomorphisms}

\label{sec:restrictions}

As in the previous subsection, we consider a poset homomorphism \( f \colon P \to Q \). While there we asked for the preservation of decompositions of the usual rank, here we will consider more general rank recompositions. While the strategy is quite parallel, note that the starting point of the previous section -- Lemma~\ref{lemma usual rank restricts} -- was quite a simple observation, while the corresponding statement here -- Corollary~\ref{cor restriction for lattices} -- is somewhat deeper. \newversion{As in the previous section, we ultimately turn to finitely presented $\R^d$-modules and their generalized rank decompositions. }

\medskip
As the generalized rank invariant is defined via limits and colimits, the first step is to understand how these are affected by the restriction functor. To this end, let \( \int \) be an interval in \( P \). We denote by \( f(\int) \) its image in \( Q \), and by \( \overline{f(\int)} \) the convex hull of this image, which is an interval in \( Q \).

\begin{lemma} \label{lemma same limit on restriction}
In the situation above, for any \( M \in \Rep Q \) we have a natural morphism
\[ \varprojlim M|_{\overline{f(\int)}} \to \varprojlim (\operatorname{res}_f M)|_\int, \]
such that for any \( i \in \int \) the triangle formed by the above morphism  and the maps \( \varprojlim M|_{\overline{f(\int)}} \to M(f(i)) \) and \( \varprojlim (\operatorname{res}_f M)|_\int \to (\operatorname{res}_f M)(i) = M(f(i)) \) commutes.

If moreover for any \( j \in \overline{f(\int)} \) the subposet
\[ \{ i \in \int \mid f(i) \leq j \} \]
is connected, then the natural morphism above is an isomorphism.
\end{lemma}

\begin{proof}
Note that \( \varprojlim M|_{\overline{f(\int)}} \) comes with a cone of maps to all \( M(j) \) for \( j \in \overline{f(\int)} \). Considering only the \( j \) of the form \( f(i) \) for some \( i \in \int \), we obtain a cone of maps to all \( M(f(i)) = (\operatorname{res}_f M)(i) \). This cone induces the desired map $\varprojlim M|_{\overline{f(\int)}} \to \varprojlim (\operatorname{res}_f M)|_\int$. The claimed commutativities hold by construction.

For the ``moreover'' point we construct a map in the opposite direction. Again we start with the cone of maps from \( \varprojlim (\operatorname{res}_f M)|_\int \) to the various \( (\operatorname{res}_f M)(i) = M(f(i)) \). We obtain maps to \( M(j) \) for \( j \in \overline{f(\int)} \) by choosing \( i \in \int \) such that \( f(i) \leq j \), and composing with the structure map \( M(f(i)) \to M(j) \). The potential issue with this construction is well-definedness, i.e., independence of the choice of \( i \).  Note that for a given \( j \) our choice for possible \( i \)'s is precisely the collection described in the extra assumption. Therefore, it suffices to check that two compatible choices for \( i' \leq i'' \) lead to the same map to \( M(j) \). This is immediate from the commutativity of both triangles in the following diagram.
\[ \begin{tikzpicture}[xscale=1.5,yscale=.5]
 \node (A) at (0,0) {\(\varprojlim (\operatorname{res}_f M)|_\int \)};
 \node (B) at (2,1) {\(M(f(i'))\)};
 \node (C) at (3,-1) {\(M(f(i''))\)};
 \node (D) at (5,0) {\(M(j)\)};
 \draw [->] (A) to (B);
 \draw [->] (A) to (C);
 \draw [->] (B) to (C);
 \draw [->] (B) to (D);
 \draw [->] (C) to (D);
\end{tikzpicture} \]
With well-definedness established, we also have a cone: for \( j' \leq j'' \) choose \( i \) such that \( f(i) \leq j' \) for both of them. From the universality of limits, we get a map \(\varprojlim (\operatorname{res}_f M)|_\int \to \varprojlim M|_{\overline{f(\int)}} \).

As above, it follows from the construction that the triangles formed by this morphism and the respective maps to the \( M(f(i)) \) commute, and, in particular, that the composition of the two maps between the limits behaves like identity after composition with a map to \( M(f(i)) \). We conclude that the composition \( \varprojlim (\operatorname{res}_f M)|_\int \to \varprojlim M|_{\overline{f(\int)}} \to \varprojlim (\operatorname{res}_f M)|_\int \) is the identity. For the other composition the same claim follows when observing that for all \( j \in \overline{f(\int)} \) there is an~\( i \) with \( f(i) \leq j \). 
\end{proof}

\begin{proposition} \label{prop same rank after restriction}
Let \( f \colon P \to Q \) be a poset homomorphism. Then for any interval \( \int \) in \( P \) and \( M \in \Rep Q \) we have
\[ \Rk_{\overline{f(\int)}} M \leq \Rk_{\int} \operatorname{res}_f M. \]

If moreover for any \( j \in  \overline{f(\int)} \) the posets \( \{ i \in \int \mid f(i) \leq j \} \) and \( \{ i \in \int \mid f(i) \geq j \} \) are connected then we have equality.
\end{proposition}

\begin{proof}
Pick any \( i \in \int \). In the diagram
\[ \begin{tikzpicture}[xscale=1.7]
 \node (A) at (0, .5) { \( \varprojlim M|_{\overline{f(\int)}}  \) };
 \node (B) at (0, -.5) { \( \varprojlim (\operatorname{res}_f M)|_\int \) };
 \node (C) at (2,0) { \( M(f(i)) \) };
 \node (D) at (4, .5) { \( \varinjlim M|_{\overline{f(\int)}}  \) };
 \node (E) at (4, -.5) { \( \varinjlim (\operatorname{res}_f M)|_\int \) };
 \draw [->] (A) to node[left] {\( \cong \)} (B);
 \draw [->] (A) to (C);
 \draw [->] (B) to (C);
 \draw [->] (C) to (D);
 \draw [->] (C) to (E);
 \draw [->] (D) to node[right] {\( \cong \)} (E);
\end{tikzpicture} \]
the left vertical map is an isomorphism, such that the left triangle commutes, by Lemma~\ref{lemma same limit on restriction}. By the dual of that lemma the right vertical map is an isomorphism and the right triangle commutes. Thus the natural maps \( \varprojlim M|_{\overline{f(\int)}} \to \varinjlim M|_{\overline{f(\int)}} \) and \( \varprojlim (\operatorname{res}_f M)|_\int \to \varinjlim (\operatorname{res}_f M)|_\int  \) are connected by isomorphisms, and in particular have the same rank.
\end{proof}

\begin{example}
One might ask if an invariant as natural as rank should always commute with restrictions, without the technical extra assumptions above. However, consider the subposets of \( \mathbb{R}^2 \) given as \( P = \{ (1,0), (0,1), (2,2) \} \) and \( Q = P \cup \{ (1,1) \} \). Consider the (indecomposable) representation of \( Q \) depicted as
\[ \begin{tikzpicture}
 \node at (-1,1) {\( M = \)};
 \node (A) at (1, 0) {\( \field \)};
 \node (B) at (0,1) {\( \field \)};
 \node (C) at (1,1) {\( \field^2 \)};
 \node (D) at (2,2) {\( \field \)};
 \draw [->] (A) to node[right=-.5mm] {\( \left[ \begin{smallmatrix} 1 \\ 0 \end{smallmatrix} \right] \)} (C);
 \draw [->] (B) to node[above=-.5mm] {\( \left[ \begin{smallmatrix} 0 \\ 1 \end{smallmatrix} \right] \)} (C);
 \draw [->] (C) to node[below right=-1mm] {\( \left[ 1 \; 1 \right] \)} (D);
\end{tikzpicture} \]
\newversion{Let \( f \colon P \to Q \) be the} natural inclusion, \( \int = P \), and then \( \overline{f(\int)}  = Q \). In this situation we have
\[ \Rk_{\overline{f(\int)}} M = 0 \lneq \Rk_{\int} \operatorname{res}_f M = 1. \]

As one should expect from the strict inequality, the assumption in the last sentence of Proposition~\ref{prop same rank after restriction} are not met. Indeed \( \{ i \in \int \mid f(i) \leq (1,1) \}  = \{(1,0), (0,1) \} \) is not connected.
\end{example}

\begin{example}
\newversion{Let $f\colon P\to Q$ be a poset homomorphism.} If \( \int = \seg{a,b} \) is a segment, then \( \{ i \in \int \mid f(i) \leq j \} \) is connected, since it contains \( a \) and any other element is compatible with \( a \). Similarly \( \{ i \in \int \mid f(i) \geq j \} \) is connected. \newversion{In particular, Proposition~\ref{prop same rank after restriction} generalizes Lemma~\ref{lemma usual rank restricts}.}
\end{example}

\newversion{The conditions for equality in Proposition~\ref{prop same rank after restriction} hold for general intervals provided one restricts to maps between lattices.}

\begin{lemma}
Let \( f \colon P \to Q \) be a morphism of lattices. Then for any interval \( \int \) in \( P \), and \( j \in \overline{f(\int)} \), the posets \( \{ i \in \int \mid f(i) \leq j \} \) and \( \{ i \in \int \mid f(i) \geq j \} \) are connected.
\end{lemma}

\begin{proof}
For \( i_1, i_2 \in \{ i \in \int \mid f(i) \leq j \} \) also their join lies in this set. The proof for the other poset is dual.
\end{proof}

\begin{corollary} \label{cor restriction for lattices}
Let \( f \colon P \to Q \) be a morphism of lattices. Then for any interval \( \int \) in \( P \) we have
\[ \Rk_{\int} \circ \operatorname{res}_f = \Rk_{\overline{f(\int)}}. \]
\end{corollary}

\begin{definition}
Let \( f \colon P \to Q \) be a poset homomorphism.

We call two collections of intervals \( \Int_P \) and \( \Int_Q \) in \( P \) and \( Q \) {\em compatible}, if
\begin{itemize}
\item for any \( \int_P \in \Int_P \) the convex hull of the image \( \overline{f(\int)_P} \) lies in \( \Int_Q \), and
\item for any \( \int_Q \in \Int_Q \) the preimage \( f^{-1}(\int_Q) \) lies in \( \Int_P \), or is empty.
\end{itemize}
\end{definition}

\begin{corollary} \label{corollary restriction commutes rank dec}
Let \( f \colon P \to Q \) be a lattice homomorphism, and let \( \Int_P \) and \( \Int_Q \) be compatible collections of intervals in \( P \) and \( Q \). Let $M$ be a representation of \( Q \)  admitting a rank decomposition $(\Rec, \Sec)$ over \( \Int_Q \), 

Then $(f^{-1}(\Rec), f^{-1}(\Sec))$ is a rank decomposition of \( \operatorname{res}_f  M \) over \( \Int_P \), where by definition $f^{-1}(\Rec) = \{ f^{-1}(R) \mid R\in\Rec, R \cap f(P) \neq \emptyset \}$ and similarly for \( \Sec \).
\end{corollary}

In the following examples we will consider intervals in \( \mathbb{R}^d \)  and their interactions with related posets. We will call an interval \( \int \) in \( \mathbb{R}^d \) \emph{half open} if
\begin{itemize}
\item \( \forall x \in \int\ \exists \varepsilon > 0 \colon \seg{x, x+ (\varepsilon, \cdots, \varepsilon)} \subseteq \int \), and
\item \( \forall x \not\in \int\ \exists \varepsilon > 0 \colon  \seg{x, x+ (\varepsilon, \cdots, \varepsilon)} \cap \int = \emptyset \).
\end{itemize}
This is equivalent to requiring that its intersection with any upward sloping line is a half-open interval on that line.

\begin{example} \label{ex upward line2}


Consider again the poset homomorphism $P\to Q$ with $P=\R$ and $Q=\R^d$ from Example~\ref{ex upward line}. Note that this is a lattice homomorphism.
Let \( \Int_P = \{ [a, b) \} \) be the collection of half-open intervals in \( \mathbb{R} \), and \( \Int_Q \) be the collection of half-open intervals in the sense above. These are compatible collections of intervals.
Therefore, for any \( M \in \Rep Q \) admitting a rank decompositions $(\Rec, \Sec)$ with elements in $\Int_Q$, its restriction to the line parametrized by \( f \) has an induced rank decomposition with elements in $\Int_P$.

The statement remains true if we restrict \( \Int_Q \) to being any smaller collection of half-open intervals which includes all ``half-open rectangles'' \( \{ [a_1, b_1) \times \cdots \times [a_d, b_d) \} \).
    We also get analogous results when considering \( \Int_P \) to be the closed (resp. open, all) intervals in \( \mathbb{R} \), and \( \Int_Q \) to be a collection of closed (resp. open, all) intervals containing the closed (resp. open, all) rectangles in \( \mathbb{R}^d \). 
\end{example}

\begin{example} \label{ex Kan as restriction}
For a collection of finite subsets \( F_i \subset \mathbb{R} \), consider the finite grid \( \prod_{i=1}^d F_i \subset \mathbb{R}^d \). As in Proposition~\ref{prop usual rank and lkan}, we want to consider left Kan extensions from the grid  to \( \mathbb{R}^d \). As in that proof, we add formal $-\infty$ coordinates, and we notice that left Kan extensions are realized as restrictions along  \( \lfloor \cdot \rfloor \).   

We now observe that \( \lfloor \cdot \rfloor \) is a lattice homomorphism: right adjoints automatically commute with meets, but for joins we use the fact that we are specifically considering grids.

Consider either of the following situations:
\begin{itemize}
\item \( \Int_{\mathbb{R}^d} \) consists of all half-open rectangles, and \( \Int_{\prod_{i=1}^d F_i } \) of all segments, or
\item \( \Int_{\mathbb{R}^d} \) consists of all half-open intervals, and \( \Int_{\prod_{i=1}^d F_i } \) of all intervals.
\end{itemize}
All of these have corresponding versions in the respective posets including ``\( - \infty \)''-s. Then they form compatible collections of intervals, along the map \( \lfloor \cdot \rfloor \).

It then follows that any rank decomposition of \( M \in \rep \prod_{i=1}^d F_i \subset \mathbb{R}^d \) induces a rank decomposition of its left Kan extension.
\end{example}

\begin{corollary}\label{cor:decomp_half-open2}
Let \( M \in \rep \mathbb{R}^d \) be finitely presented. Let \( \Int \) either be the collection of half-open rectangles or the collection of all half-open intervals. Then there are unique disjoint multi-sets \( \Rec \) and \( \Sec \) of elements of \( \Int \), such that
\[ \Rk_{\Int} M = \Rk_{\Int} \field_{\Rec} - \Rk_{\Int} \field_{\Sec}. \]
\end{corollary}

\begin{proof}
Since \( M \) is finitely presented it is the left Kan extension of a representation of a finite grid. (Consider the positions of all generators and relations of \( M \), and then form the smallest grid containing all these points.) For this finite grid, the statement holds by \newversion{Corollary}~\ref{cor:min_gen-rank_decomp}. As a final step, by the above example we know that rank decompositions are preserved by left Kan extensions.
\end{proof}

The following fact is an interesting additional consequence of the results of this section. It follows by comparing the proofs of Corollaries~\ref{cor:decomp_half-open} and~\ref{cor:decomp_half-open2}.
  \begin{corollary} \label{cor relate usual rank to half open}
Let \( M \in \rep \mathbb{R}^d \) be finitely presented. Let \( \Rec \) and \( \Sec \) be multi-sets of half-open rectangles. Then the following are equivalent:
\begin{enumerate}
\item \( ( \Rec, \Sec ) \) is a rank decomposition of \( \Rk_\Int M \) over the collection~$\Int$ of half-open rectangles;
\item the usual rank invariant of \( M \) is decomposed as \( \Rk M = \Rk \field_{\Rec} - \Rk \field_{\Sec} \).
\end{enumerate}
That is, we obtain the same decomposition over the half-open rectangles, independent of whether we consider the usual rank invariant or the generalized rank invariant. 
\end{corollary}

\begin{remark}
This corollary shows that considering decompositions of the usual rank into half-open rectangles, while a priori outside the formal framework of Definition~\ref{def:rank_decomp}, may be considered a rank decomposition as in Definition~\ref{def:rank_decomp} with respect to the class of half-open rectangles.
\end{remark}
}

\section{Rank-exact sequences and resolutions}
\label{sec:rank-exact}

\revisedversion{
  In this section we focus on the usual rank invariant and prove that the families of lower hooks and of upper hooks each yield a basis for its decomposition,
  both in the finite poset case (Theorem~\ref{thm:basis of K for finite}) and in the infinite upper semi-lattice case (Theorem~\ref{thm:upper semilattice}). The reason for considering these particular families of intervals is \newversion{that,}
  as explained in the introduction, these are the intervals that come up naturally when one tries to factor the rank invariant through projective resolutions. Our proof highlights this connection at the homological algebra level\newversion{.}

Our exposition is organized as follows. In Section~\ref{sec:exact_struct} we introduce the class of short rank-exact sequences, show that it defines the structure of an exact category on~$\rep \pos$, characterize its indecomposable projectives and injectives as being respectivel the lower hook modules and the upper hook modules, and introduce its associated Grothendieck group. Then, in Sections~\ref{sec:finite_posets} and~\ref{sec:upper_semi-lattices} we prove that finitely presented modules admit finite projective resolutions in the rank-exact category, both in the finite poset case and in the infinite upper semi-lattice case. From there we conclude on the existence and uniqueness of rank decompositions by interval modules supported on lower or upper hooks.  }

\subsection{\revisedversion{The rank-exact category}}
\label{sec:exact_struct}

Exact categories were introduced by Quillen \cite{Quillen_exact} (a similar definition was given by Heller \cite{Heller_exact}), with the aim of having a minimal categorical setup in which the standard methods of homological algebra of abelian categories can be applied. \newversion{See B{\"u}hler's survey~\cite{buhler2010exact} for a good introduction to the subject.}

More precisely, an exact category is an additive category together with a class of \newversion{{\em kernel-cokernel pairs}, i.e., pairs of morphisms $\xymatrix{M\ar^-{i}[r]&M'\ar^-{p}[r]&M''}$ such that $i$ is a kernel of~$p$ and $p$ is a cokernel of~$i$.}
    We call the first morphism \newversion{of such a pair}
    an admissible monic, and the second \newversion{morphism}
    an admissible epic. With this notation, we require that:
\begin{itemize}
\item[E0] \newversion{all identity morphisms are both admissible monics and admissible epics;}
\item[E1] compositions of admissible monics (epics) are admissible monics (epics) again;
\item[E2] push-outs of admissible monics along arbitrary maps exist and are admissible monics again; pull-backs of admissible epics along arbitrary maps exist and are admissible epics again.
\end{itemize}
\revisedversion{Such a class of \newversion{kernel-cokernel pairs}
  is called an \emph{exact structure}.}

\newversion{Abelian categories, such as for instance the category of persistence modules over a fixed poset, have a canonical exact structure given by all the pairs of morphisms appearing in their short exact sequences. However, other (weaker) exact structures may be considered as well, by restricting the focus to certain classes of short exact sequences---those are then called {\em distinguished} short exact sequences. This is what we do in this section.}

\bigskip

\revisedversion{
  For the first part of this section, $\pos$ is an arbitrary poset. We
  consider the following class of short exact sequences.}

  \begin{definition}
A short exact sequence \( 0 \to A \to B \to C \to 0 \) in \( \rep \pos \) will be called \emph{rank-exact} if \( \Rk B = \Rk A + \Rk C \). We denote the collection of rank-exact sequences by \( \Erank \).
  \end{definition}

\revisedversion{
It is not a priori clear that the rank-exact sequences form an exact structure. We will show this using the following construction.

\begin{proposition}[{\cite[Propositions~1.4 and 1.7]{DRSS}}] \label{proposition exact from proj}
Let \( \mathcal{X} \) be a class of objects in an abelian category. Consider the class of exact sequences
\[ \mathcal{E}_{\mathcal{X}} = \{ 0 \to A \to B \to C \to 0 \mid \forall X \in \mathcal{X} \colon \Hom(X, B) \to \Hom(X, C) \text{ epic} \}. \]
Then \( \mathcal{E}_{\mathcal{X}} \) defines an exact structure.
\end{proposition}

Conversely to this construction, for a given exact structure \( \mathcal{E} \) \newversion{on an abelian category}, one may consider all objects \( X \) such that \( \Hom(X, -) \) takes distinguished short exact sequences to short exact sequences of abelian groups. These objects are called \( \mathcal{E} \)-projective; \( \mathcal{E} \)-injective objects are defined dually, i.e., by replacing \( \Hom(X, -) \)  with \( \Hom(-, X) \).
\newversion{\begin{remark}\label{rem:proj from exact}
It follows immediately from the construction that all objects in \( \mathcal{X} \) are \( \mathcal{E}_{\mathcal{X}} \)-projective, \newversion{and therefore that finite direct sums of objects in \( \mathcal{X} \) are also \( \mathcal{E}_{\mathcal{X}} \)-projective}.
\end{remark}}
  
For our instance of rank-exact sequences the indicator representations supported on intervals of the following
forms will turn out to be the relevant representations.} Here we use the  convention that \( - \infty < i < \infty \) for any \( i \in \pos \).
\begin{align*}
\dhook{i, j} & = \{ \ell \in \pos \mid i \leq \ell \not\geq j \} && \text{for } i < j \in \pos \cup \{ \infty \} & \mbox{\em lower hook}\\
\uhook{i, j} & = \{ \ell \in \pos \mid i \not\geq \ell \leq j \} && \text{for } i < j \in \pos \cup \{ - \infty \} & \mbox{\em upper hook}
\end{align*}

The following lemma connects these indicator modules to ranks.

\begin{lemma} \label{lemma rank vs L-shape}
Let \( \seg{i,j} \) be a segment in \( \pos \).

\begin{enumerate}
\item There is a short exact sequence of functors 
\[ 0 \to \Hom( \field_{\dhook{i,j}}, -) \to (-)_i \to (-)_j, \]
\revisedversion{from $\rep \pos$ to $\vec_\field$}, where $M_i$ stands for $M(i)$. 
\item As functions on poset representations, we have
\[ \Rk_{\seg{i,j}} = \dim (-)_i - \dim \Hom(  \field_{\dhook{i,j}}, -). \]
\end{enumerate}
\end{lemma}

\begin{proof}
\revisedversion{For $i\in P$, let $P_i = \field_{\langle i, \infty\langle}$.} 
Note that \( (-)_i \cong \Hom( P_i, - ) \) and similarly for \( j \). Now the first claim follows from the exact sequence
\[ P_j \to P_i \to \field_{\dhook{i,j}} \to 0 \]
and the left exactness of the \( \Hom \)-functor.

For the second claim, observe that \( \Rk_{\seg{i,j}} \) is the dimension of the image of the map \( (-)_i \to (-)_j \). In particular the formula for the rank is obtained by taking dimensions of the terms in the short exact sequence
\[ 0 \to \Hom( \field_{\dhook{i,j}}, -) \to (-)_i \to \text{image} \to 0.  \]\qedhere
\end{proof}

\newversion{This result allows us to connect rank-exactness to \( \Hom \)s from hook modules.}

\begin{proposition} \label{prop rank exact by Hom}
For \( i < j \) in \( \pos \), and a short exact sequence
\[ \mathbb E \colon 0 \to A \to B \to C \to 0 \]
in \( \rep \pos \), the following are equivalent.
\begin{enumerate}
\item \( \Rk_{\seg{i,j}} B = \Rk_{\seg{i,j}} A + \Rk_{\seg{i,j}} C \);
\item \( \Hom( \field_{\dhook{i, j}}, \mathbb E ) \) is exact;
\item \( \Hom( \mathbb E, \field_{\uhook{i,j}} ) \) is exact.
\end{enumerate}
\end{proposition}

\begin{proof}
By Lemma~\ref{lemma rank vs L-shape}(2), we see that (1) is equivalent to
\[ \dim  \Hom(  \field_{\dhook{i,j}}, B) =  \dim \Hom(  \field_{\dhook{i,j}}, A) + \dim \Hom(  \field_{\dhook{i,j}}, C). \]
\revisedversion{By left-exactness of $\Hom(  \field_{\dhook{i,j}}, -)$,} this equation holds if and only if \( \Hom( \field_{\dhook{i, j}}, \mathbb E ) \) is exact. This shows that (1) is equivalent to (2).

The equivalence of (1) and (3) is dual.
\end{proof}

\revisedversion{We are now ready to show that \( \Erank \) indeed defines an exact structure.}

\begin{theorem} \label{thm is exact cat}
The collection \(\Erank \) of rank-exact sequences defines the structure of an exact category on \( \rep \pos \).
 \revisedversion{Moreover, \newversion{finite} direct sums of objects of the form \( \field_{\dhook{i, j}} \), where \( i < j \in P \cup \{ \infty \} \), are projective with respect to this exact structure. Dually, \newversion{finite} direct sums of objects \( \field_{\uhook{i, j}} \), with \( i < j \in P \cup \{ - \infty \} \), are \( \Erank \)-injective.}
\end{theorem}
\begin{proof}
  \revisedversion{In view of Proposition~\ref{prop rank exact by Hom}, we have \( \Erank = \mathcal{E}_{\{ \field_{\dhook{i, j}} \mid i < j \}} \) in the notation of Proposition~\ref{proposition exact from proj}. It now follows from that proposition that \( \Erank \) is an exact structure. The claim on \( \Erank \)-projectives follows by \newversion{Remark~\ref{rem:proj from exact}},
    and the corresponding statement for injectives is dual.}
\end{proof}

\begin{definition}\label{def:rel_Grothendieck_group}
Let \( \mathcal{E} \) be an exact category. The \emph{Grothendieck group} \( \K(\mathcal{E}) \) is the free abelian group on the objects of \( \mathcal{E} \), subject to the relation that
\( [B] = [A] + [C] \) whenever \( 0 \to A \to B \to C \to 0 \) is a distinguished short exact sequence in \( \mathcal{E}. \)
\end{definition}

\begin{remark}
Alternatively, the Grothendieck group may be characterized as a universal map from objects of \( \mathcal{E} \) to an abelian group which is additive on distinguished short exact sequences. 
\end{remark}
\revisedversion{
\begin{remark} Grothendieck groups are central to homological algebra, and so it is unsurprising that they have already appeared in TDA in various forms. E.g., Patel~\cite{patel2018generalized} uses the standard Grothendieck group \newversion{(i.e., the one associated with all the short exact sequences in the category)} to define generalized persistence diagrams for one-parameter persistence modules valued in abelian categories, and Asashiba et al.~\cite{asashiba2019approximation} use the split Grothendieck group \newversion{(i.e., the one associated  the split exact sequences)} to obtain interval approximations of persistence modules indexed by a two-parameter grid. As a result, they present the generalized persistence diagram of a module~$\Mod$ given in~\cite{kim2018generalized} as a rank-preserving `projection' of~$\Mod$ onto the split Grothendieck group. 
The resulting approximation is thus essentially a minimal rank decomposition of $\Mod$ with respect to the generalized rank invariant over all intervals. The split Grothendieck group is natural in this context as it allows for the formal addition of (equivalence classes of) modules. This paper is the first in TDA to consider Grothendieck groups for exact categories. 
\end{remark}
}

\begin{lemma}
Let \( \pos \) be a poset. Then \( \Rk \) defines a \revisedversion{homomorphism of abelian groups}  \( \K(\Erank) \to \mathbb{Z}^{\RelationAsSet{\pos}} \).
\end{lemma}

\begin{proof}
By construction, \( \Rk \) is additive on rank-exact sequences,  therefore it factors through the Grothendieck group.
\end{proof}

\subsection{Finite posets}
\label{sec:finite_posets}
\revisedversion{As remarked above, standard constructions from homological algebra can be applied in an exact category. We now recall the definition of the constructions we need in this section; it is therefore assumed that the underlying category is $\rep \pos$, \newversion{where $\pos$ is a finite poset}.  An exact category \( \mathcal{E} \) has \emph{enough projectives} if, for every object $A$ in $\mathcal{E}$, there exists an \( \mathcal{E} \)-projective object $P$ and an admissible epic $P\to A$. The notion of \emph{enough injectives} is defined dually. A sequence of morphisms
\[\cdots \rightarrow A_1\xrightarrow{f_1} A_2 \xrightarrow{f_2} A_3 \rightarrow \cdots\]
is \emph{\( \mathcal{E}\)-exact} if, for all $i$, ${\rm im}(f_{i-1}) \to A_i \to {\rm im}(f_i)$ is a distinguished short exact sequence (here we have used that $\rep \pos$ has kernels and cokernels; see \cite[Lemma 8.4]{buhler2010exact}). An \emph{$\mathcal{E}$-projective resolution} of $A$ is an $\mathcal{E}$-exact sequence 
\[\cdots P_2 \rightarrow P_1 \rightarrow P_0 \to M\]
where each $P_i$ is $\mathcal{E}$-projective. \emph{$\mathcal{E}$-Injective resolutions} are defined dually. 

Since the objects we consider in this section are finite dimensional representations of finite posets, it follows from \cite[Section 2.1]{auslander_reiten_smalo_1995} that if the exact structure has enough projectives, then it also has \emph{$\mathcal{E}$-projective covers} (\newversion{i.e.,} left minimal morphisms in \cite[Section 2.1]{auslander_reiten_smalo_1995}). From this, one constructs a \emph{minimal} $\mathcal{E}$-projective resolution by inductively computing  $\mathcal{E}$-projective covers for kernels, precisely as in the classical setting. Since $\mathcal{E}$-projective covers are unique up to isomorphism (\cite[Proposition 2.1]{auslander_reiten_smalo_1995}), so are minimal $\mathcal{E}$-projective resolutions. Furthermore, it follows from \cite[Theorem 2.2]{auslander_reiten_smalo_1995} that if $P_\bullet\to A$ is an $\mathcal{E}$-projective resolution of $A$, then $P_\bullet \cong P_\bullet' \oplus P_\bullet''$ where $P_\bullet'$ is a minimal $\mathcal{E}$-projective resolution of $A$. The case of minimal $\mathcal{E}$-injective resolutions is dual. 
} 

\begin{theorem} \label{thm:enough proj}
Let \( \pos \) be a finite poset. \revisedversion{Then \( \Erank \) has enough projectives, which are \newversion{finite} direct sums of lower hook modules, and enough injectives, which are \newversion{finite} direct sums of upper hook modules. Moreover, every object has a finite \( \Erank \)-projective and a finite \( \Erank \)-injective resolution.}
\end{theorem}
\begin{proof}
\revisedversion{Since \( \pos \) is finite, the category \( \rep \pos \) is equivalent to the module category of a finite dimensional algebra. In particular we can apply the results of \cite{Auslander-Solberg}. Proposition~1.10 of that paper gives the desired description of \( \Erank \)-projectives, and Theorem~1.12 shows that there are enough projectives.}

To see that any object has a finite \( \Erank \)-projective resolution, note that
\[ \Hom( \field_{\dhook{i, j}}, \field_{\dhook{m, n}}) \neq 0 \quad \Longrightarrow \quad i \geq m \text{ and } j \geq n \]
and \( \End( \field_{\dhook{i, j}}) = \field \). It follows that the indices appearing in a minimal \( \Erank \)-projective resolution increase from term to term, and thus we have a finite resolution since $\pos$ is finite.

\revisedversion{The claims for \( \Erank \)-injectives and \( \Erank \)-injective resolutions are dual.}
\end{proof}

\begin{corollary} \label{cor:L-shapes generate}
The \revisedversion{elements} \( [ \field_{\dhook{i, j}} ] \) generate \( \K(\Erank) \).
\end{corollary}

\begin{proof}
\revisedversion{For any \( \Mod \in \rep P \), consider the finite  \( \Erank \)-projective resolution
\[P_k\xrightarrow{p_k}\cdots P_2 \xrightarrow{p_2} P_1 \xrightarrow{p_1} P_0 \xrightarrow{p_0} \Mod.\]
By definition, 
\begin{align*} \Rk \Mod &= \Rk P_0 - \Rk {\rm im}(p_1) = \Rk P_0 - (\Rk P_1 - \Rk {\rm im}(p_2))\\
&= \Rk P_0 - \Rk P_1 + (\Rk P_2 - \Rk{\rm im}(p_3)) = \cdots = \sum_{n=0}^k (-1)^n \Rk P_n.
\end{align*}}
In \( \K(\Erank) \) this gives
\[ [\Mod] = \sum_{n=0}^{k} (-1)^n [P_n]. \]
\end{proof}

\begin{theorem} \label{thm:basis of K for finite}
Let \( \pos \) be a finite poset. Then
\[ \Rk \colon \K(\Erank) \to \mathbb{Z}^{\RelationAsSet{\pos}} \]
is an isomorphism.
Moreover, the following three sets are bases for \( \K(\Erank) \):
\begin{align*}
  &\{ [\field_{\seg{i,j}}] \mid i \leq j \in P \},\\[0.5em]
  &\{ [\field_{\dhook{i,j}} ] \mid i < j \in P \cup \{ \infty \} \}, \ \text{and} \\[0.5em]
  &\{ [\field_{\uhook{i,j}} \mid i < j \in P \cup \{ - \infty \} \}.
\end{align*}
\end{theorem}

\begin{proof}
By Theorem~\ref{th:basis}, the map \( \Rk \colon \K(\Erank) \to \mathbb{Z}^{\RelationAsSet{\pos}} \) is surjective. By Corollary~\ref{cor:L-shapes generate}, \( \K(\Erank) \) is generated by \( | \RelationAsSet{\pos} | \) elements. \revisedversion{For the map to be surjective, the image of a choice of \( | \RelationAsSet{\pos} | \) generators must be linearly independent.}  It follows that \( \Rk \) is an isomorphism, and moreover that these generators are linearly independent. In particular we have established the second basis for \( \K(\Erank) \). The final one is dual. The first one follows from Theorem~\ref{th:basis}, together with the fact that \( \Rk \) is an isomorphism.
\end{proof}

\revisedversion{
 Let us illustrate Theorem~\ref{thm:basis of K for finite} with a concrete example.
\begin{example}\label{ex:rel_Grothendieck_group_decomp}
  We revisit the example from Figure~\ref{fig:decomp_indec2_grid}.
  Note that we have the following rank-exact sequences, \newversion{where each representation is indecomposable}:
  \[ \resizebox{1\hsize}{!}{ $\displaystyle
 0\ \longrightarrow\ 
 \vcenter{ \xymatrix{ 0\ar[r] & 0\ar[r] & 0 \\ 0\ar[r]\ar[u] & k\ar^-{\id}[r]\ar[u] & k\ar[u] \\ 0\ar[r]\ar[u] & 0\ar[r]\ar[u] & 0\ar[u]  }}
  \ \longrightarrow\ 
  \vcenter{ \xymatrix{ k\ar^-{\id}[r] & k\ar[r] & 0 \\ k\ar^-{\left[\begin{smallmatrix}1\\0\end{smallmatrix}\right]}[r]\ar^-{\id}[u] & k^2\ar^-{\left[\begin{smallmatrix}1&1\end{smallmatrix}\right]}[r]\ar_-{\left[\begin{smallmatrix}1&0\end{smallmatrix}\right]}[u] & k\ar[u] \\ 0\ar[r]\ar[u] & 0\ar[r]\ar[u] & 0\ar[u] }}
  \ \oplus\ 
  \vcenter{\xymatrix{ 0\ar[r] & 0\ar[r] & 0 \\ 0\ar[r]\ar[u] & k\ar^-{\id}[r]\ar[u] & k\ar[u] \\ 0\ar[r]\ar[u] & k\ar^-{\id}[r]\ar_-{\id}[u] & k\ar_-{\id}[u]   }}
  \ \longrightarrow\ 
  \vcenter{\xymatrix{ k\ar^-{\id}[r] & k\ar[r] & 0 \\ k\ar^-{\left[\begin{smallmatrix}1\\0\end{smallmatrix}\right]}[r]\ar^-{\id}[u] & k^2\ar^-{\left[\begin{smallmatrix}1&1\end{smallmatrix}\right]}[r]\ar_-{\left[\begin{smallmatrix}1&0\end{smallmatrix}\right]}[u] & k\ar[u] \\ 0\ar[r]\ar[u] & k\ar^-{\id}[r]\ar_-{\left[\begin{smallmatrix}0\\1\end{smallmatrix}\right]}[u] & k\ar_-{\id}[u] }}
  \ \longrightarrow\ 0
 $}\]
and
\[ \resizebox{1\hsize}{!}{ $\displaystyle
  0\ \longrightarrow\ 
  \vcenter{\xymatrix{ k\ar^-{\id}[r] & k\ar[r] & 0 \\ k\ar^-{\left[\begin{smallmatrix}1\\0\end{smallmatrix}\right]}[r]\ar^-{\id}[u] & k^2\ar^-{\left[\begin{smallmatrix}1&1\end{smallmatrix}\right]}[r]\ar_-{\left[\begin{smallmatrix}1&0\end{smallmatrix}\right]}[u] & k\ar[u] \\ 0\ar[r]\ar[u] & 0\ar[r]\ar[u] & 0 \ar[u] }}
\ \longrightarrow\ 
  \vcenter{\xymatrix{ k\ar^-{\id}[r] & k\ar[r] & 0 \\ k\ar^-{\id}[r]\ar^-{\id}[u] & k\ar[r]\ar_-{\id}[u] & 0\ar[u] \\ 0\ar[r]\ar[u] & 0\ar[r]\ar[u] & 0\ar[u] }}
  \ \oplus\ 
  \vcenter{\xymatrix{ 0\ar[r] & 0\ar[r] & 0 \\ k\ar^-{\id}[r]\ar[u] & k\ar^-{\id}[r]\ar[u] & k\ar[u] \\ 0\ar[r]\ar[u] & 0\ar[r]\ar[u] & 0\ar[u]  }}
  \ \oplus\ 
    \vcenter{\xymatrix{ 0\ar[r] & 0\ar[r] & 0 \\ 0\ar[r]\ar[u] & k\ar[r]\ar[u] & 0\ar[u] \\ 0\ar[r]\ar[u] & 0\ar[r]\ar[u] & 0\ar[u]  }}
    \ \longrightarrow\ 
      \vcenter{\xymatrix{ 0\ar[r] & 0\ar[r] & 0 \\ k\ar^-{\id}[r]\ar[u] & k\ar[r]\ar[u] & 0\ar[u] \\ 0\ar[r]\ar[u] & 0\ar[r]\ar[u] & 0\ar[u]  }}
      \longrightarrow\ 0
      $}\]
\newversion{In the Grothendieck group~$\K(\Erank^{\rm fp})$, these rank-exact sequences yield the following equations:
\[
\resizebox{1\hsize}{!}{ $\displaystyle
  \left[\vcenter{\xymatrix{ k\ar^-{\id}[r] & k\ar[r] & 0 \\ k\ar^-{\left[\begin{smallmatrix}1\\0\end{smallmatrix}\right]}[r]\ar^-{\id}[u] & k^2\ar^-{\left[\begin{smallmatrix}1&1\end{smallmatrix}\right]}[r]\ar_-{\left[\begin{smallmatrix}1&0\end{smallmatrix}\right]}[u] & k\ar[u] \\ 0\ar[r]\ar[u] & 0\ar[r]\ar[u] & 0\ar[u] } }\right]
+
\left[   \vcenter{\xymatrix{ 0\ar[r] & 0\ar[r] & 0 \\ 0\ar[r]\ar[u] & k\ar^-{\id}[r]\ar[u] & k\ar[u] \\ 0\ar[r]\ar[u] & k\ar^-{\id}[r]\ar_-{\id}[u] & k\ar_-{\id}[u]   }}\right]
=
\left[  \vcenter{ \xymatrix{ 0\ar[r] & 0\ar[r] & 0 \\ 0\ar[r]\ar[u] & k\ar^-{\id}[r]\ar[u] & k\ar[u] \\ 0\ar[r]\ar[u] & 0\ar[r]\ar[u] & 0\ar[u]  }} \right]
+
\left[\vcenter{\xymatrix{ k\ar^-{\id}[r] & k\ar[r] & 0 \\ k\ar^-{\left[\begin{smallmatrix}1\\0\end{smallmatrix}\right]}[r]\ar^-{\id}[u] & k^2\ar^-{\left[\begin{smallmatrix}1&1\end{smallmatrix}\right]}[r]\ar_-{\left[\begin{smallmatrix}1&0\end{smallmatrix}\right]}[u] & k\ar[u] \\ 0\ar[r]\ar[u] & k\ar^-{\id}[r]\ar_-{\left[\begin{smallmatrix}0\\1\end{smallmatrix}\right]}[u] & k\ar_-{\id}[u] }}\right]
$}
\]
and
\[
\resizebox{1\hsize}{!}{ $\displaystyle
  \left[\vcenter{\xymatrix{ k\ar^-{\id}[r] & k\ar[r] & 0 \\ k\ar^-{\id}[r]\ar^-{\id}[u] & k\ar[r]\ar_-{\id}[u] & 0\ar[u] \\ 0\ar[r]\ar[u] & 0\ar[r]\ar[u] & 0\ar[u] }}\right]
+
  \left[\vcenter{\xymatrix{ 0\ar[r] & 0\ar[r] & 0 \\ k\ar^-{\id}[r]\ar[u] & k\ar^-{\id}[r]\ar[u] & k\ar[u] \\ 0\ar[r]\ar[u] & 0\ar[r]\ar[u] & 0\ar[u]  }}\right]
+
\left[\vcenter{\xymatrix{ 0\ar[r] & 0\ar[r] & 0 \\ 0\ar[r]\ar[u] & k\ar[r]\ar[u] & 0\ar[u] \\ 0\ar[r]\ar[u] & 0\ar[r]\ar[u] & 0\ar[u]  }}\right]
=
  \left[\vcenter{\xymatrix{ k\ar^-{\id}[r] & k\ar[r] & 0 \\ k\ar^-{\left[\begin{smallmatrix}1\\0\end{smallmatrix}\right]}[r]\ar^-{\id}[u] & k^2\ar^-{\left[\begin{smallmatrix}1&1\end{smallmatrix}\right]}[r]\ar_-{\left[\begin{smallmatrix}1&0\end{smallmatrix}\right]}[u] & k\ar[u] \\ 0\ar[r]\ar[u] & 0\ar[r]\ar[u] & 0 \ar[u] }}\right]
+
\left[\vcenter{\xymatrix{ 0\ar[r] & 0\ar[r] & 0 \\ k\ar^-{\id}[r]\ar[u] & k\ar[r]\ar[u] & 0\ar[u] \\ 0\ar[r]\ar[u] & 0\ar[r]\ar[u] & 0\ar[u]  }}\right]
$}
\]
Hence:}
\begin{align*} \resizebox{0.245\hsize}{!}{ $\displaystyle
  \left[\vcenter{\xymatrix{ k\ar^-{\id}[r] & k\ar[r] & 0 \\ k\ar^-{\left[\begin{smallmatrix}1\\0\end{smallmatrix}\right]}[r]\ar^-{\id}[u] & k^2\ar^-{\left[\begin{smallmatrix}1&1\end{smallmatrix}\right]}[r]\ar_-{\left[\begin{smallmatrix}1&0\end{smallmatrix}\right]}[u] & k\ar[u] \\ 0\ar[r]\ar[u] & k\ar^-{\id}[r]\ar_-{\left[\begin{smallmatrix}0\\1\end{smallmatrix}\right]}[u] & k\ar_-{\id}[u] }}\right]$}
 & = \resizebox{0.7\hsize}{!}{ $\displaystyle
    \left[\vcenter{\xymatrix{ k\ar^-{\id}[r] & k\ar[r] & 0 \\ k\ar^-{\left[\begin{smallmatrix}1\\0\end{smallmatrix}\right]}[r]\ar^-{\id}[u] & k^2\ar^-{\left[\begin{smallmatrix}1&1\end{smallmatrix}\right]}[r]\ar_-{\left[\begin{smallmatrix}1&0\end{smallmatrix}\right]}[u] & k\ar[u] \\ 0\ar[r]\ar[u] & 0\ar[r]\ar[u] & 0\ar[u] } }\right]
    + \left[   \vcenter{\xymatrix{ 0\ar[r] & 0\ar[r] & 0 \\ 0\ar[r]\ar[u] & k\ar^-{\id}[r]\ar[u] & k\ar[u] \\ 0\ar[r]\ar[u] & k\ar^-{\id}[r]\ar_-{\id}[u] & k\ar_-{\id}[u]   }}\right]
    - \left[  \vcenter{ \xymatrix{ 0\ar[r] & 0\ar[r] & 0 \\ 0\ar[r]\ar[u] & k\ar^-{\id}[r]\ar[u] & k\ar[u] \\ 0\ar[r]\ar[u] & 0\ar[r]\ar[u] & 0\ar[u]  }} \right]
    $}\\[1ex]
  &= \resizebox{0.7\hsize}{!}{ $\displaystyle
  \left[\vcenter{\xymatrix{ k\ar^-{\id}[r] & k\ar[r] & 0 \\ k\ar^-{\id}[r]\ar^-{\id}[u] & k\ar[r]\ar_-{\id}[u] & 0\ar[u] \\ 0\ar[r]\ar[u] & 0\ar[r]\ar[u] & 0\ar[u] }}\right]
  + \left[\vcenter{\xymatrix{ 0\ar[r] & 0\ar[r] & 0 \\ k\ar^-{\id}[r]\ar[u] & k\ar^-{\id}[r]\ar[u] & k\ar[u] \\ 0\ar[r]\ar[u] & 0\ar[r]\ar[u] & 0\ar[u]  }}\right]
  + \left[\vcenter{\xymatrix{ 0\ar[r] & 0\ar[r] & 0 \\ 0\ar[r]\ar[u] & k\ar[r]\ar[u] & 0\ar[u] \\ 0\ar[r]\ar[u] & 0\ar[r]\ar[u] & 0\ar[u]  }}\right]
    $}\\[1ex]
  & \hspace{13pt} \resizebox{0.7\hsize}{!}{ $\displaystyle
    - \left[\vcenter{\xymatrix{ 0\ar[r] & 0\ar[r] & 0 \\ k\ar^-{\id}[r]\ar[u] & k\ar[r]\ar[u] & 0\ar[u] \\ 0\ar[r]\ar[u] & 0\ar[r]\ar[u] & 0\ar[u]  }}\right]
    + \left[   \vcenter{\xymatrix{ 0\ar[r] & 0\ar[r] & 0 \\ 0\ar[r]\ar[u] & k\ar^-{\id}[r]\ar[u] & k\ar[u] \\ 0\ar[r]\ar[u] & k\ar^-{\id}[r]\ar_-{\id}[u] & k\ar_-{\id}[u]   }}\right]
    - \left[  \vcenter{ \xymatrix{ 0\ar[r] & 0\ar[r] & 0 \\ 0\ar[r]\ar[u] & k\ar^-{\id}[r]\ar[u] & k\ar[u] \\ 0\ar[r]\ar[u] & 0\ar[r]\ar[u] & 0\ar[u]  }} \right]
    $},
  \end{align*}
\newversion{which, via the group homomorphism $\Rk \colon \K(\Erank) \to \mathbb{Z}^{\RelationAsSet{\pos}}$, recovers the minimal rank decomposition of Figure~\ref{fig:decomp_indec2_grid}.}
\end{example}
}

\subsection{Finitely presented representations of upper semi-lattices}
\label{sec:upper_semi-lattices}

Let \( \pos \) be an upper semi-lattice---meaning that any finite set of elements of \( P \) has a join, i.e., a least upper bound. For any finite subposet \( \subpos \) that is closed under taking joins, we define the operator $\lfloor \cdot \rfloor_\subpos$ taking any element \( p \in \pos \) for which the set \( \{ s \in \subpos \mid s \leq p \} \) is non-empty to the (unique) maximum element $\lfloor p \rfloor_\subpos = \max\{ s \in \subpos \mid s \leq p\}$.

\revisedversion{The left Kan extension of a persistence module along a poset homomorphism is given by a colimit computation~\cite[Theorem 6.2.1]{riehl2017category}. In our situation, this reduces to }
\[ \lkan\Mod(p) = \begin{cases} \Mod(\lfloor p \rfloor_{\subpos}) & \text{if } \revisedversion{\lfloor p \rfloor_{Q}} \text{ exists}\\ 0 & \text{otherwise} \end{cases} \]
\revisedversion{and for $p\leq p'$, the map $\lkan\Mod(p) \to  \lkan\Mod(p')$ is either trivial or equal to $\Mod(\lfloor p \rfloor_{\subpos}) \to \Mod(\lfloor p' \rfloor_{\subpos})$ if both $\revisedversion{\lfloor p \rfloor_{Q}}$ and $\revisedversion{\lfloor p' \rfloor_{Q}}$ exist.}

This description of left Kan extensions has the following immediate consequence.

\begin{lemma}\label{lem:Lanpresreseq}
In the situation described above, left Kan extensions preserve rank-exact sequences.
\end{lemma}

\begin{proof}
This follows immediately from the description of the left Kan extensions.
\end{proof}

\begin{lemma}\label{lem:morpos}
Let \( f\colon S \to \pos \) be any morphism of posets. Then
\[ \lkan \field_{\dhook{s,t}} = \field_{\dhook{f(s), f(t)}} \]
for any \( s < t \in S \cup \{ \infty \} \).
\end{lemma}

\begin{proof}
\revisedversion{First, we recall the following general fact: let $F$ and $G$ be functors between abelian categories and assume that $F$ is left adjoint to $G$. If $G$ is exact, then $F(M)$ is projective (with respect to the standard exact structure) if $M$ is projective. To see this, observe that $\Hom(F(M),-) \cong \Hom(M,G(-))$ from the adjoint property, and that the right hand side is a composition of two exact functors. We conclude that $\Hom(F(M),-) $ is exact, too.

The left Kan extension is right exact since it is left adjoint to the restriction functor (precomposition by $f$) . Furthermore, the restriction functor is exact, and therefore the left Kan extension sends projectives to the corresponding projectives.} Now the claim follows from the fact that we have a projective presentation \[ \field_{\dhook{t, \infty}} \to  \field_{\dhook{s, \infty}}  \to \field_{\dhook{s, t}} \to 0.\] 
\end{proof}

\begin{proposition}\label{prop:finitely-pres_finite-resols}
Let \( \pos \) be an upper semi-lattice. \revisedversion{Then finitely presented representations, together with rank-exact sequences, form an exact category with enough projectives and injectives. The projectives and injectives in this exact category are precisely upper and lower hooks, respectively. Moreover, any object has a finite projective and a finite injective resolution.}
\end{proposition}

\begin{proof}
Let \( \Mod \) be a finitely presented \revisedversion{with respect to the standard exact structure}. Let \( \subpos \) be the join-closure of all the indices appearing in a projective presentation of \( \Mod \). Then it \revisedversion{follows from the formula given for left Kan extensions above that }
 \( \Mod = \lkan \Mod_0 \), for some representation \( \Mod_0\) of \( \subpos \).

By Theorem~\ref{thm:enough proj}, \( \Mod_0 \) has a finite rank-projective resolution. By Lemmas~\ref{lem:Lanpresreseq} and~\ref{lem:morpos}, the left Kan extension of this rank-projective resolution will be a new rank-projective resolution. \revisedversion{In particular all terms appearing in this projective resolution are sums of upper hook representations. It follows that any projective is a summand of a sum of upper hook representations, thus itself a sum of upper hook representations.}
\end{proof}

\revisedversion{Let \( \Erank^{\rm fp} \) denote the exact structure induced by rank-exact sequences on finitely presented representations.} The exact same argument as for Corollary~\ref{cor:L-shapes generate} gives the following.

\begin{corollary}
The \revisedversion{elements} \( [ \field_{\dhook{i, j}} ] \) generate \( \K(\Erank^{\rm fp}) \).
\end{corollary}

\begin{theorem} \label{thm:upper semilattice}
  Let \( \pos \) be an upper semi-lattice. Then, \revisedversion{\( \K(\Erank^{\rm fp}) \) is free}, and the set \( \{ [\field_{\dhook{i, j}}] \mid i < j \in P \cup \{ \infty \} \} \) is a basis of \( \K(\Erank^{\rm fp}) \), and furthermore, the map
\[ \Rk \colon \K(\Erank^{\rm fp}) \to \mathbb{Z}^{\RelationAsSet{\pos}} \]
is injective.
\end{theorem}

\begin{proof}
We already know that the family \( \{ [\field_{\dhook{i, j}}] \mid i < j \in P \cup \{ \infty \} \} \) generates \( \K(\Erank^{\rm fp}) \). Assume that these \revisedversion{elements} are not linearly independent. That means that there is a finite subset which is not linearly independent. However, by Theorem~\ref{thm:basis of K for finite}, for any finite set the ranks will be linearly independent.
It follows that the generating family is in fact a basis, and that \( \Rk \) is injective.
\end{proof}

\section{Application to multi-parameter persistence}
\label{sec:rank-decomp_pers}

In multi-parameter persistence, the poset~$\pos$ under consideration is either $\R^d$, viewed as a product of $d$ copies of the totally ordered real line, or a subposet of~$\R^d$---usually $\Z^d$ or some finite grid $\prod_{i=1}^d \integint{n_i}$. \revisedversion{In this setting the segments are  rectangles, i.e., products of 1-d intervals.}

\subsection{The finite grid case}
\label{sec:rank-decomp_multi-d_grid}

In this case, Corollary~\ref{cor:min_gen-rank_decomp} reformulates as follows:
\begin{corollary}\label{cor:gen_R-S_multi-param}
  Given an arbitrary collection $\Int$ of intervals in  a finite grid $\grid = \prod_{i=1}^d \libr 1, n_i\ribr \subset \R^d$, the generalized rank invariant~$\Rk_\Int \Mod$ of any pfd persistence module $\Mod$ indexed over~$\grid$  admits a unique minimal rank decomposition $(\Rec, \Sec)$  over~$\Int$.
\end{corollary}

Taking $\Int$ to be the collection of all closed rectangles in the grid~$\grid$ yields the following reformulation of Corollary~\ref{cor:min_rank_decomp}:
\begin{corollary}\label{cor:R-S_multi-param}
The usual rank invariant of any pfd persistence module $\Mod$ indexed over a finite grid $\grid = \prod_{i=1}^d \libr 1, n_i\ribr \subset \R^d$  admits a unique minimal rank decomposition $(\Rec, \Sec)$, where $\Rec$ and $\Sec$ are finite multi-sets of (closed) rectangles in $\grid$.  
\end{corollary}

\begin{figure}[b]
  \centering
  \includegraphics[width=\textwidth]{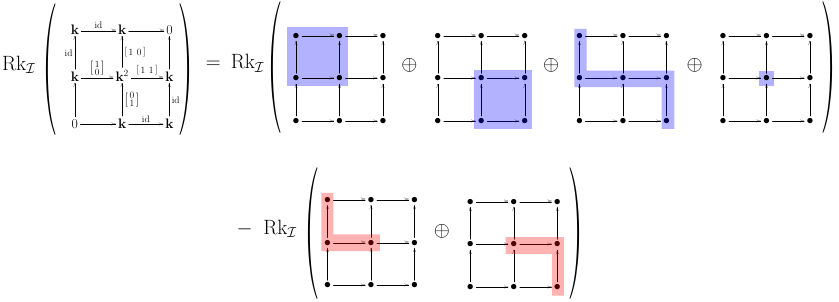}
  \caption{Minimal rank decomposition of the generalized rank invariant of the  module~$\Mod$ from Figure~\ref{fig:decomp_indec2_int_grid} over the full collection~$\Int$ of intervals in the $3\times 3$ grid. Blue is for intervals in~$\Rec$, while red is for intervals in~$\Sec$.
  }
  \label{fig:decomp_indec2_fullint_grid}
\end{figure}

\begin{figure}[tb]
  \centering
  \includegraphics[width=\textwidth]{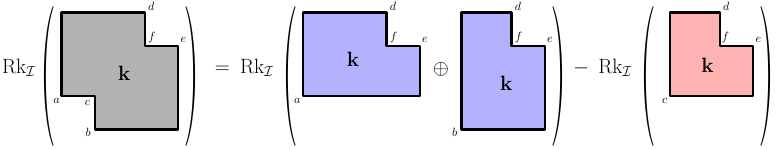}
  \caption{Taking $\Int$ to be the collection of all intervals with one generator and at most two cogenerators (which includes in particular all rectangles), the generalized rank invariant of the interval 2-parameter persistence module~$\field_\int$ on the left-hand side decomposes minimally as the difference between the generalized rank invariants of the two modules on the right-hand side. Blue is for intervals in $\Rec$ while red is for intervals in $\Sec$.} 
  \label{fig:decomp_2-2_hooks}
\end{figure}

\begin{figure}[tb]
  \centering
  \includegraphics[width=\textwidth]{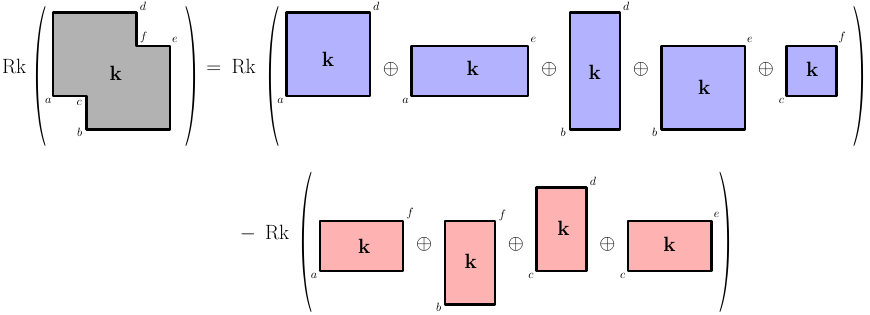}
  \caption{The usual rank invariant of the interval module~$\field_\int$ on the left-hand side decomposes minimally as the difference between the usual rank invariants of the two rectangle-decomposable modules on the right-hand side. Blue is for rectangles in $\Rec$ while red is for rectangles in $\Sec$. 
  }
  \label{fig:decomp_2-2}
\end{figure}

Figures~\ref{fig:decomp_indec2_fullint_grid} and~\ref{fig:decomp_2-2_hooks} illustrate Corollary~\ref{cor:gen_R-S_multi-param}, while Figures~\ref{fig:decomp_indec2_grid} and~\ref{fig:decomp_2-2} illustrate  Corollary~\ref{cor:R-S_multi-param}.

To  compute the minimal rank decompositions, we can apply the formula of Corollary~\ref{cor:alpha_formula}, which reformulates as follows
in the case of the usual rank invariant of a pfd persistence module~$\Mod$ indexed over a finite grid $\grid = \prod_{i=1}^d \libr 1, n_i\ribr \subset \R^d$: $\forall s\leq t\in \grid$,
\begin{equation}\label{eq:incl_excl_rect}
  \alpha_{\seg{s,t}} = \sum_{\begin{smallmatrix}s'\leq s\\\|s'-s\|_\infty \leq 1\end{smallmatrix}} \sum_{\begin{smallmatrix}t'\geq t\\\|t'-t\|_\infty \leq 1\end{smallmatrix}} (-1)^{\|s'-s\|_1 + \|t'-t\|_1}\,\Rk \Mod (s',t'). 
\end{equation}
The case $d=1$ gives the well-known inclusion-exclusion formula relating the persistence diagram of a one-parameter persistence module to its rank invariant~\cite{cohen2007stability}. The case $d=2$ gives the inclusion-exclusion formula for computing the multiplicities of summands in rectangle-decomposable 2-parameter persistence modules~\cite{botnan_et_al:LIPIcs:2020:12180}.

Figure~\ref{fig:decomp_2-2} illustrates~\eqref{eq:incl_excl_rect} on an interval module~$\field_\int$ in a finite 2-d grid. Indeed, on such modules, \eqref{eq:incl_excl_rect} yields the following expression for the minimal rank decomposition of $\Rk \field_\int$:
  \begin{align*}
    \Rec &:= \left\{\seg{g_u, c_v} \mid g_u\leq c_v\ \mbox{and}\ 2u \equiv 2v\!\!\!\! \mod 2\right\},\\
    \Sec &:= \left\{\seg{g_u, c_v} \mid g_u\leq c_v\ \mbox{and}\ 2u \not\equiv 2v\!\!\!\! \mod 2\right\},
  \end{align*}
  where $g_1, \cdots, g_k$ are the \revisedversion{minimal elements in~$\int$} (sorted by increasing  abscissae), $c_1, \cdots, c_l$ are its \revisedversion{maximal elements} (sorted likewise),  $g_{i+\frac{1}{2}} := g_i \vee g_{i+1}$ (join) for each $i\in\integint{k-1}$,  and $c_{j+\frac{1}{2}} := c_j \wedge c_{j+1}$ (meet) for each $j\in\integint{l-1}$. \revisedversion{Here, $u$ (resp. $v$) ranges over all integers and half-integers in the interval~$[1, k]$ (resp.~$[ 1, l]$).} 

\begin{remark}\label{rem:complexity_decomp_usual-rank}
    In the general case where $\grid = \prod_{i=1}^d \libr 1, n_i\ribr
    \subset \R^d$, by applying~\eqref{eq:incl_excl_rect} to every pair
    of comparable indices~$s\leq t\in\grid$ in sequence, one computes
    the minimal rank decomposition of the usual rank invariant in time
    $O\left( 2^{2d}\, \# \RelationAsSet{\grid}\right)$, assuming
    constant-time access to the ranks $\Rk \Mod (s', t')$ and
    constant-time arithmetic operations\footnote{We are considering an
      implementation that iterates over the indices $s', t'$ such that
      $\|s'-s\|_\infty\leq 1$ and $\|t'-t\|_\infty\leq 1$ by
      increasing order of the $1$-norms $\|s'-s\|_1$ and $\|t'-t\|_1$,
      so that the $1$-norms do not have to be re-computed from scratch
      at every step. Such an implementation boils down to iterating
      over the vertices of the unit hypercube in~$\R^d$ by increasing
      order of the number of $1$'s in their coordinates.}. This bound
    is in~$O\left( 2^{2d}\, \prod_{i=1}^d n_i^2\right)$, and when $d$
    is fixed, it is linear in the size of the \revisedversion{domain} of the usual
    rank invariant \revisedversion{viewed} as a map $\RelationAsSet{\grid}\to\Z$. When the
    module~$\Mod$ comes from a simplicial filtration over the
    grid~$\grid$ with $n = \max_{i} n_i$ simplices in total, the usual
    rank invariant itself can be pre-computed and stored, e.g. by
    naively computing the ranks $\Rk \Mod(s,t)$ for each pair $s\leq
    t\in\grid$ independently, which takes $O(n^{2d+\omega})$ time in
    total, where $2\leq \omega<2.373$ is the exponent for matrix
    multiplication~\cite{milosavljevic2011zigzag}. Adding in the
    computation time for the minimal rank decomposition yields a bound
    in~$O(n^{2d+\omega} + (2n)^{2d})$. While naive, this approach
    already compares favorably, in terms of running time, to the
    computation of other (stronger) invariants such as for instance
    the direct-sum decomposition of~$\Mod$---for which the best known
    algorithm runs in $O(n^{d(2\omega+1)})$
    time~\cite{dey2019generalized}. Moreover, the running time of
    computing the minimal rank decomposition is dominated by the
    running time of pre-computing the usual rank invariant, for which
    there is room for improvement. In the special case where $d=2$ for
    instance, assuming the filtration is $1$-critical (i.e., each
    simplex has a unique minimal time of appearance in the
    filtration), there is an $O(n^4)$-time algorithm to compute the
    usual rank invariant~\cite{botnan_et_al:LIPIcs:2020:12180}, and
    computing its minimal rank decomposition also takes $O(n^4)$
    time. By comparison, the best known algorithm to compute the
    direct-sum decomposition of~$\Mod$ in this setting takes
    $O(n^{2\omega+1})$ time~\cite{dey2019generalized}.
\end{remark}

\subsection{The \( \mathbb{R}^d \) case}

\begin{figure}[tb]
  \centering
  \includegraphics[width=\textwidth]{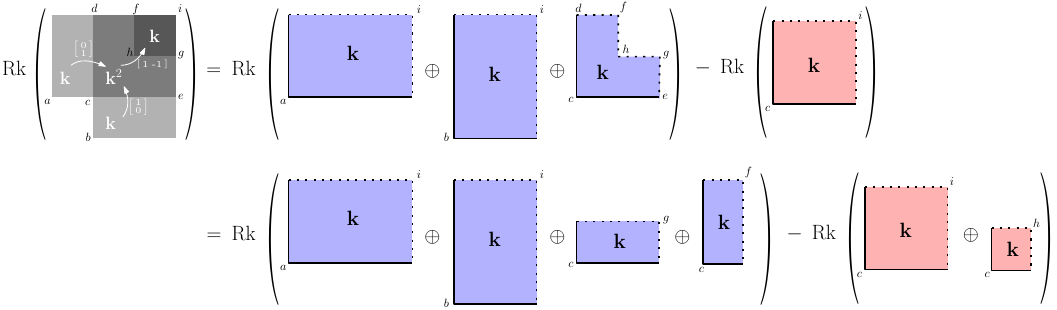}
  \caption{Left: an indecomposable persistence module with $2$ generators, one at $a$ and one at $b$, and a relation equating them at $h$ (indices $d, e, f, g, i$ lie at infinity).
    Right: 
minimal rank decompositions of the usual rank invariant of the module,  respectively over the lower hooks (top row) and over the half-open rectangles (bottom row). Blue is for intervals in $\Rec$ while red is for intervals in $\Sec$. Solid  boundaries belong to the intervals, while dotted boundaries do not.} 
  \label{fig:decomp_indec}
\end{figure}

\revisedversion{
Since $\R^d$ is a lattice and we are interested in finitely presented modules, the rectangles used for the decomposition of the usual rank invariant are half-open as in Section~\ref{sec:restrictions_usual}, even though the usual rank invariant itself is evaluated on closed rectangles\footnote{This  infringes the requirement from Definition~\ref{def:rank_decomp} that the intervals used in the decomposition be taken from within the collection of segments. However, as mentioned previously, this requirement is meant only to ensure the uniqueness of the minimal decomposition, and can therefore be dropped here since uniqueness is otherwise ensured.}. Corollary~\ref{cor:decomp_half-open} reformulates as follows: 
\begin{corollary}\label{cor:rank-decomp_multipers_rect}
  The usual rank invariant of every finitely presented persistence module~$\Mod$ over $\R^d$ admits a unique minimal rank decomposition over the half-open rectangles in~$\R^d$.
\end{corollary}
Meanwhile,
Theorem~\ref{thm:upper semilattice} implies the following: 
\begin{corollary}\label{cor:rank-decomp_multipers_hook}
  The usual rank invariant of every finitely presented persistence module~$\Mod$ over $\R^d$ admits a unique minimal rank decomposition over the lower hooks in~$\R^d$.
\end{corollary}
These two results are illustrated in Figure~\ref{fig:decomp_indec}. The first one can also be derived from the second one, by further decomposing the rank invariants of hook modules over the half-open rectangles. The decomposition can in fact be done in the Grothendieck group~$\K(\Erank^{\rm fp})$ directly, via rank-exact resolutions of hook modules by half-open rectangle modules, as described next. 

}

\begin{lemma} \label{lem:Koszul-type}  
Assume \( \int = \bigcup_{i=1}^n \int_i \), with  each $\int_i$ being a  down-set in the poset \( \int \).   Then, the sequence
\[ 0 \to \field_{\int} \to \bigoplus_{i=1}^n \field_{\int_i} \to \bigoplus_{i < j} \field_{\int_i \cap \int_j} \to \bigoplus_{i < j < \ell} \field_{\int_i \cap \int_j \cap \int_\ell} \to \cdots \to \field_{\bigcap_i I_i} \to 0 \]
with component maps 
\[ \field_{\cap_{i \in X} \int_i} \to \field_{\cap_{i \in Y} \int_i} \colon \begin{cases} (-1)^{\card\{x \in X \mid x < y\}}\, \text{projection} & \text{if } Y = X \cup \{y\} \\ 0 & \text{otherwise} \end{cases} \]
is rank exact. Note that the assumption that each $\int_i$ is a down-set in $\int$ guarantees the projections in the component maps to be well-defined.
\end{lemma}

\begin{proof}
We employ induction on \( n \). Let \( \int' = \bigcup_{i=2}^n \int_i \). The short exact sequence
\[ 0 \to \field_{\int} \to \field_{\int_1} \oplus \field_{\int'} \to \field_{\int_1 \cap \int'} \to 0 \]
is rank exact. Inductively, we have rank exact sequences for \( \int' \) and for \( \int_1 \cap \int' = \bigcup_{i=2}^n \int_1 \cap \int_i \) as in the two rows of the following diagram.
\[ 
\adjustbox{scale=0.9,center}{\begin{tikzcd}[sep=6mm]
0 \ar[r] & \field_{\int'} \ar[r] \ar[d] & \bigoplus_{1 < i} \field_{\int_i} \ar[r] \ar[d] & \bigoplus_{1 < i < j} \field_{\int_i \cap \int_j} \ar[r] \ar[d] & \cdots \ar[r] & \field_{\cap_{i>1} \int_i} \ar[r] \ar[d] & 0 \\
0 \ar[r] & \field_{\int_1 \cap \int'} \ar[r] & \bigoplus_i \field_{\int_1 \cap \int_i} \ar[r] & \bigoplus_{1 < i < j} \field_{\int_1 \cap \int_i \cap \int_j} \ar[r] & \cdots \ar[r] & \field_{\cap_i \int_i} \ar[r] & 0
\end{tikzcd}} \]
We can add the vertical arrows and check that the entire diagram commutes (for appropriate sign conventions). In particular its total complex is rank exact again. Putting this total complex in the upper row, we obtain the following commutative diagram.
\[ 
\adjustbox{scale=0.9,center}{\begin{tikzcd}[sep=5mm]
0 \ar[r] & 0 \ar[r] \ar[d] & \field_{\int'} \ar[r] \ar[d] & \bigoplus_{1 < i} \field_{\int_i} \oplus \field_{\int_1 \cap \int'} \ar[r] \ar[d] & \bigoplus_{i < j} \field_{\int_i \cap \int_j} \ar[r] \ar[d] & \cdots \ar[r] & \field_{\cap_{1} \int_i} \ar[r] \ar[d] & 0 \\
0 \ar[r] & \field_{\int} \ar[r] & \field_{\int_1} \oplus \field_{\int'} \ar[r] & \field_{\int_1 \cap \int'} \ar[r] & 0 \ar[r] & \cdots \ar[r] & 0 \ar[r] & 0 
\end{tikzcd}} \]
Choosing the vertical arrows to be inclusion of and projection to the corresponding summand we obtain a new commutative diagram. Again we form the total complex, and observe that up to homotopy the two copies of \( \field_{\int'} \) and the two copies of \( \field_{\int_1 \cap \int'} \) cancel. Thus we get the rank exact sequence in the statement of the lemma.
\end{proof}

\revisedversion{
  From Lemma~\ref{lem:Koszul-type} and  Theorem~\ref{thm:upper semilattice} we derive the following categorical enhancement of Corollary~\ref{cor:rank-decomp_multipers_rect}:
}
\begin{theorem} \label{thm:basis_for_Rn}
For the poset \( \mathbb{R}^d \), the set
\[ \{ \field_{[a_1, b_1) \times \cdots \times [a_d, b_d)}] \mid a_i < b_i \in \mathbb{R} \cup \{ \infty \} \} \]
is a basis of \( \K(\Erank^{\rm fp}) \).
\end{theorem}

\begin{proof}
Note that
\[ \dhook{ \mathbf{a}, \mathbf{b} } = \bigcup_{i = 1}^d [a_1, \infty) \times \cdots \times [a_{i-1}, \infty) \times [a_i, b_i) \times [a_{i+1}, \infty) \times \cdots \times [a_d, \infty), \]
and all the sets on the right-hand side are down-sets in \( \dhook{ \mathbf{a}, \mathbf{b} } \). Thus, Lemma~\ref{lem:Koszul-type} applies. In particular, any \( [ \field_{\dhook{ \mathbf{a}, \mathbf{b} }} ] \) is a linear combination of vectors in our candidate basis, which therefore generates \( \K(\Erank^{\rm fp}) \) by  Theorem~\ref{thm:upper semilattice}.

The linear independence follows from Proposition~\ref{prop:complete}---alternatively, linear independence can also be verified by carefully analyzing the base change from the basis of Theorem~\ref{thm:upper semilattice} to the candidate basis here.
\end{proof}

\subsection{Restrictions to lines}
\label{sec:rank-decomp_multi-d_to_1-d}

\revisedversion{
From Example~\ref{ex upward line} we get the following:
\begin{corollary}\label{cor:R-S_restrict-line}
Let $\Mod$ be a pfd persistence module over $\R^d$ such that the usual rank invariant $\Rk \Mod$ admits a rank decomposition $(\Rec, \Sec)$ over half-open  (resp. open, closed, all) rectangles. Then, for any upward sloping line $\line$  in $\R^d$, $(\Rec, \Sec)$ restricts to a rank decomposition $({\Rec}_{|\line}, {\Sec}_{|\line})$ of  $\Rk {\Mod}_{|\line}$ over half-open (resp. open, closed, all) intervals.
\end{corollary}

}


\begin{figure}[tb]
  \centering
  \includegraphics[width=\textwidth]{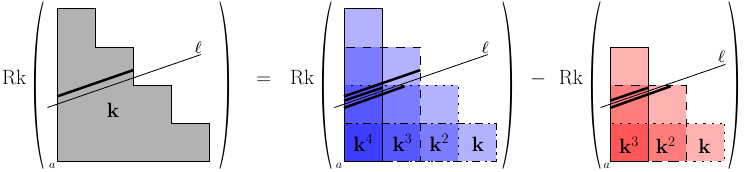}
  \caption{Restricting an interval module $\field_\int$ to an upward sloping line $\ell$ (left) yields a restriction of the minimal rank decomposition of $\Rk \field_\int$ to $\ell$ (right)---for clarity, the rectangles' boundaries are shown with different line styles. Here, the restricted rank decomposition is not minimal, as the two interval summands of $\field_{\Sec|_\ell}$ cancel out with two of the three interval summands of $\field_{\Rec|_\ell}$. } 
  \label{fig:line_restriction}
\end{figure}

Note that the restriction of a minimal decomposition may not be minimal, as different rectangles in $\Rec$ and $\Sec$ may restrict to the same 1-d interval---see Figure~\ref{fig:line_restriction} for an illustration. However, by Corollary~\ref{cor:min-decomp_exists_unique}, the minimal rank decomposition $(\Rec^*, \Sec^*)$ of $\Rk \Mod_{|\line}$ is easily obtained by removing all the common elements in ${\Rec}_{|\line}$ and ${\Sec}_{|\line}$. Furthermore, as illustrated in Figure~\ref{fig:line_restriction} and formalized in the following result,  $(\Rec^*, \Sec^*)$ actually coincides with the persistence barcode of the one-parameter module $\Mod|_\ell$.

\begin{corollary}\label{cor:R-S_1param}
\revisedversion{For every} pfd persistence module $\Mod$ indexed over the real line, \revisedversion{the usual rank invariant~$\Rk \Mod$} admits a unique minimal rank decomposition $(\Rec, \Sec)$ \revisedversion{over all intervals of~$\R$}, given by $\Rec=\dgm \Mod$, the persistence barcode of~$\Mod$, and $\Sec=\emptyset$.
\end{corollary}
\begin{proof}
\revisedversion{Call $\Int$ the collection of all closed intervals of~$\R$, and note that $\widehat{\Int}$ as defined in~(\ref{eq:hatInt}) is then the collection of all intervals of~$\R$. The structure theorem for pfd one-parameter persistence modules~\cite{Crawley-Boevey2012}, combined with the fact that the usual rank invariant is additive on finite direct sums, implies that $(\dgm \Mod, \emptyset)$ is a rank decomposition of~$\Rk \Mod=\Rk_\Int \Mod$ over~$\widehat{\Int}$. This decomposition is minimal and Theorem~\ref{thm:uniqueness} implies its uniqueness. }
\end{proof}

\subsection{Stability}
\label{sec:stability}

We conclude this section by saying a few words about the stability of our rank decompositions.
Recall from Corollary~\ref{cor:rank_decomp_vs_rank_decomp-first} that we have $\field_{\Rec} \oplus \field_{\Sec'} \simeq \field_{\Rec'} \oplus \field_{\Sec}$ for any two rank decompositions $(\Rec, \Sec)$ and $(\Rec', \Sec')$ of the same persistence module~$\Mod$, or of two persistence modules~$\Mod, \Mod'$ sharing the same (usual) rank invariant. In effect, this is telling us that two rank decompositions are equivalent whenever their ground modules have the same rank invariant. Using the matching (pseudo-)distance~$\distm$ from~\cite{landi2018rank}, we can derive a metric version of this statement (Theorem~\ref{th:matching-distance_stability}), which bounds the defect of equivalence between  two rank decompositions in terms of the fibered distance between  the rank invariants of their ground modules. Recall that the matching distance between two pfd persistence modules $\Mod, \Nod$ in~$\R^d$ is defined as follows:
\begin{equation}\label{eq:dmatch}
\distm(\Mod, \Nod) = \sup_{\line\ \text{upward sloping}}\ \omega(\line)\ \distb(\Mod|_\line, \Nod|_\line),
\end{equation}
where $\distb$ denotes the usual bottleneck distance between one-parameter persistence modules, and where the weight $\omega(\line)$ assigned to any upward sloping  line~$\line$ parametrized by $\lambda \mapsto (1-\lambda)s +\lambda t$ with $\lambda$ ranging over $\R$ while $s,t$ are fixed in~$\R^d$ and satisfy $s_i<t_i$ for each $i=1, \cdots, d$, is
\[
\omega(\line) = \frac{\min_i t_i-s_i}{\max_i t_i-s_i}>0.
\]
%
%
\begin{theorem}\label{th:matching-distance_stability}
  Let $\Mod, \Mod'$ be pfd persistence modules indexed over $\R^d$. Then, for any rank decompositions $(\Rec, \Sec)$ and $(\Rec', \Sec')$ of $\Rk \Mod$ and $\Rk \Mod'$ respectively \revisedversion{over all rectangles in~$\R^d$}, we have:
  \[
  \distm(\field_{\Rec} \oplus \field_{\Sec'}, \field_{\Rec'} \oplus \field_{\Sec}) \leq \distm(\Mod, \Mod').
  \]
\end{theorem}
\begin{proof}
  Take any upward sloping line $\line$ in $\R^d$.
By~\eqref{eq:dmatch},
  we have:
  \[
  \distb(\Mod_{|\line}, \Mod'_{|\line}) \leq \omega(\line)^{-1}\, \distm(\Mod, \Mod').
  \]
Meanwhile, by Corollary~\ref{cor:R-S_restrict-line}, $(\Rec_{|\line}, \Sec_{|\line})$ is a rank decomposition of $\Rk \Mod_{|\line}$, and $(\Rec'_{|\line}, \Sec'_{|\line})$ is a rank decomposition of $\Rk \Mod'_{|\line}$ \revisedversion{over all intervals of~$\R$}. By Proposition~\ref{prop:complete} and the structure theorem for pfd one-parameter persistence modules~\cite{Crawley-Boevey2012}, we then have $\Mod_{|\line} \oplus \field_{\Sec_{|\line}} \simeq \field_{\Rec_{|\line}}$ and $\Mod'_{|\line} \oplus \field_{\Sec'_{|\line}} \simeq \field_{\Rec'_{|\line}}$, from which we deduce:  
  \[
  \distb(\Mod_{|\line}, \Mod'_{|\line}) \geq \distb(\Mod_{|\line} \oplus \field_{\Sec_{|\line}} \oplus \field_{\Sec'_{|\line}}, \Mod'_{|\line} \oplus \field_{\Sec_{|\line}} \oplus \field_{\Sec'_{|\line}}) = \distb(\field_{\Rec_{|\line}} \oplus \field_{\Sec'_{|\line}}, \field_{\Rec'_{|\line}} \oplus \field_{\Sec_{|\line}}).
  \]
  %
  Combined with the previous equation, this gives:
  \[
  \distb(\field_{\Rec_{|\line}} \oplus \field_{\Sec'_{|\line}}, \field_{\Rec'_{|\line}} \oplus \field_{\Sec_{|\line}}) \leq  \omega(\line)^{-1}\, \distm(\Mod, \Mod').
  \]
  The result follows then by taking the supremum on the left-hand side over all possible choices of upward sloping lines $\line$.
\end{proof}

It is worth pointing out that different choices of rank decompositions $(\Rec, \Sec)$ and $(\Rec', \Sec')$ for $\Mod$ and $\Mod'$ may yield different values for the matching distance $\distm(\field_{\Rec} \oplus \field_{\Sec'}, \field_{\Rec'} \oplus \field_{\Sec})$.  It turns out that the rank decompositions that maximize this distance are precisely the minimal rank decompositions, which therefore satisfy a universal property also in terms of the metric between decompositions:
\begin{proposition}\label{prop:min-decomp_max-dist}
  Let $\Mod, \Mod'$ be pfd persistence modules indexed over $\R^d$. Then, for any rank decompositions $(\Rec, \Sec)$ and $(\Rec', \Sec')$ of $\Rk \Mod$ and $\Rk \Mod'$ respectively \revisedversion{over all rectangles in~$\R^d$}, we have:
  \[
  \distm(\field_{\Rec} \oplus \field_{\Sec'}, \field_{\Rec'} \oplus \field_{\Sec}) \leq \distm(\field_{\Rec^*} \oplus \field_{\Sec'^*}, \field_{\Rec'^*} \oplus \field_{\Sec^*}),
  \]
where $(\Rec^*, \Sec^*)$ and $(\Rec'^*, \Sec'^*)$ are the minimal rank decompositions of $\Rk \Mod$ and $\Rk \Mod'$ respectively \revisedversion{over all rectangles in~$\R^d$}---which exist as soon as $(\Rec, \Sec)$ and $(\Rec', \Sec')$ do, by Theorem~\ref{thm:uniqueness}. 
\end{proposition}
\begin{proof}
  Let $\Tec := \Rec \setminus \Rec^* = \Sec \setminus \Sec^*$, and $\Tec' := \Rec' \setminus \Rec'^* = \Sec' \setminus \Sec'^*$. Note that $\Tec, \Tec'$ are well-defined by Theorem~\ref{thm:uniqueness}. Then, for any upward sloping line~$\line$, we have:
  \begin{align*}
  \distb(\field_{\Rec_{|\line}} \oplus \field_{\Sec'_{|\line}},\ \field_{\Rec'_{|\line}} \oplus \field_{\Sec_{|\line}})\\[0.5em]
  &\hspace{-2.5cm}=\  \distb(\field_{{\Rec^*}_{|\line}} \oplus \field_{{\Sec'^*}_{|\line}} \oplus \field_{\Tec_{|\line}} \oplus \field_{\Tec'_{|\line}},\ \field_{{\Rec'^*}_{|\line}} \oplus \field_{{\Sec^*}_{|\line}} \oplus \field_{\Tec_{|\line}} \oplus \field_{\Tec'_{|\line}})\\[0.5em]
  &\hspace{-2.5cm}\leq\ \distb(\field_{{\Rec^*}_{|\line}} \oplus \field_{{\Sec'^*}_{|\line}},\ \field_{{\Rec'^*}_{|\line}} \oplus \field_{{\Sec^*}_{|\line}}).
  \end{align*}
  The result follows then after multiplying by~$\omega(\line)$ and  taking the supremum on both sides over all possible choices of upward sloping lines~$\line$.  
\end{proof}

Finally, we can also bound the defect of equivalence between  two rank decompositions in terms of the defect of isomorphism between their ground modules---measured by the interleaving distance~$\disti$. This is a straight consequence of our Theorem~\ref{th:matching-distance_stability} and of Theorem~1 from~\cite{landi2018rank}:
\begin{corollary}\label{cor:matching-distance_stability}
  Let $\Mod, \Mod'$ be pfd persistence modules indexed over $\R^d$. Then, for any rank decompositions $(\Rec, \Sec)$ and $(\Rec', \Sec')$ of $\Rk \Mod$ and $\Rk \Mod'$ respectively \revisedversion{over all rectangles in~$\R^d$}, we have:
  \[
  \distm(\field_{\Rec} \oplus \field_{\Sec'}, \field_{\Rec'} \oplus \field_{\Sec}) \leq \disti(\Mod, \Mod').
  \]
\end{corollary}

\section{Signed barcodes and prominence diagrams for multi-parameter persistence modules}
\label{sec:signed_barcodes}

In the context of topological data analysis, the minimal rank decomposition $(\Rec, \Sec)$ of the
(usual or generalized)
rank invariant of a persistence module $\Mod \colon \R^d \to \Vec_{\field}$ encodes its structure visually.
In the particular case of the usual rank invariant, we saw in the  examples of Section~\ref{sec:rank-decomp_pers} that we get direct access to the following pieces of information:
\begin{itemize}
\item the rank $\Rk \Mod (s,t)$ between any pair of indices $s\leq t\in\R^d$, obtained as the number of rectangles in~$\Rec$ that contain both $s$ and $t$ minus the number of rectangles in~$\Sec$ that contain both $s$ and $t$;
  \item the barcode of the restriction of~$\Mod$ to any upward sloping line~$\ell$, obtained by simplifying the restriction of $(\Rec, \Sec)$ to $\ell$, each bar of which comes from the intersection of a rectangle in $\Rec$ or $\Sec$ with~$\ell$. 
\end{itemize}
The main drawback of representing rectangles as rectangles is that their overlaid arrangement quickly becomes hard to read---see e.g. Figure~\ref{fig:line_restriction}.

\subsection{Signed barcodes}

An alternate representation of the rectangles is by their diagonal with positive slope in $\R^d$. We call this representation the {\em signed barcode} of~$\Rk \Mod$, where each bar is the diagonal (with positive slope) of a particular rectangle in $\Rec$ or $\Sec$, and where the sign is positive for bars coming from $\Rec$ and negative for bars coming from~$\Sec$---see Figure~\ref{fig:straight-line_barcodes} for an illustration. Like the rectangles, the bars are considered with multiplicity.

\begin{figure}[tb]
  \centering
  \includegraphics[width=\textwidth]{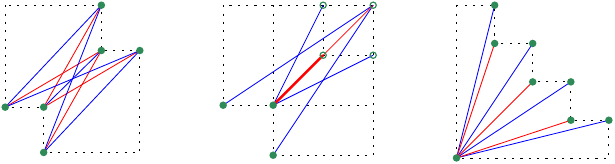}
  \caption{From left to right: signed barcodes corresponding to the usual rank decompositions of Figures~\ref{fig:decomp_2-2}, \ref{fig:decomp_indec}~(bottom row), and~\ref{fig:line_restriction} respectively. Blue bars are diagonals of rectangles in $\Rec$ and therefore counted positively, while red bars are diagonals of rectangles in $\Sec$ and therefore counted negatively. The bars' endpoints are marked in green (as a solid dot when the endpoint lies in the rectangle, as a circled dot when it does not---e.g. when it lies at infinity), to discriminate them from intersections. The thick red line segment in the center picture shows the overlap between a shorter red bar and a longer red bar sharing the same lower endpoint and slope. 
    }
  \label{fig:straight-line_barcodes}
\end{figure}

The signed barcode of $\Rk \Mod$ gives direct access to the same pieces of information as the rectangular representation---see Figure~\ref{fig:barcodes_interpretation}:
\begin{itemize}
\item the rank $\Rk \Mod (s, t)$ between any pair of indices $s\leq t\in\R^d$ is obtained as the number of positive bars that connect the down-set $s^-=\{u\in \R^d \mid u\leq s\}$ to the up-set $t^+=\{u\in \R^d \mid u\geq t\}$, minus the number of negative bars that connect $s^-$ to $t^+$---exactly as with persistence barcodes in the one-parameter case (see Figure~\ref{fig:decomp_1d}), except bars are now signed;
\item the barcode of the restriction of~$\Mod$ to any upward sloping line~$\ell$ is obtained by simplifying the restriction of $(\Rec, \Sec)$ to $\ell$, each bar of which comes from the projection of a bar $(s,t)$ in the signed barcode onto~$\ell$ according to the following rule---coming from the intersection of~$\ell$ with the rectangle $\seg{s, t}$: project $s$ onto the point $s'=\ell \cap \partial s^+$, and $t$ onto $t'=\ell \cap \partial t^-$, if these two points exist and satisfy $s'\leq t'$ (otherwise the projection is empty).
\end{itemize}
Beyond these features, the signed barcode makes it possible to visually grasp the global structure of the usual rank invariant $\Rk \Mod$, and in particular, to infer the directions along which topological features have the best chances to persist.

\begin{figure}[tb]
  \centering
  \includegraphics[width=0.8\textwidth]{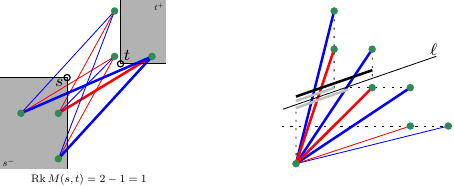}
  \caption{Left: computing $\Rk \Mod(s,t)$ for a pair of indices $s\leq t$ (the thick bars are the ones connecting the down-set $s^-$ to the up-set $t^+$). Right: restricting the minimal rank decomposition of $\Rk \Mod$ to an upward sloping line~$\ell$---the thick blue and red bars are the ones projecting to non-empty bars along~$\ell$, and among those projections, the thick gray bars get cancelled out during the simplification while the thick black bar remains in the barcode of~$\Mod|_\ell$.}
  \label{fig:barcodes_interpretation}
\end{figure}

\subsection{Signed prominence diagrams}
\label{sec:signed_pers_diags}

To each bar with endpoints $s\leq t$ in the usual signed barcode of a module~$\Mod: \R^d\to\Vec_{\field}$, we can associate its {\em signed prominence}, which is the $d$-dimensional vector $t-s$ if the bar corresponds to a rectangle in~$\Rec$, or $s-t$ if the bar corresponds to a rectangle in~$\Sec$. We call {\em signed prominence diagram} of~$\Mod$ the resulting collection of vectors in~$\R^d$---see Figure~\ref{fig:signed_prominences} for an illustration.

\begin{figure}[tb]
  \centering
  \includegraphics[width=0.8\textwidth]{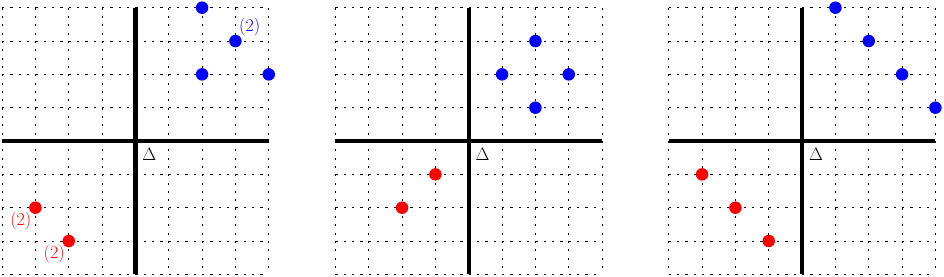}
  \caption{The signed prominence diagrams corresponding to the signed barcodes of Figure~\ref{fig:straight-line_barcodes}, in the same order. Blue dots correspond to blue bars (hence to rectangles in~$\Rec$), while red dots correspond to red bars (hence to rectangles in~$\Sec$). Multiplicities differring from~$1$ are indicated explicitly.
The union~$\Delta$ of the two coordinate axes plays the role of the diagonal, as explained below.}
  \label{fig:signed_prominences}
\end{figure}

\begin{example}
  In the one-parameter setting, the signed prominence diagram encodes the lengths of the bars in the unsigned barcode, i.e., the vertical distances of the points to the diagonal in the persistence diagram. It is a discrete measure on the real line.
\end{example}

In a signed prominence diagram, the union~$\Delta$ of the hyperplanes perpendicular to the coordinate axes and passing through the origin plays the role of the diagonal: a bar whose signed prominence lies close to~$\Delta$ can be viewed as noise, whereas a bar whose signed prominence lies far away from~$\Delta$ can be considered significant for the structure of the module~$\Mod$. The right way to formalize this intuition is via smoothings, as in the one-parameter case.

\begin{definition}\label{def:e-shift-smoothing}
  Given a persistence module $\Mod$ indexed over~$\R^d$, and a vector $\e\in \R_{\geq 0}^d$, the {\em $\e$-shift}~$\Mod[\e]$ is the module defined pointwise by $\Mod[\e](t) = \Mod(t+\e)$ and $\Mod[\e](s\leq t) = \Mod(s+\e \leq t+\e)$. There is a canonical morphism of persistence modules $\Mod\to\Mod[\e]$, whose image is called the {\em $\e$-smoothing} of~$\Mod$, denoted by~$\Mod^\e$.
\end{definition}

\begin{example}\label{ex:e-smoothing_rect}
  The $\e$-shift of an arbitrary rectangle module $\field_{\rec}$ is the rectangle module $\field_{\rec-\e}$, where by definition $\rec-\e = \{t-\e \mid t\in\rec\}$. The $\e$-smoothing of $\field_{\rec}$ is the rectangle module $\field_{\rec^\e}$, where by definition $\rec^\e$ is the rectangle $\rec\cap(\rec-\e)$, obtained from~$\rec$ by shifting the upper-right corner of~$\rec$ by~-$\e$. Note that $\field_{\rec^\e}$ is the trivial module whenever $\rec\cap (\rec-\e)= \emptyset$.
\end{example}

As it turns out, usual rank decompositions commute with smoothings, more precisely:

\begin{lemma}\label{lem:e-smoothed_decomp}
Suppose \revisedversion{the usual rank invariant of}~$\Mod: \R^d\to \Vec_{\field}$ admits a rank decomposition $(\Rec, \Sec)$ \revisedversion{over arbitrary rectangles in~$\R^d$}. Then, for any $\e\in\R_{\geq 0}^d$, the pair $(\Rec^\e, \Sec^\e)$ where $\Rec^\e = \{\rec^\e \mid \rec\in\Rec\}$ and $\Sec^\e = \{\sec^\e \mid \sec\in\Sec\}$ \revisedversion{decomposes the  usual rank invariant} of $\Mod^\e$. When $(\Rec, \Sec)$ is minimal, so is $(\Rec^\e, \Sec^\e)$ after removing the empty rectangles from~$\Rec^\e$ and~$\Sec^\e$. 
\end{lemma}
\begin{proof}
  For any indices $s\leq t\in\R^d$, the commutativity of the square
  \[\xymatrix{
    \Mod(s) \ar[r]\ar[d] & \Mod(t) \ar[d] \\
    \Mod(s+\e) \ar[r] & \Mod(t+\e)
  }\]
  implies that $\Rk \Mod^\e (s, t) = \Rk \Mod(s, t+\e)$. Then, the usual rank decomposition of~$\Mod$ at indices~$(s, t+\e)$ gives:
  \begin{align*}
    \Rk \Mod^\e(s, t) = \Rk \Mod(s, t+\e) &= \sum_{\rec \in \Rec} \Rk \field_\rec(s, t+\e) - \sum_{\sec \in \Sec} \Rk \field_\sec(s, t+\e)\\
    &= \sum_{\rec \in \Rec} \Rk \field_\rec^\e(s, t) - \sum_{\sec \in \Sec} \Rk \field_\sec^\e(s, t)\\
    \mbox{\footnotesize (Example~\ref{ex:e-smoothing_rect})}~ &= \sum_{\rec \in \Rec} \Rk \field_{\rec^\e}(s, t) - \sum_{\sec \in \Sec} \Rk \field_{\sec^\e}(s, t).
  \end{align*}
When $(\Rec, \Sec)$ is minimal, the minimality of~$(\Rec^\e, \Sec^\e)$ after removing the empty rectangles comes from the fact that each $\rec^\e$ is obtained from $\rec$ by shifting its upper-right corner by -$\e$, so $\rec^\e = \sec^\e \neq \emptyset$ implies $\rec = \sec$.
\end{proof}

\begin{figure}[b]
  \centering
  \includegraphics[width=0.9\textwidth]{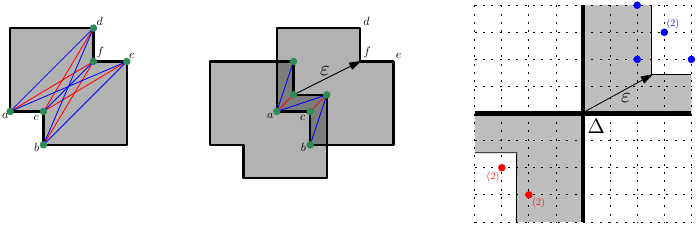}
  \caption{Behavior of the signed barcode and prominence diagram under $\e$-smoothing. Left: the input module~$\Mod$ from Figure~\ref{fig:decomp_2-2}, overlaid with its signed barcode. Center: the $\e$-smoothing~$\Mod^\e$ of~$\Mod$, shown in dark gray and overlaid with its own signed barcode---obtained by shifting the right endpoints in the signed barcode of~$\Mod$ by~-$\e$. Right: effect of the ~$\e$-smoothing on the signed prominence diagram.}
  \label{fig:smoothed_barcode_diagram}
\end{figure}

Thus, the effect of $\e$-smoothing~$\Mod$ on its signed barcode is to shift the right endpoints of the bars by~-$\e$, removing those bars for which the shifted right endpoint is no longer greater than or equal to the left endpoint. The effect on its signed prominence diagram is to shift the positive vectors by~-$\e$ and the negative vectors by~$\e$, removing those vectors that cross~$\Delta$. Alternatively, one can inflate~$\Delta$ by~$\e$, and remove the vectors that lie in the inflated~$\Delta$, as illustrated in Figure~\ref{fig:smoothed_barcode_diagram}.

So, the remoteness of the signed prominence of a bar from~$\Delta$, or equivalently, the width of the rectangle corresponding to this bar in the usual rank decomposition of~$\Mod$,  measures how resilient that bar or rectangle is under smoothings of~$\Mod$, and thus how important it is for the structure of the usual rank invariant.

\section{Experiments}
\label{sec:experiments}

In this section we consider the signed barcodes in three different two-parameter settings. In the first experiment (Section~\ref{sec:experiment1}), our point set has three distinct holes, and we will see how these holes can be inferred from the signed barcode. In the second experiment (Section~\ref{sec:experiment2}), we explore the stability properties of the signed barcodes by considering a family of point sets sampled from the unit circle with an increasing amount of noise. In the final experiment (Section~\ref{sec:experiment3}), we apply our methods in the context of two-parameter clustering. Before getting to the details of the experiments, we recall some helpful basic concepts from topological data analysis in Section~\ref{sec:general_constructions}.

\begin{remark}
  Since we are working with two-parameter persistence modules, we have chosen to rely entirely on RIVET~\cite{lesnick2015interactive} to compute the usual rank invariant, mainly for the sake of simplicity and speed of implementation. As pointed out in Remark~\ref{rem:complexity_decomp_usual-rank} and further discussed in Section~\ref{sec:conclusion}, the use of RIVET---or part thereof---is not mandatory but can be considered as an option. 
\end{remark}
  
\subsection{General constructions}
\label{sec:general_constructions}

Let $(P,d)$ denote a finite metric space. The \emph{neighborhood graph at scale $r$} is the graph $N(P)_r$ with vertex set $P$, and with an edge connecting $p$ and $q$ if $d(p,q) \leq r$. The \emph{Vietoris--Rips complex at scale $r$}, ${\rm VR}(P)_r$, is defined as the \emph{clique complex} on $N(P)_r$, i.e., the largest simplicial complex having $N(P)_r$ as its 1-skeleton.  The Vietoris--Rips complex is an important tool in single-parameter persistent homology and it can be further refined in the presence of a real valued function $f\colon P\to \R$: the \emph{Vietoris--Rips bifiltration of ${\rm VR}(P)_\infty$ (with respect to $f$)} is given by 
\[{\rm VR}(P,f)_{r,s} := {\rm VR}(f^{-1}(-\infty, s])_r.\] Applying simplicial homology with coefficients in a field~$\field$ yields a persistence module $M$ over~$\R^2$, \[ M(r,s) = H_p({\rm VR}(P,f)_{r,s} ).\]

When considering points in the plane it is convenient to visualize the Vietoris--Rips bifiltration by the bifiltration of the plane given by $f$ and the distance function, i.e., 
\[ U^f_{r,s} = \left\{z\in \R^2: \min_{p\in P, f(p) \leq s} ||p-z|| \leq r/2.\right\}\]
As is well-known, $H_p(U^f)$ offers only an approximation of $H_p({\rm VR}(P,f))$, and thus there might be slight homological discrepancies between the planar subsets as visualized, and the Vietoris--Rips bifiltration. 

In practice we will consider a discretization of $M$. This is done by first selecting a finite number of thresholds $X=\{r_0, \ldots, r_{g_1-1}\}$ and $Y=\{s_0 \ldots, s_{g_2-1}\}$ for $r$ and $s$, respectively, and then restricting $M$ to the grid $X\times Y$. In all the plots we will identify $X\times Y$ with the grid $\{0,1, \ldots, g_1 - 1\}\times \{0, 1, \ldots, g_2-1\}$. From Example~\ref{example restriction to grid} we know that, if $(\Rec,\Sec)$ is a rank decomposition of $M$, then $(\Rec|_{X\times Y}, \Sec|_{X\times Y})$ is a rank decomposition of $M|_{X\times Y}$.

\subsection{Experiment 1: planar points and the height function}
\label{sec:experiment1}
\begin{itemize}
\item The point set $P$ consists of 150 planar points as shown in Figure~\ref{fig:leftright-pc} (left).
\item The function $f$ is the height function, i.e., $f(p_x, p_y) = p_y$.
\item The homology degree is~1 and $\field=\Z_2$. 
\item The discretizations are given by restriction to the grids $G_1=\{r_0, \ldots, r_{49}\}\times \{s_0, \ldots, s_{49}\}$ and $G_2 = \{r_4, r_9, \ldots, r_{49}\}\times \{s_4, s_9, \ldots, s_{49}\}$. The values $r_i$ are chosen such that the difference $r_{i+1}-r_i$ is constant,  while the values $s_i$ are selected such that each interval $[s_i, s_{i+1})$ contains the same number of function values. 
\end{itemize}
The point set exhibits three significant holes of varying sizes appearing at different heights. Moving from bottom and up we denote these holes by A, B and C, respectively. The evolution of these holes in the bifiltration is shown in  Figure~\ref{fig:leftright-pc} (right), and the associated signed barcodes of $M|_{G_1}$ and $M|_{G_2}$ are shown in Figure~\ref{fig:leftright-barcode}. 

\begin{figure}[tb]
\centering
\begin{subfigure}{0.3\textwidth}
  \centering
  \includegraphics[width=1\linewidth]{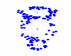}

\end{subfigure}%
\begin{subfigure}{.7\textwidth}
  \centering
  \includegraphics[width=1\linewidth]{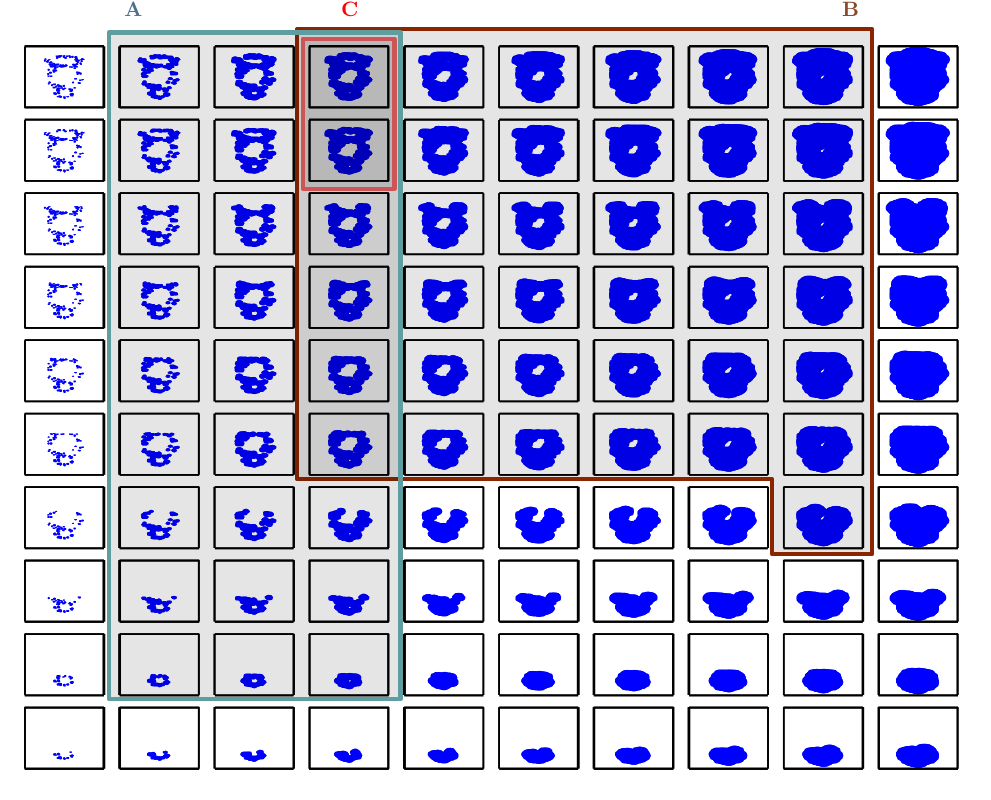}
\end{subfigure}
\caption{Left: The point set in Experiment 1. Right: The associated bifiltration of the distance and height functions. The lifespan of each 1-dimensional feature A, B, C in the bifiltration is highlighted in a specific color.}
\label{fig:leftright-pc}
\end{figure}

\begin{figure}
\centering
\begin{subfigure}{0.5\textwidth}
  \centering
  \includegraphics[width=1\linewidth]{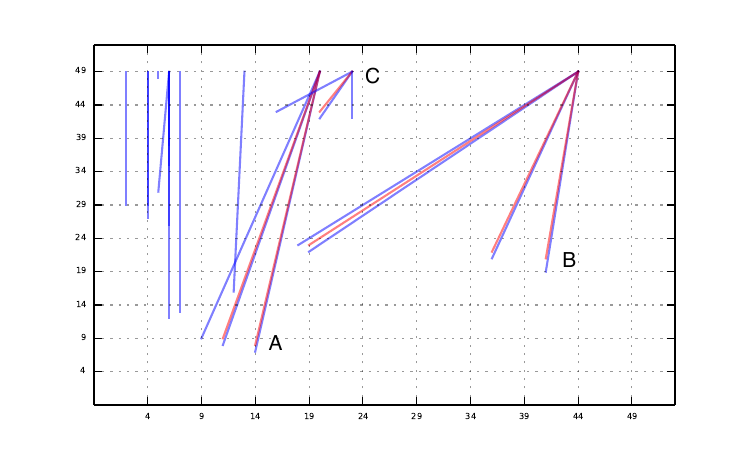}
  \caption{}
  \label{fig:leftright-barcode-1}
\end{subfigure}%
\begin{subfigure}{.5\textwidth}
  \centering
  \includegraphics[width=1\linewidth]{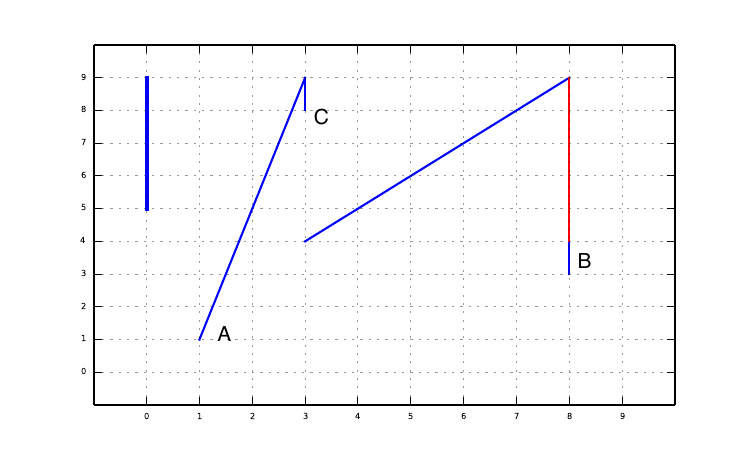}
  \caption{}
  \label{fig:leftright-barcode-2}
\end{subfigure}
\caption{The signed barcode of Experiment 1 for two different choices of grid-size: $50\times 50$ (A) and $10\times 10$ (B).}
\label{fig:leftright-barcode}
\end{figure}

\begin{figure}
\centering
\begin{subfigure}{0.5\textwidth}
  \centering
\includegraphics[width=\linewidth]{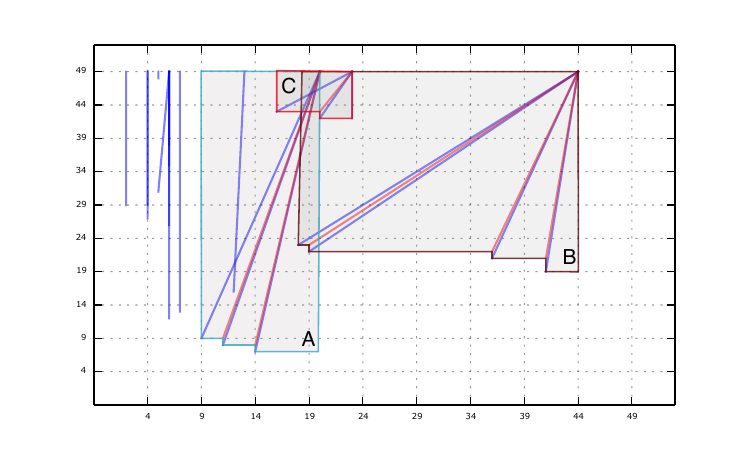}
\caption{}
\label{fig:leftright-intervals-1}
\end{subfigure}%
\begin{subfigure}{.5\textwidth}
  \centering
\includegraphics[width=\linewidth]{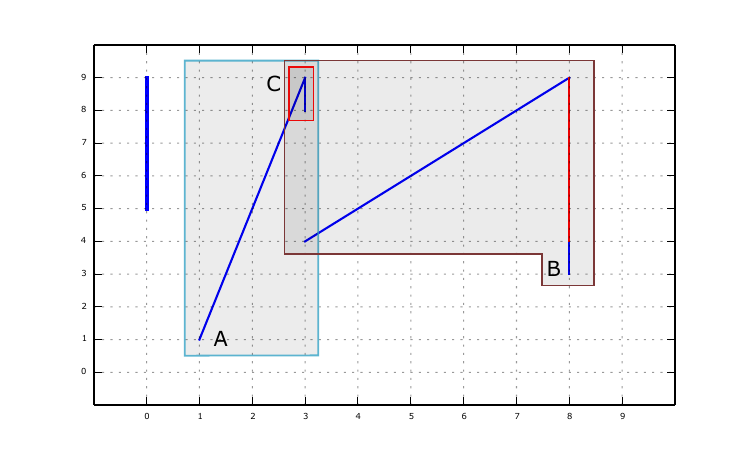}
\caption{}
\label{fig:leftright-intervals-2}
\end{subfigure}
\caption{The supports of the three features A, B and C in Experiment 1 for our two choices of grid-size: $50\times 50$ (left) and $10\times 10$ (right). The supports in the $10\times 10$ grid coincide with the supports shown in Figure~\ref{fig:leftright-pc}. They have been slightly inflated here for better readability.}
\label{fig:leftright-intervals}
\end{figure}

Let us first consider the coarser grid, whose signed barcode is shown in Figure~\ref{fig:leftright-barcode-2}.  To the left in the figure there are a several features (indicated by a line of increased thickness) persisting only in the vertical direction. These features correspond to tiny loops formed by ``noise'' in the sampling, as can be seen in the first column of Figure~\ref{fig:leftright-pc}~(right). The holes A and C are generated at a unique index and therefore each of them gives rise to a single bar in the signed barcode. On the other hand, hole B appears at two incomparable indices and therefore gives rise to two rectangles in $\Rec|_{G_2}$ born at incomparable grades, and a rectangle in $\Sec|_{G_2}$ accounting for double counting. The pattern of the bars suggests a generalized rank decomposition as shown in Figure~\ref{fig:leftright-intervals-2}. We see that these intervals correspond precisely to the supports shown in Figure~\ref{fig:leftright-pc}. The existence of a ``single'' feature could be confirmed by computing generalized rank invariants.

Shifting our focus to the finer grid, the first thing we observe is that the choice of a finer discretization increases the number of bars in the signed barcode. This is not surprising as the same feature will now appear at an additional number of incomparable points. Again, the collection of bars strongly suggests the features persisting over certain intervals as seen in Figure~\ref{fig:leftright-intervals-1}.

\subsection{Experiment 2: noisy circles and stability}
\label{sec:experiment2}
\begin{itemize}
\item Six point sets $P_0, \ldots, P_5$, where $P_0$ consists of 80  points that are evenly placed along the unit circle, and $P_i = \{ p + i\,\e_p : p\in P\}$  where $\e_p$ is a fixed random vector with coordinates  sampled independently and uniformly from $(-0.1, 0.1)$.  See Figure~\ref{fig:circles}.
\item The function $f$ is the height function, i.e., $f(p_x, p_y) = p_y$.
\item The homology degree is~1 and $\field=\Z_2$.
\item The discretization is obtained by restricting to \[G_1 = \{r_0, \ldots, r_{39}\}\times \{s_0, \ldots, s_{39}\},\] where the values are chosen such that the differences $r_{i+1}-r_i$ and $s_{j+1}-s_j$ are constant. For the purpose of visualization we shall use the coarser grid $G_2 = \{r_3, r_7, \ldots, r_{39}\}\times \{s_3, s_7, \ldots, s_{39}\}$ --- see Figure~\ref{fig:circle-bifil}.
\end{itemize}
The purpose of this experiment is evidently to observe the evolution of the signed barcodes as the points are moving linearly from $\{p\}$ to $\{p+5\,\e_p\}$. The signed barcodes are shown in (A) through (F) of Figure~\ref{fig:bc_noisy}. We observe the following:
\begin{itemize}
\item ($P_0$) As expected: at lower heights, the feature appears at larger distance thresholds than at higher heights. 
\item ($P_1$) With the introduction of noise in the data, the feature persists longer at lower heights, and this causes the addition of two near-horizontal generators around height index 30, as well as a vertical generator around scale index 33. To account for double-counting of the rank, two near-horizontal co-generators and one vertical co-generator appear as well. Note that while the second to last column of  Figure~\ref{fig:circle-bifil} suggests the presence of a non-trivial $H_1$, that is not the case when working with the Vietoris--Rips complex.
\item ($P_2-P_5$) As the perturbation increases the support of the signed barcode corresponding to the circular signal in the data gradually shrinks: it appears at a later scale and persists for a shorter amount of time. Furthermore, the noise has \revisedversion{introduced} several short-lived cycles in the data, giving rise to an increasing amount of near-vertical lines to the left in the plot. 
\end{itemize}

\begin{figure}
\centering
\begin{subfigure}{0.15\textwidth}
  \centering
  \includegraphics[width=\linewidth]{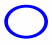}
  \caption{$P_0$}
  \end{subfigure}%
\begin{subfigure}{.15\textwidth}
    \includegraphics[width=\linewidth]{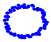}
      \caption{$P_1$}
\end{subfigure}%
\begin{subfigure}{.15\textwidth}
    \includegraphics[width=\linewidth]{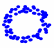}
          \caption{$P_2$}
\end{subfigure}%
\begin{subfigure}{.15\textwidth}
    \includegraphics[width=\linewidth]{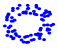}
          \caption{$P_3$}
\end{subfigure}%
\begin{subfigure}{.15\textwidth}
    \includegraphics[width=\linewidth]{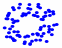}
          \caption{$P_4$}
\end{subfigure}%
\begin{subfigure}{.15\textwidth}
    \includegraphics[width=\linewidth]{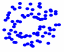}
          \caption{$P_5$}
\end{subfigure}%
\caption{The six point sets from Section~\ref{sec:experiment2}.}
\label{fig:circles}
\end{figure}
\begin{figure}
  \centering
  \includegraphics[width=0.7\linewidth]{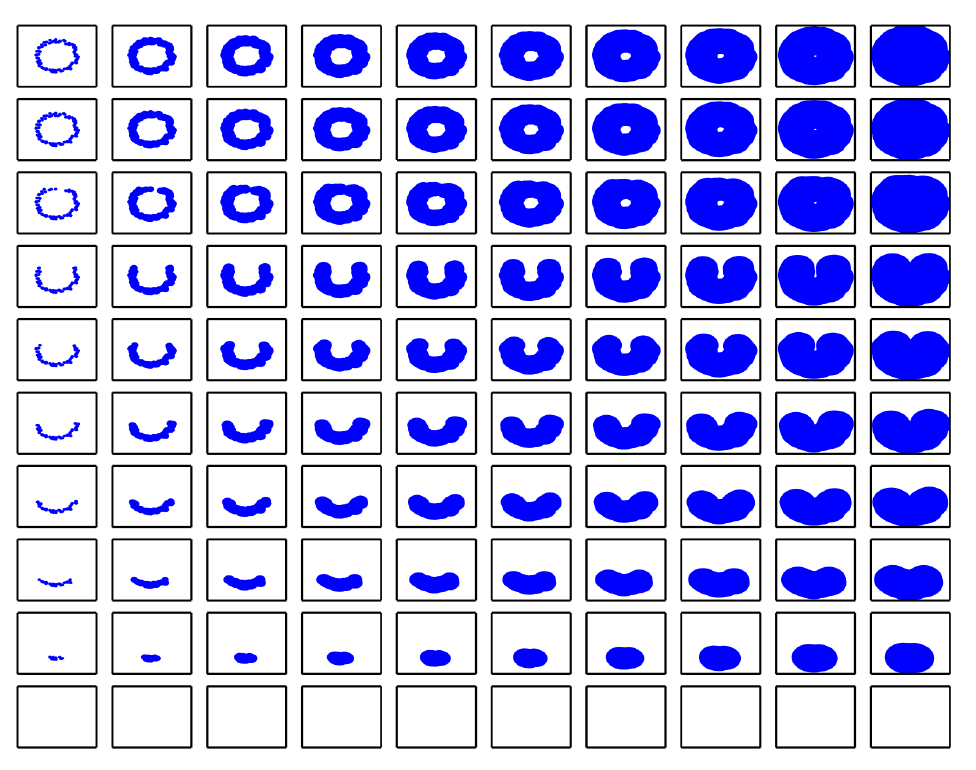}
        \caption{A bifiltration of $P_1$ over $G_2$.}
		\label{fig:circle-bifil}
\end{figure}

\begin{figure}
\centering
\begin{subfigure}{0.5\textwidth}
  \centering
  \includegraphics[width=\linewidth]{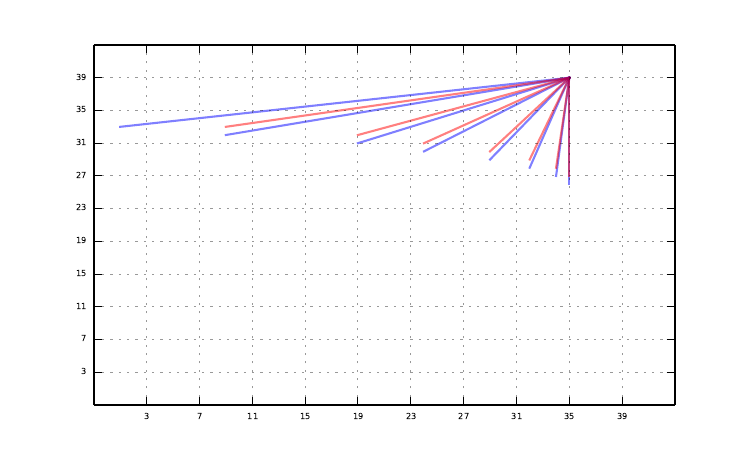}
  \caption{$P_0$}
  \end{subfigure}%
\begin{subfigure}{.5\textwidth}
    \includegraphics[width=\linewidth]{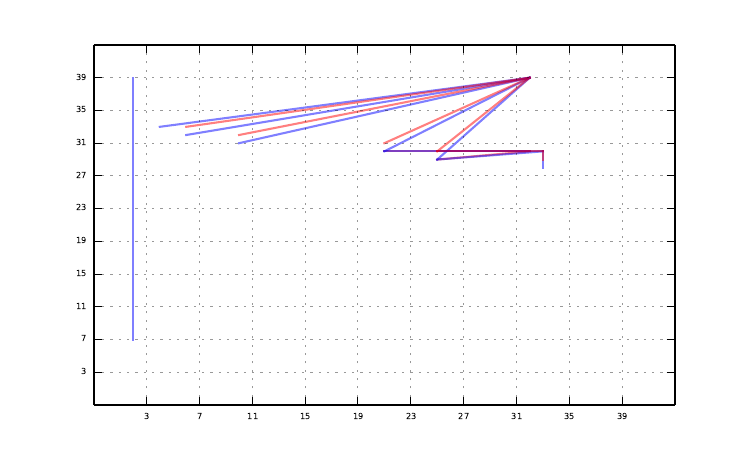}
      \caption{$P_1$}
\end{subfigure}%

\begin{subfigure}{.5\textwidth}
    \includegraphics[width=\linewidth]{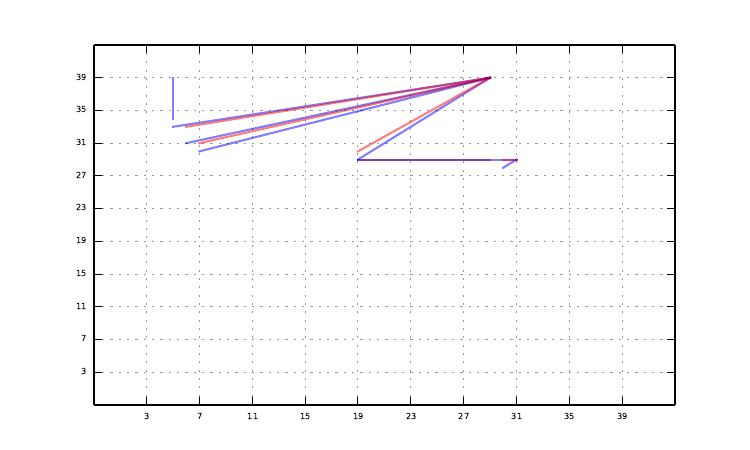}
          \caption{$P_2$}
\end{subfigure}%
\begin{subfigure}{.5\textwidth}
    \includegraphics[width=\linewidth]{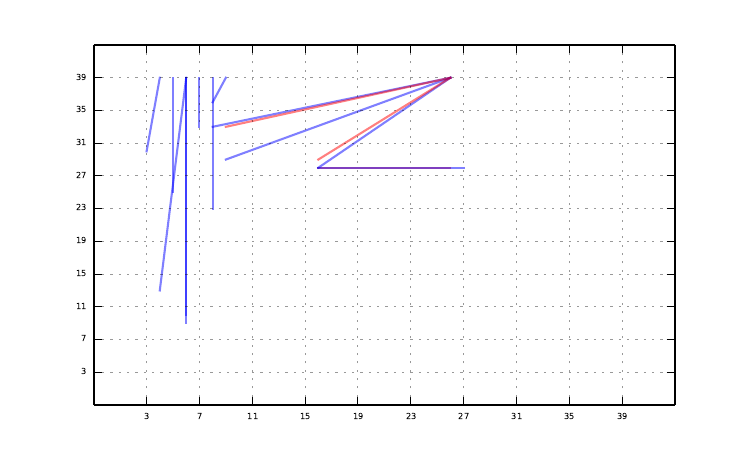}
          \caption{$P_3$}
\end{subfigure}%

\begin{subfigure}{.5\textwidth}
    \includegraphics[width=\linewidth]{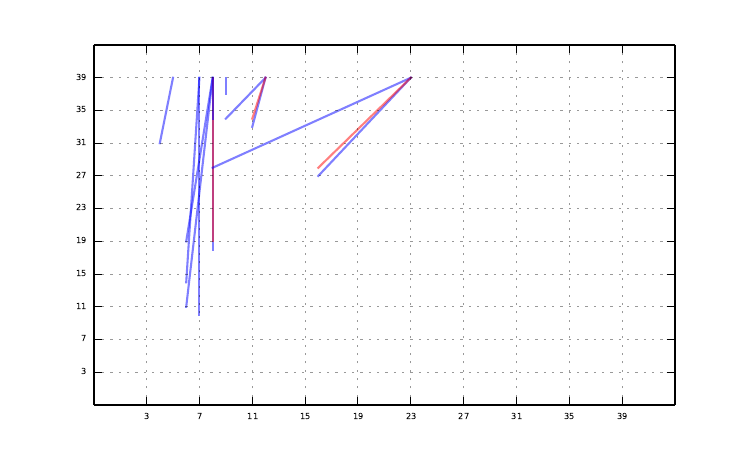}
          \caption{$P_4$}
\end{subfigure}%
\begin{subfigure}{.5\textwidth}
    \includegraphics[width=\linewidth]{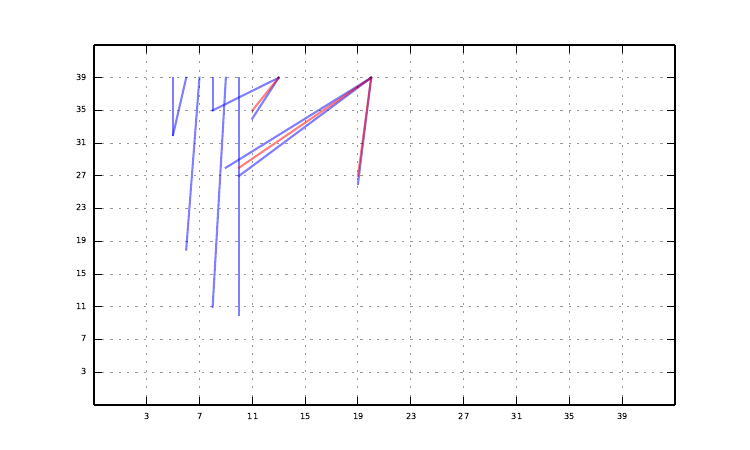}
          \caption{$P_5$}
\end{subfigure}%
\caption{The signed barcodes of Experiment 2 over the grid $G_1$ ($40\times 40)$. }
\label{fig:bc_noisy}
\end{figure}

\subsection{Experiment 3: two-parameter clustering}
\label{sec:experiment3}
\begin{itemize}
\item The point set $P$ consists of the $N=90$ planar points sampled from three Gaussian distributions as shown in Figure~\ref{fig:nonthin-bifil}.
\item The function $f$ is a local co-density estimate, i.e., 
\[f(p) = \#\{q \in P : d(p,q) > \epsilon\},\qquad \qquad \text{for a fixed $\epsilon\geq 0$}.\]
\item The homology degree is~0 and $\field=\Z_2$.
\item The discretization is obtained by restriction to $G = \{r_0, \ldots, r_{9}\}\times \{s_0, \ldots, s_{9}\}$ --- see Figure~\ref{fig:nonthin-bifil}.
\end{itemize}

The resulting persistence module is not interval-decomposable. Geometrically, this is due to the fact that the three clusters A, B, C merge in three different ways at incomparable grades, as shown in the highlighted squares of Figure~\ref{fig:nonthin-bifil}.  Hence, one obtains the following diagram of vector spaces and linear maps: 
\[ \begin{gathered}
  \xymatrix{
  \field^2 \\
  & \field^2 \\
  \field^3 \ar_(.7){\left[ \begin{smallmatrix} 1& 0 & 0 \\ 0 & 1 & 1\end{smallmatrix} \right]}[uu]
  \ar_-{\left[\begin{smallmatrix} 1 & 1 & 0 \\ 0 & 0 & 1\end{smallmatrix}\right]}[ur]
  \ar_-{\left[\begin{smallmatrix} 1 & 0 & 1 \\ 0 & 1 & 0 \end{smallmatrix}\right]}[rr]
  &&  \field^2
}
  \xymatrix{\\\cong\ \\}
\xymatrix{
  \field \\
  & \field \\
  \field^2 \ar^-{\left[ \begin{smallmatrix} 1& 1\end{smallmatrix} \right]}[uu]
  \ar_-{\left[\begin{smallmatrix} 1 & 0\end{smallmatrix}\right]}[ur]
  \ar_-{\left[\begin{smallmatrix} 0 & 1 \end{smallmatrix}\right]}[rr]
 &&  \field  }
  \xymatrix{\\\bigoplus\ \\}
\xymatrix{
  \field \\
  & \field \\
  \field \ar^-{1}[uu]
  \ar_-{1}[ur]
  \ar_-{1}[rr]
  &&  \field  }
\end{gathered}
\]

The obtained signed barcode and prominence diagram are shown in Figure~\ref{fig:bc_nonthin}. As expected, the lifespans of the three clusters A, B, C appear as three separate subsets of the bars, as shown in Figure~\ref{fig:bc_nonthin_lifespans}. Moreover, these three subsets of bars can be discriminated from the rest of the barcode by their prominence, as seen from the prominence diagram. Checking whether any one of these three subsets of bars does correspond to the lifespan of some feature can then be done by computing the coefficient assigned to  the corresponding interval in the minimal generalized rank decomposition of~$\Mod$.

\begin{figure}
  \centering
  \includegraphics[width=0.8\linewidth]{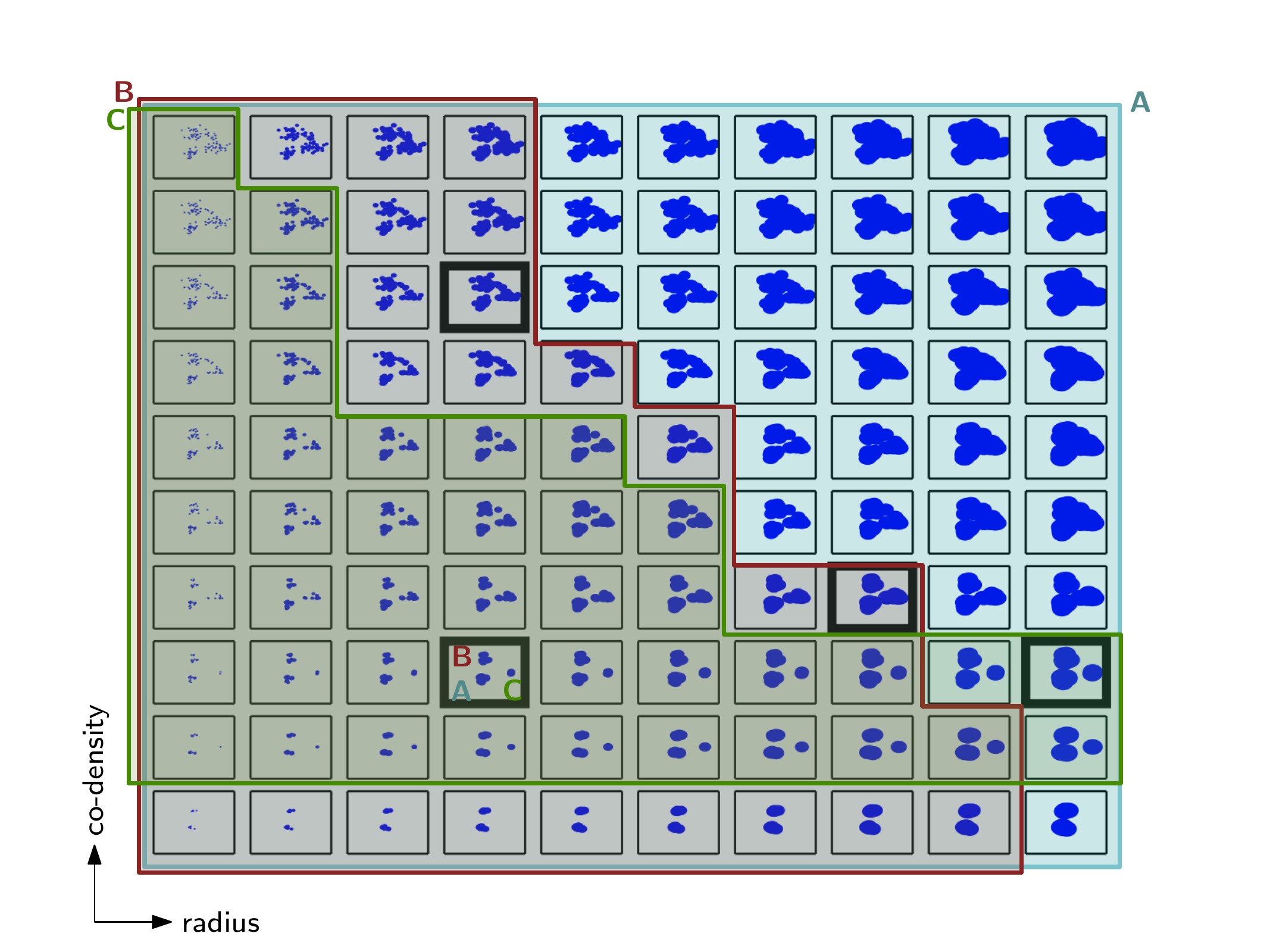}
        \caption{The bifiltration in Experiment 3. The highlighted squares show that three clusters (named A, B, C) merge in three different ways at incomparable scales. The lifespan of each one of these three clusters is marked by an interval with matching color.}
		\label{fig:nonthin-bifil}
\end{figure}

\begin{figure}
\includegraphics[width=0.6\textwidth]{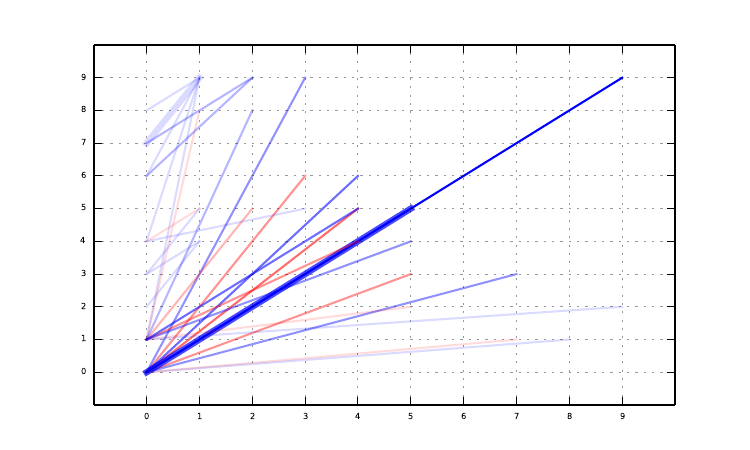}
\hspace{0.02\textwidth}
\includegraphics[width=0.35\textwidth]{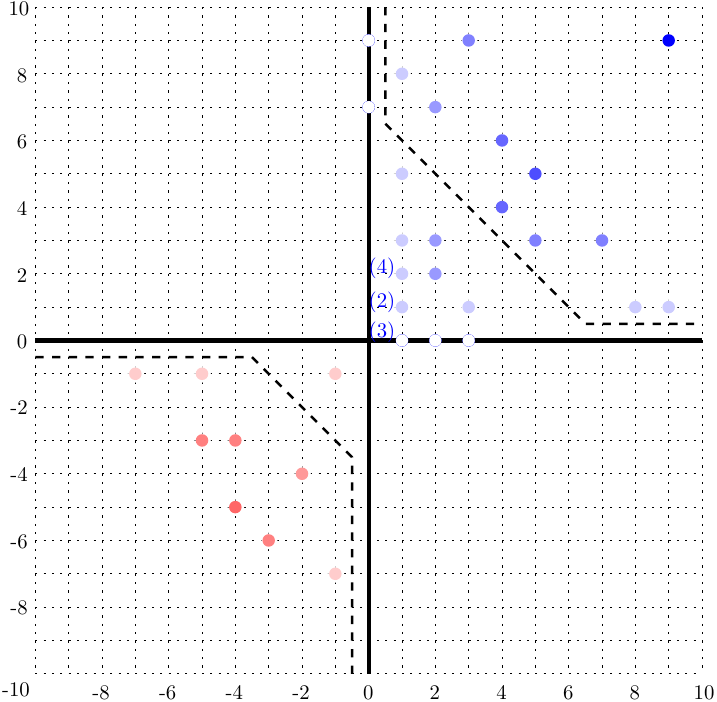}
  \caption{Left: signed barcode of Experiment~3 over the $10\times 10$ grid.  Thicker bars overlap with another bar. Right: corresponding prominence diagram, where the bars coming from the lifespans of~$A,B,C$ are separated from the rest of the bars by the dashed curves. Each bar with endpoints $s\leq t$ in the barcode (and diagram) has an intensity   proportional to $\min\{t_x-s_x,\,t_y-s_y\}$; in particular, horizontal and vertical bars are invisible.}
  \label{fig:bc_nonthin}
\end{figure}%

\begin{figure}
 \includegraphics[width=0.49\textwidth]{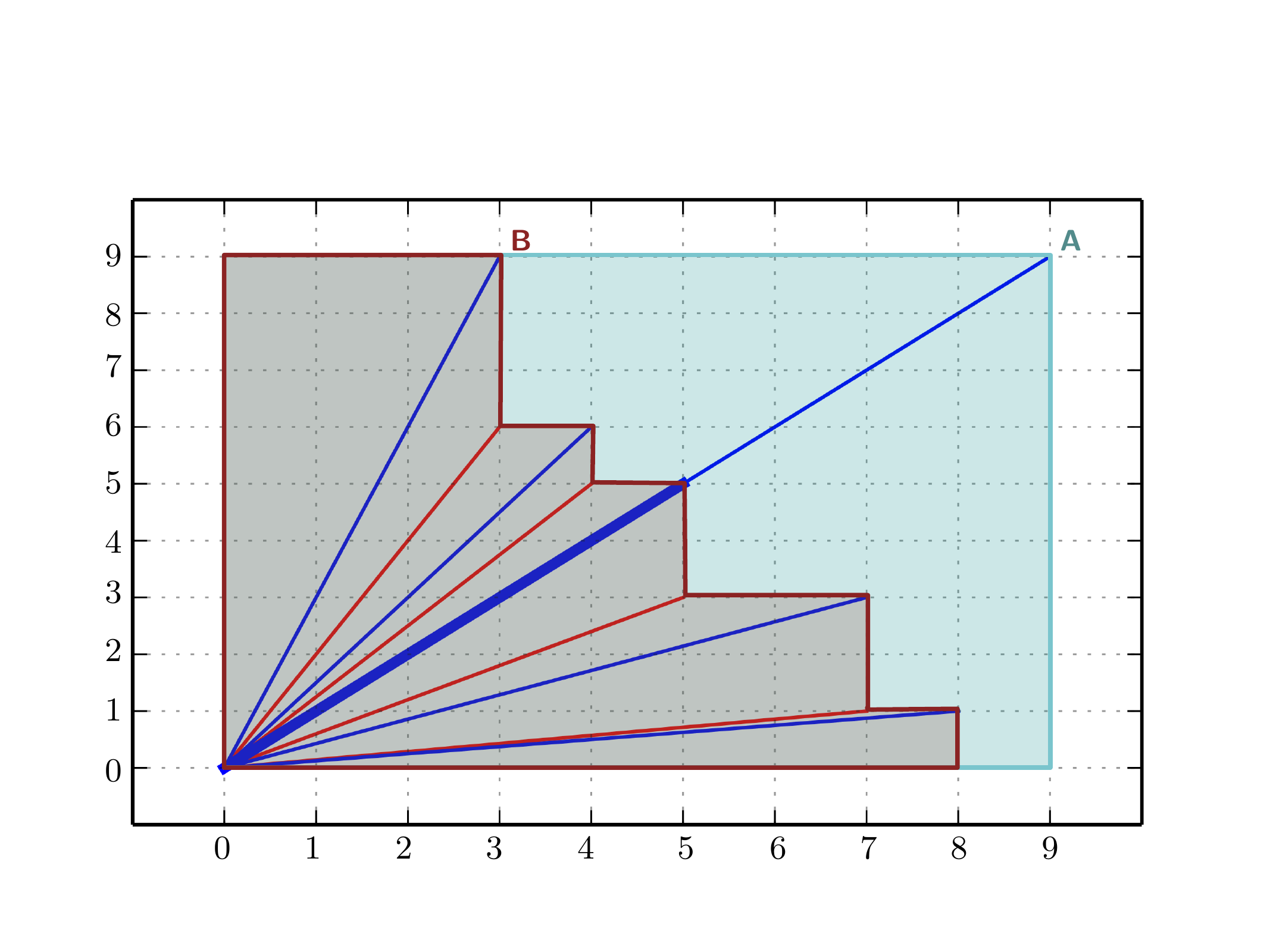}
 \includegraphics[width=0.49\linewidth]{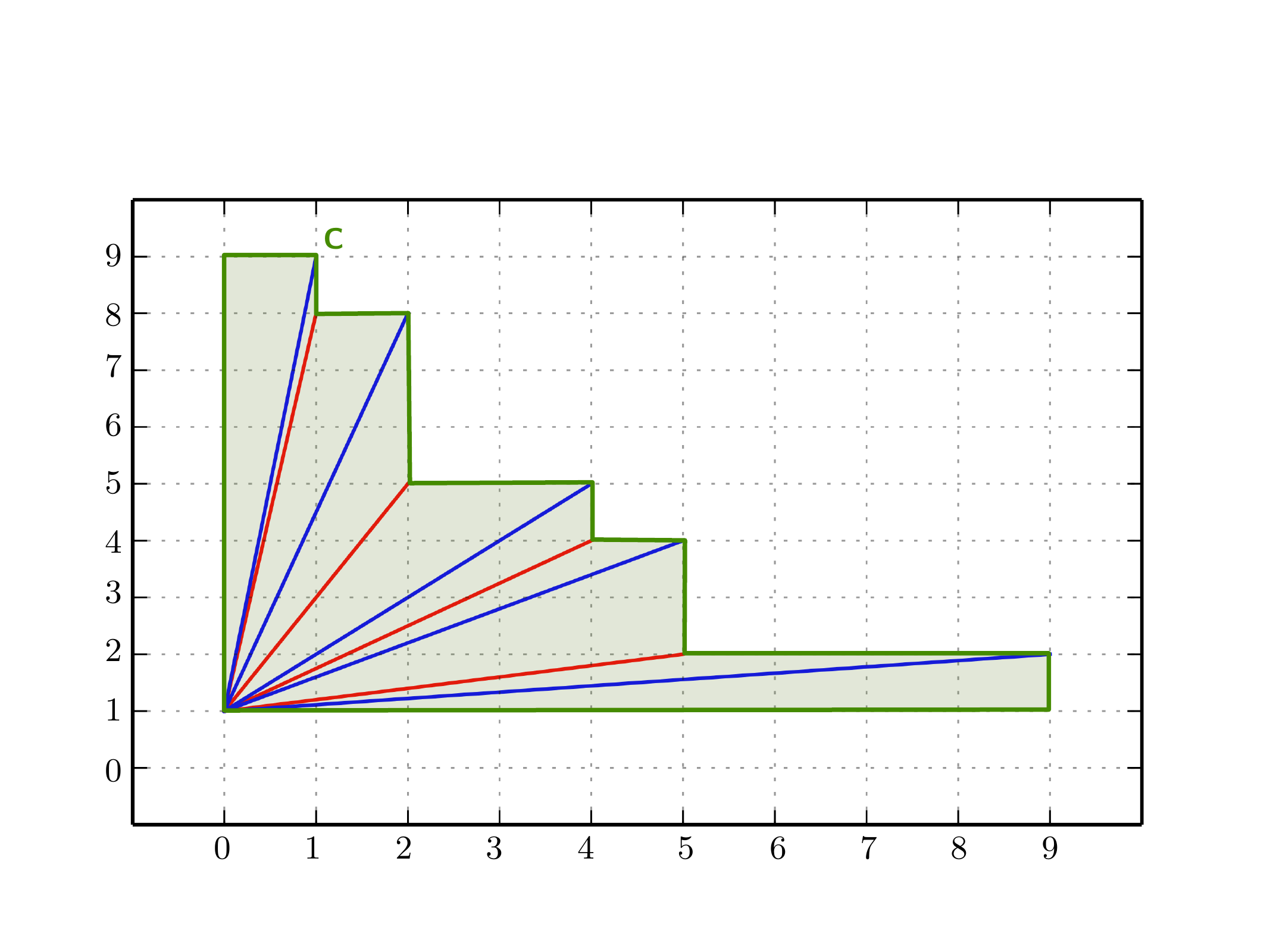}
\caption{Lifespans of $A,B$ (left) and $C$ (right) in the signed barcode.}
\label{fig:bc_nonthin_lifespans}
\end{figure}

\section{Discussion}
\label{sec:conclusion}

We conclude the paper by further discussing some aspects of our work and the prospects they raise. 

\subsection{Generalized rank-exact resolutions} 
\label{sec:genrank-res}

Considering short exact sequences on which rank is additive is more subtle when we consider generalized ranks. The following example shows that in general we cannot expect to obtain an exact structure \revisedversion{(Section~\ref{sec:exact_struct})} in this way.

\begin{example}
Consider the \( 2 \times 2 \)-grid, and the interval \( \int \) indicated below.
\[ \begin{tikzcd}[
 execute at end picture={
  \fill[opacity=.2] (a.north west) -- (a.south west) -- (b.south west) -- (c.south west) -- (c.south east) -- (b.north east) -- cycle;
 }]
|[alias=a]| \circ \ar[r] & |[alias=b]| \circ \\
\circ \ar[u] \ar[r] & |[alias=c]| \circ \ar[u]
\end{tikzcd} \]
For this grid all indecomposables are thin, and we will denote them by their dimension vectors. Consider the following commutative diagram with exact rows.
\[ \begin{tikzcd}
0 \ar[r] & \begin{smallmatrix} 0 \, 1 \\ 0 \, 0 \end{smallmatrix} \ar[r] \ar[d,equal] & \begin{smallmatrix} 1 \, 1 \\ 0 \, 0 \end{smallmatrix} \oplus \begin{smallmatrix} 0 \, 1 \\ 0 \, 1 \end{smallmatrix} \ar[r] \ar[d] & \begin{smallmatrix} 1 \, 1 \\ 0 \, 1 \end{smallmatrix} \ar[r] \ar[d] & 0  \\
0 \ar[r] & \begin{smallmatrix} 0 \, 1 \\ 0 \, 0 \end{smallmatrix} \ar[r] & \begin{smallmatrix} 1 \, 1 \\ 0 \, 0 \end{smallmatrix} \ar[r] & \begin{smallmatrix} 1 \, 0 \\ 0 \, 0 \end{smallmatrix} \ar[r] & 0  
\end{tikzcd} \]
Here the middle vertical arrow is the projection to the first summand. Note that \( \Rk_{\int} \) is zero on all objects in this diagram, except the right upper corner, where it is one. In particular \( \Rk_\int \) is additive on the lower sequence, but not on the upper one. But the upper sequence is a pull-back of the lower one. This shows that the collection of short exact sequences such that \( \Rk_{\int} \) is additive does not constitute an exact structure.
\end{example}

\subsection{About the collection~$\Int$ of intervals involved in the rank decompositions}

\revisedversion{Definition~\ref{def:rank_decomp} forces} the collection~$\Int$ of intervals \revisedversion{used for} the \revisedversion{decomposition of the} generalized rank invariant to be the same as the collection \revisedversion{over which it is evaluated}. Thus, $\Int$ plays a double role:
\begin{itemize}
\item that of a {\em dictionary} of shapes (formally, a basis of rank functions) over which the rank decompositions are built;
\item that of a {\em test set} of shapes over which the (generalized) rank invariant is evaluated, i.e., the existence of ``features'' is probed.
\end{itemize}
It is natural to try decorrelating the two roles by assigning a different collection of intervals to each one of them, say $\Intd$ for the dictionary and  $\Intt$ for the test set. The \revisedversion{question} becomes then to decompose (generalized) rank invariants $\Rk_{\Intt} \Mod$, or more generally, to decompose maps $r:\Intt\to\Z$, over the family of rank invariants $\left(\Rk_{\Intt} \field_\rec\right)_{\rec\in\Intd}$. \revisedversion{Problems may arise with the existence or uniqueness of the decompositions though, and generally speaking:
\begin{itemize}
  \item letting $\Intt\supsetneq\Intd$ may prevent rank decompositions from existing at all, and when they exist, it implies that the minimal one coincides with the one obtained with $\Intd=\Intt$, thus bringing no further insight;
  \item letting $\Intt\subsetneq\Intd$ may yield smaller-sized rank decompositions, such as for instance in Figures~\ref{fig:decomp_indec2_int_grid} and~\ref{fig:decomp_indec_int}, however it may also prevent the uniqueness of the minimal rank decomposition, thus hurting its interpretability as illustrated in Figure~\ref{fig:decomp_2-2_inv}.   
\end{itemize}
Yet, as we have seen on several occasions, it can sometimes be useful to take $\Intd\neq\Intt$. For instance, taking for $\Intd$ the half-open rectangles in $\R^d$ and for $\Intt$ the closed rectangles allowed us to decompose the usual rank invariants of finitely presented modules in Section~\ref{sec:rank-decomp_pers}; or, taking for $\Intd$ the hooks and for $\Intt$ the closed segments allowed us to work with rank-exact resolutions and thus lift the rank decompositions to the corresponding Grothendieck group in Section~\ref{sec:rank-exact}. Beyond these special cases, determining when similar winning combinations of $\Intd\neq \Intt$ occur remains wide open.
}

\begin{figure}[tb]
  \centering
  \includegraphics[width=\textwidth]{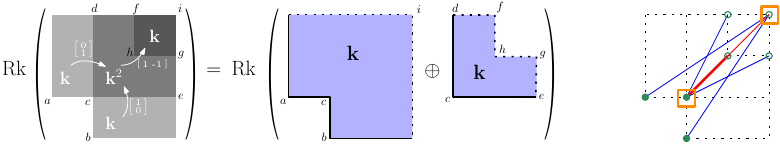}
  \caption{Left and center: one of several possible minimal rank decompositions, over the collection of all intervals of $\R^2$,  of the usual rank invariant of the finitely presented module~$\Mod$ from Figure~\ref{fig:decomp_indec}. Right: \revisedversion{the corresponding signed barcode, with orange squares indicating} how the bars are clustered into two groups coming from the two elements in the decomposition  at the center (note: the bottom-left orange square does not include the longer red bar).} 
  \label{fig:decomp_indec_int}
\end{figure}

\begin{figure}[tb]
  \centering
  \includegraphics[width=\textwidth]{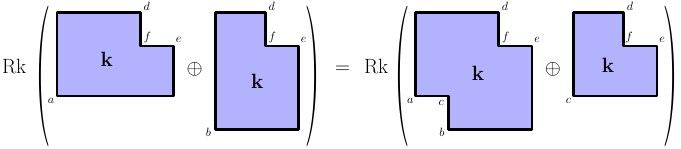}
  \caption{Two interval-decomposable modules with the same usual rank invariant. The module on the right-hand side is one among several possible minimal usual rank decompositions of the module on the left-hand side over the full collection of intervals of~$\R^2$. Notice how this decomposition creates the illusion that a feature spans the entire interval with perimeter $(a,c,b,e,f,d)$, whereas no such feature actually exists in the module on the left-hand side.} 
  \label{fig:decomp_2-2_inv}
\end{figure}

\subsection{Limitations of rank decompositions} As rank decompositions encode the structure of rank invariants, they are just as powerful descriptors as the rank invariants themselves. In particular, they are not  complete descriptors for multi-parameter persistence modules nor, more generally, for poset representations. For instance, we see from Figure~\ref{fig:decomp_indec2_fullint_grid} that the direct sum of the indecomposable module on the left-hand side of the figure with the red  module on the right-hand side has the same generalized ranks as the blue  module in that same figure. As a consequence, they are indistinguishable from each other based solely on their generalized rank invariants. An important question that comes up is how much of the structure of a persistence module may be missed by these descriptors.

If we restrict our focus to interval-decomposable representations over finite posets, then we know from Proposition~\ref{prop:complete} that the generalized rank invariant (with $\Int$ being the full collection of intervals in the poset) is a complete descriptor, therefore so is its minimal rank decomposition. Of course, the usual rank invariant itself is not a complete descriptor, as illustrated in Figure~\ref{fig:decomp_2-2_inv}. \revisedversion{It is, however, complete on the subcategory of segment-decomposable representations, by Proposition~\ref{prop:complete}.}
An important question in practice is how to choose the collection $\Int$ of intervals, between the collection of closed segments (which corresponds to taking the usual rank invariant) and the full collection of intervals in the poset, to find the right trade-off between the richness of the invariant and the complexity of its calculation.

\subsection{Signed barcodes versus RIVET}

As explained in Section~\ref{sec:signed_barcodes}, signed barcodes play a role similar to that of classical barcodes in one-parameter persistence: they carry information about the global structure of the rank invariant, and they allow for certain visually-enabled procedures to retrieve the value of the rank invariant at indices $s\leq t$, or the barcodes of the restriction of the module  to a given upward sloping line. Thus, they appear as a viable alternative to
RIVET~\cite{lesnick2015interactive} for the exploration of
multi-parameter persistence modules.
They can also be used in conjunction with RIVET, or parts thereof,
for instance: using RIVET to compute minimal presentations, then using the
dynamic programming procedure from~\cite{botnan_et_al:LIPIcs:2020:12180} to compute the usual rank
invariant, and finally using~\eqref{eq:incl_excl_rect} to compute the corresponding minimal
rank decomposition and signed barcode.

Unlike RIVET, signed barcodes are
readily available for persistence modules with an arbitrary number of
parameters. They are easy to compute once the usual rank invariant has been computed, and in the two-parameter \revisedversion{setting} the total \revisedversion{worst-case} running time compares favorably to that of RIVET.
They are also
easy to encode, and their storage size never exceeds that of the usual
rank invariant (since there are no more distinct rectangles
$\seg{s,t}$ than there are pairs of comparable indices $s\leq t$ in
the input filtration). On the downside, for now there still remains work to be done---out
of the scope of this paper---to design a corresponding data structure
that allows for efficient (logarithmic-time) queries as in RIVET.

\subsection{Statistics and machine learning with signed barcodes}

While we have been mainly discussing the interpretability of our rank decompositions via their associated signed barcodes, motivated by the exploration of multi-parameter persistence modules, we should recall that, in the one-parameter setting, barcodes are also used to extract features from data, with applications in machine learning and artificial intelligence. To this end, a number of vectorizations and kernels for (unsigned) barcodes have been proposed---see e.g.~\cite{adams2017persistence,bubenik2015statistical,carriere2017sliced,kusano2016persistence,10.5555/3327546.3327666,reininghaus2015stable}, the vast majority of which enjoy stability properties in terms of the bottleneck distance between barcodes. Here, considering the nature of the signed barcodes, and the (pseudo-)distance put on their corresponding rank decompositions in Section~\ref{sec:stability}, it is possible that a subset of these techniques adapts naturally, with similar stability guarantees following from Theorem~\ref{th:matching-distance_stability}.
%
It will then be interesting to compare the outcome, both in terms of performances in machine learning applications and in terms of computation time, to \revisedversion{existing} vectorizations or kernels for multi-parameter persistence modules~\cite{NEURIPS2020_fdff71fc,corbet2019kernel,vipond2020multiparameter}.  The key point is that these previous techniques essentially slice the input module by  arbitrarily chosen collections of lines, and aggregate the resulting fibered barcodes---either by concatenation, by integration, or through vineyards; the added value of the signed barcode is to provide insight into which lines are most relevant for the slicing, potentially allowing for faster and more accurate computations, producing both richer and smaller-sized features. 

The metric space of rank decompositions may also be in itself a relevant object of study. An obvious question is whether the matching distance between rank decompositions can be translated into an optimal transportation distance between signed barcodes, possibly up to some constants. If the answer is positive, then, similarly to the space of unsigned barcodes in one-parameter persistence, the space of signed barcodes in multi-parameter persistence will be comparable to a space of (signed) measures, and the techniques developed in the one-parameter setting for doing statistics with unsigned barcodes may be adapted to work with signed barcodes as well. Meanwhile, the computation of the metric between rank decompositions will be greatly facilitated by its translation into a combinatorial matching problem---currently the signed barcodes provide insight into the structure of each individual rank decomposition and its corresponding rank invariant, however it is still unknown how to exploit this knowledge to find the most relevant directions for the matching distance between rank decompositions.

\bibliographystyle{abbrv}
\bibliography{./biblio}

\appendix
\section{Möbius inversions}
\label{sec:localfinite-mobius}
\revisedversion{Möbius inversions have played an integral role in shaping generalized persistence ~\cite{asashiba2019approximation,kim2018generalized,McCleary_2020,mccleary2020positivity,patel2018generalized}. 
In this section we employ Möbius inversions to prove existence of rank decompositions in the setting of Section~\ref{sec:rk-decomp_downward-finite}. We remark that our proof of Proposition~\ref{prop:alpha_mobius} is an adaptation of the proofs of~\cite{kim2018generalized,asashiba2019approximation} to a more general class of intervals. From this result one easily obtains an algorithm for computing a minimal rank decomposition for finite grids.}

Let \( (\Int, \subseteq) \) be a locally finite poset (Definition~\ref{def:locallyfinite}). Note that we are using subset notation here for the order relation because this is the example we will be interested in hereafter. However, for the results of this subsection \( \Int \) is an abstract poset. The incidence algebra is given as all functions \( \RelationAsSet[\subseteq]{\Int} \to \mathbb{Z} \), with products given by convolution:
\[
(f \convol g)(\jnt, \int) = \sum_{\jnt\supseteq\knt\supseteq\int} f(\jnt, \knt)\, g(\knt,\int).
\]
Note that the multiplicative unit is  \( \one_{\jnt=\int} \), that is, the function sending identical pairs to \( 1 \) and all other pairs to \( 0 \).

A natural element to consider is the function \( \zeta \), sending all comparable pairs to \( 1 \), and all other pairs to 0. It is shown in \cite[Proposition~3.1]{rota1964foundations} that \( \zeta \) is invertible, and that its inverse, \( \mu \), is given by (any of) the following explicit recursions:
\[
\mu(\jnt, \int) \quad =  \quad \begin{cases} 1 & \text{if } \jnt = \int \\ - \sum_{\jnt\supseteq\knt\supsetneq\int} \mu(\jnt, \knt) & \text{otherwise.} \end{cases} \quad = \quad \begin{cases} 1 & \text{if } \jnt = \int \\ - \sum_{\jnt\supsetneq\knt\supseteq\int} \mu(\knt, \int) & \text{otherwise.} \end{cases} 
\]
Note that $\mu$ is well-defined because \( \Int \) is assumed to be locally finite.

\begin{proposition} \label{prop explicit mobius inverse}
Suppose that, for any \( \int \in \Int \), there is a finite set \( \int^+ \subseteq \Int \) with the property that
\[ \{ \jnt \in \Int \mid \jnt \supsetneq \int \} = \{ \jnt \in \Int \mid \exists \knt \in \int^+ \colon \jnt \supseteq \knt \}, \]
and that any subset of \( \int^+ \) has a join in \( \Int \).
Then:
\[ \mu( \jnt, \int ) = \begin{cases} 1 & \text{ if } \jnt = \int \\ \sum_{\substack{x \subseteq \int^+ \\ \vee x = \jnt }} (-1)^{\card x} & \text{ if } \jnt \supsetneq \int \end{cases}. \]
\end{proposition}

This follows from Propositions~1 and 2 in Section~5 of \cite{rota1964foundations}. However, since the notation therein is a bit heavy, we give a direct argument here.

\begin{proof}
We only need to verify that \( \zeta \convol \mu = \one_{\jnt=\int} \) for \( \mu \) being given by the formula of the proposition. Indeed, for \( \jnt \supsetneq \int \) we have\revisedversion{
\begin{align*}
(\zeta \convol \mu)(\jnt, \int) = & \sum_{\jnt \supseteq \knt \supseteq \int} \mu(\jnt, \int) = 1 + \sum_{\jnt \supseteq \knt \supsetneq \int} \sum_{\substack{x \subseteq \int^+ \\ \vee x = \knt }} (-1)^{\card x} \\
= & 1 + \sum_{\varnothing \neq x \subseteq \{ \lnt \in \int^+ \mid \lnt \subseteq \jnt \}} (-1)^{\card x} = \sum_{x \subseteq \{ \lnt \in \int^+ \mid \lnt \subseteq \jnt \}} (-1)^{\card x} = 0 = \one_{\jnt=\int}(J,I)  \qedhere
\end{align*}}
\end{proof}

\begin{remark} \label{rem Mobius one sided}
  Note that Proposition~\ref{prop explicit mobius inverse} does not need our assumption that \( \Int \) is locally finite in any essential way. Without that assumption, we can still define $\mu$ by the formula in the proposition, and the proof still shows that \( \mu \) is a right inverse for \( \zeta \) (with respect to the partially defined convolution product).
\end{remark}

Consider now the collection of all functions \( r \colon \Int \to \mathbb{Z} \) with upward finite support, that is:
\[ \forall \int \in \Int \colon \quad \card \{ \jnt \in \Int \mid \jnt \supseteq \int \text{ and } r(\jnt) \neq 0 \} < \infty. \]
This collection naturally is a right module over the incidence algebra, by the multiplication formula
\[ (r \convol f)(\int) = \sum_{\jnt\supseteq\int} r(\jnt) f(\jnt,\int), \]
and we have the M\"obius inversion formula---see \cite[Proposition~3.2]{rota1964foundations}:
\begin{equation}\label{eq:Mobius formula}
  \forall r,s \colon \Int\to\Z \ \mbox{with upward finite support},\quad
 r = s \convol \zeta \quad \Longleftrightarrow \quad r \convol \mu = s.
\end{equation}

\subsection*{Application to rank decompositions}

In the following, we write $\mult_{\int} \Rec$ for the multiplicity of element~$\int$ in a given multiset $\Rec$.

 \begin{proposition}\label{prop:alpha_mobius}
 Let $\Int$ be a locally finite collection of intervals of~$\pos$, and let $r\colon \Int\to\Z$ have upward finite support.
 A pair $(\Rec, \Sec)$ of pointwise finite multisets of elements of \( \Int \) is a rank decomposition of \( r \) if and only if
 \[ (r \convol \mu)(\int) = \mult_{\int} \Rec - \mult_{\int} \Sec \qquad \forall \int \in \Int. \]
 In particular, the minimal rank decomposition of \( r \) is given as follows, where $\int^n$ means that the multi-set contains $n$ copies of the element~$\int$:
 \[ \Rec = \bigsqcup_{\substack{\int \in \Int \\(r \convol \mu)(\int) > 0}} \int^{(r \convol \mu)(\int)} \text{ and } \Sec = \bigsqcup_{\substack{\int \in \Int  \\(r \convol \mu)(\int) < 0}} \int^{(r \convol \mu)(\int)}. \]
 \end{proposition}

\begin{proof}
By Proposition~\ref{prop:rank_count}, we know that
\[  \Rk_{\int} \field_{\Rec} = \card \{ \rec \in \Rec \mid \rec \supseteq \int \} = \sum_{\rec \supseteq \int} \mult_{\rec} \Rec, \]
that is, \( \Rk_\Int \field_{\Rec} =  (\mult \Rec) \convol \zeta \). The analogous equation holds for \( \Sec \). Thus, by~\eqref{eq:Mobius formula} we have the equivalences:
\begin{align*}
( \Rec, \Sec) \text{ is a rank decomposition of } r \quad \Longleftrightarrow & \quad r = \Rk_\Int \field_{\Rec} - \Rk_\Int \field_{\Sec} \\
\Longleftrightarrow & \quad r = ( \mult \Rec -  \mult \Sec ) \convol \zeta \\
\Longleftrightarrow & \quad r \convol \mu =  \mult \Rec -  \mult \Sec. \qedhere
\end{align*}
\end{proof}

\end{document}